\documentclass[11pt,a4paper,reqno]{amsart}
\usepackage{graphicx} 
\usepackage{verbatim,amssymb,hyperref,mathrsfs,amsfonts,bbm}
\usepackage{dsfont}
\usepackage{ytableau}
\usepackage[all]{xy}
\usepackage{wasysym}
\usepackage[utf8]{inputenc}
\usepackage{BOONDOX-calo}
\usepackage{tikz-cd,amscd,pb-diagram}
\usetikzlibrary{arrows,decorations.markings}
\tikzstyle{dot}=[draw, fill =black, circle, inner sep=0pt, minimum size=2pt]

\numberwithin{equation}{section}

\theoremstyle{plain}
  \newtheorem{thm}{Theorem}[section]
  \newtheorem{cor}[thm]{Corollary}
  \newtheorem{lem}[thm]{Lemma}
  \newtheorem{prop}[thm]{Proposition}
\theoremstyle{definition}
  \newtheorem{Def}[thm]{Definition}
\theoremstyle{remark}
  \newtheorem{rem}[thm]{Remark}
  \newtheorem{eg}[thm]{Example}
  
  \newtheorem*{rem*}{Remark}
\numberwithin{equation}{section}

\newcommand{\bone}{\boldsymbol{1}}
\newcommand{\bbone}{\mathbbm{1}}
\newcommand{\bl}{\boldsymbol{\lambda}}
\newcommand{\bL}{\boldsymbol{\Lambda}}
\newcommand{\bbeta}{\boldsymbol{\beta}}
\newcommand{\balpha}{\boldsymbol{\alpha}}
\newcommand{\bvarnothing}{\boldsymbol{\varnothing}}
\newcommand{\bm}{\boldsymbol{\mu}}
\newcommand{\bnu}{\boldsymbol{\nu}}
\newcommand{\bk}{\boldsymbol{\kappa}}
\newcommand{\UU}{\mathbf{U}}
\newcommand{\uu}{\upsilon}
\newcommand{\fs}{\mathcal{s}}
\newcommand{\bfs}{\pmb{\fs}}
\newcommand{\br}{\mathbf{r}}
\newcommand{\bs}{\mathbf{s}}
\newcommand{\bt}{\mathbf{t}}
\newcommand{\bu}{\mathbf{u}}
\newcommand{\bv}{\mathbf{v}}

\newcommand{\by}{\mathbf{y}}
\newcommand{\bz}{\mathbf{z}}
\newcommand{\wt}{\mathsf{wt}}
\newcommand{\core}{\mathsf{core}}
\newcommand{\quot}{\mathsf{quot}}
\newcommand{\mv}{\mathbf{mv}}
\newcommand{\mvi}{\mathrm{mv}}
\renewcommand{\min}{\mathrm{min}}
\renewcommand{\max}{\mathrm{max}}
\newcommand{\ZZ}{\mathbb{Z}}

\newcommand{\cs}{\mathsf{s}}
\newcommand{\HH}{\mathcal{H}}

\newcommand{\BB}{\mathcal{B}}
\newcommand{\B}{\mathsf{B}}
\newcommand{\C}{\mathsf{C}}
\newcommand{\CC}{\mathbb{C}}
\newcommand{\rr}{\rho}
\newcommand{\FF}{\mathbb{F}}
\newcommand{\CF}{\mathcal{F}}
\newcommand{\Uv}{U_v(\widehat{\mathfrak{sl}}_e)}
\newcommand{\Sp}{S}
\newcommand{\Hub}{\mathsf{Hub}}
\newcommand{\hub}{\mathsf{hub}}
\newcommand{\PP}{\mathcal{P}}
\newcommand{\res}{\mathsf{res}}
\newcommand{\cont}{\mathsf{cont}}

\newcommand{\iInd}{i\text{-Ind}}
\newcommand{\iRes}{i\text{-Res}}

\newcommand{\len}{\mathcal{l}}

\newcommand{\Be}{\mathbf{e}}
\newcommand{\EW}{\widehat{\mathbf{W}}}
\newcommand{\AW}{\mathbf{W}}
\newcommand{\sym}[1]{\mathfrak{S}_{#1}}
\newcommand{\DDot}[1]{\mathrel{\overset{#1}{\bullet}}}
\newcommand{\CDot}[1]{\mathrel{\overset{#1}{\circ}}}
\newcommand{\AAF}{\mathcal{A_\mathsf{F}}}
\newcommand{\AAnobar}{\mathcal{A}_e^{\ell}}
\newcommand{\AAbar}{\overline{\mathcal{A}}_e^{\ell}}
\newcommand{\bij}{\iota}
\newcommand{\Sc}{\mathbf{Sc}}
\newcommand{\EP}{\pmb{\varnothing}}
\newcommand{\Ht}{\mathsf{ht}}

\newcommand{\T}{\mathsf{T}}
\newcommand{\AR}{\mathsf{AR}}
\newcommand{\AS}{\mathsf{A}}
\newcommand{\RS}{\mathsf{R}}
\newcommand{\cy}{\mathfrak{y}}

\def\Z{{\mathbb Z}}

\title[Moving vectors and core blocks]{Moving vectors and core blocks of Ariki-Koike algebras}

\author{Yanbo Li}
\address[Y. Li]{School of Mathematics and Statistics, Northeastern University at Qinhuangdao, Qinhuangdao, 066004, People's Republic of China.}
\email{liyanbo707@163.com}

\author{Xiangyu Qi}
\address[X. Qi]{School of Mathematics and Statistics, Beijing Institute of Technology, Beijing, 100081, People's Republic of China.}
\email{qixiangyumath@163.com}

\author{Kai Meng Tan}
\address[K. M. Tan]{Department of Mathematics, National University of Singapore, Block S17, 10 Lower Kent Ridge Road. Singapore 119760.}
\email{tankm@nus.edu.sg}

\date{January 2025}

\thanks{The third author thanks Joseph Chuang for his useful comments on the affine Weyl group, moving vectors and Rouquier blocks.}

\subjclass[2020]{20C08, 05E10}

\keywords{Moving vectors, core blocks, Scopes equivalence, Kostka number, graded decomposition numbers}

\begin{document}

\maketitle

\begin{abstract}
We classify the core blocks of Ariki-Koike algebras by their moving vectors.  Using this classification, we obtain a necessary and sufficient condition for Scopes equivalence between two core blocks, and express the number of simple modules lying in a core block as a classical Kostka number.  Under certain conditions on the multicharge and moving vector, we further relate the graded decomposition numbers of these blocks in characteristic zero to the graded decomposition numbers of the Iwahori-Hecke algebras of type $A$.
\end{abstract}

\section{Introduction}
The Ariki-Koike algebra, introduced in the early 1990s by Susumu Ariki and Kazuhiko Koike \cite{AK}, is a significant family of algebras that generalises the well-known Iwahori-Hecke algebras of type $A$.
It arises in the context of the study of symmetries in the representation theory of Hecke algebras and categorification of quantum groups, and
is intimately connected with the category $\mathcal{O}$ for affine Lie algebras, as well as with the representation theory of
$\mathfrak{sl}_2$ and other Lie theoretic structures.
Its representation theory has deep connections with various important mathematical objects and phenomena, such as Young tableaux, the theory of symmetric functions, Lusztig’s theory of cell modules and Kazhdan-Lusztig polynomials.
The more recent introduction of the Khovanov-Lauda-Rouquier (KLR) algebras has endowed the Ariki-Koike algebras with a $\ZZ$-grading, providing new insights into their structure. In particular, the $v$-decomposition numbers arising from the canonical basis of the Fock space representation of the quantum affine algebra $\Uv$ are proved to be the graded decomposition numbers of Ariki-Koike algebras in characteristic zero \cite{BK-GradedDecompNos}, generalising the theorem of Ariki for ungraded decomposition numbers \cite{Ariki}.

Core blocks of Ariki-Koike algebras are first defined by Fayers \cite{Fayers-Coreblock}.
These are blocks for which the multipartitions lying in them are all multicores, and are believed to be the simplest non-simple blocks.
Despite their perceived simplicity, these blocks can have arbitrarily large weight, and the research conducted so far has focused on blocks with very small weight.
In fact the authors are not aware of any work on these blocks with level $\ell \geq 3$ and weight $k \geq 4$.

An important equivalence between the blocks of Ariki-Koike algebras is the Scopes equivalence.
This is a special type of Morita equivalence in which the correspondence between distinguished modules such as the Specht modules or the simple modules can be easily described.
In particular,
Scopes equivalence preserves decomposition numbers.
However,
determining when two blocks are Scopes equivalent is generally a challenging problem, even for level 1, i.e.\ the Iwahori-Hecke algebras of type $A$.
Dell'Arciprete \cite{D} and the first and third authors \cite{LT} provided some different sufficient conditions for two blocks to be Scopes equivalent.
However, as noted in \cite[Remark 4.15]{LT}, neither condition is necessary for two core blocks to be Scopes equivalent.

The concept of moving vectors was first introduced by the first and second authors \cite{LQ-movingvector} to classify the blocks of Ariki-Koike algebras with finite representation type.
This landmark idea plays a far more significant role in controlling the representation theory of the blocks than their weights,
and it serves a function similar to the weights of blocks in Iwahori-Hecke algebras of type $A$ for the higher-level Ariki-Koike algebras.
For their purposes, moving vectors were only defined in \cite{LQ-movingvector} for the so-called FLOTW multicharges, where the components in each multicharge are increasing, with the last no more than $e$ larger than the first.
Here, we provide a slightly different perspective to this groundbreaking concept and generalise it to any arbitrary multicharge.

In this paper, we present a new understanding of the core blocks of Ariki-Koike algebras through their moving vectors.
Our first main result (Theorem \ref{T:moving-vector-of-core-block}) establishes that a block of an Ariki-Koike algebra is a core block if and only if its moving vector has a zero component.
Using this classification of core blocks, we quickly determine that the connected components of the weight graph defined by Lyle and Ruff in \cite{LyleRuff-Decompositionnumber} are the `same' as those of the undirected graph that arises naturally from its moving vector (Theorem \ref{T:weight-graph}), and provide a straightforward method (Corollary \ref{C:decomposable}) to determine if a multipartition is decomposable in the sense of Fayers \cite{Fayers-Coreblock}.

Our characterisation of core blocks by their moving vectors yields several important results, as the multicores lying in them have very special properties.
We exploit this to obtain our second main result (Theorem \ref{T:Scopes}): a simple necessary and sufficient condition for two core blocks to be Scopes equivalent.
In particular, this demonstrates that, for level 2, all core blocks lying in the same affine Weyl group orbit are all Scopes equivalent.
This perhaps explain why the core blocks of Iwahori-Hecke algebras of type $B$ (i.e.\ level 2) are more tractable than others.

In our investigation of Scopes equivalence between core blocks, we encounter blocks in every Scopes equivalence class that have extremely rich properties.
Using these blocks, we derive our third main result (Theorem \ref{T:number-simple}), which expresses the number of simple modules in a core block as a classical Kostka number.

All the above results are valid for arbitrary multicharge.
Our final main result (Theorem \ref{T:decomp}) however requires some condition on the multicharge and only applies to a subset of core blocks with such multicharge.
For these blocks, Theorem \ref{T:decomp} establishes a relationship between their $v$-decomposition numbers and those in level one.

We now indicate the organisation of this paper. After providing the necessary background in the next section, we look at the moving vectors in Section \ref{S:moving-vectors}.  Section \ref{S:Scopes} is dedicated to a closer examination of the core blocks, where we provide a necessary and sufficient condition for two core blocks to be Scopes equivalent. In Section \ref{S:simple-decomp}, we relate the number of simple modules lying in a core block to a classical Kostka number, as well as the $v$-decomposition numbers of core blocks satisfying certain conditions to those in level $1$.

\section{Preliminaries}

\subsection{Common notations and conventions} \label{SS:notations}

Throughout, $\mathbb{Z}^+$ denotes the set of positive integers.
For $k \in \ZZ$, we write
$$\ZZ_{< k } := \{ x \in \ZZ \mid x < k\},$$
and similarly for $\ZZ_{\leq k}$, $\ZZ_{> k}$  and $\ZZ_{\geq k}$.
For $a,b \in \ZZ$,
$$[a,\, b] := \{ x \in \ZZ \mid a \leq x \leq b \}.$$

We fix $e,\ell \in \mathbb{Z}^+$ with $e \geq 2$.
For $m \in \{e,\ell\}$, we write $a \equiv_m b$ for $a \equiv b \pmod m$.
We also identify $\ZZ/m\ZZ$ with $[0,\, m-1]$.

For $\bt = (t_1,\dotsc, t_m) \in \ZZ^m$, write
$$|\bt| := \sum_{i=1}^m t_i.$$
In addition
$$\bone := (1,\dotsc, 1) \in \ZZ^m.$$
Of course, $\bone$ depends on $m$, which will always be clear from the context in this paper.
Moreover, for $x \in \ZZ$ with $x = am + b$, where $a\in \ZZ$ and $b \in \ZZ/m\ZZ$, define
$$\bv_m(x) := (\overbrace{a,\dotsc, a}^{m-b \text{ times}}, \overbrace{a+1,\dotsc, a+1}^{b \text{ times}}) \in \ZZ^m,
$$
which satisfies $|\bv_m(x)| = x$.

Furthermore, $\bbone_{\mathtt{p}}$ denotes the indicator function on the statement \smash[t]{$\mathtt{p}$, i.e.\ $\bbone_{\mathtt{p}} = \begin{cases}
1, &\text{if $\mathtt{p}$ is true};\\
0, &\text{otherwise}.
\end{cases}
$}

\subsection{$\beta$-sets and abaci} \label{SS:beta sets}

A $\beta$-set $\B$ is a subset of $\Z$ such that both $\max(\B)$ and $\min(\Z \setminus \B)$ exist. Its charge $\fs(\B)$ is defined as
$$\fs(\B) := |\B \cap \Z_{\geq 0}| - |\Z_{<0} \setminus \B|.$$
We denote the set of all $\beta$-sets by $\BB$ and the set of all $\beta$-sets with charges $s$ by $\BB(s)$.  Clearly, $\BB= \bigcup_{s \in \Z} \BB(s)$ (disjoint union).

Let $\B \in \BB$. The $\infty$-abacus display of $\B$ is a horizontal line with positions labelled by $\Z$ in an ascending order going from left to right, by putting a bead at the position $a$ for each $a \in \B$.
To obtain the $e$-abacus display of $\B$, we cut its $\infty$-abacus display into sections $\{ ae, ae+1, \dotsc, ae + e-1 \}$ ($a \in \Z$), and put the section $\{ ae, ae+1, \dotsc, ae + e-1 \}$ directly on top of $\{ (a+1)e, (a+1)e+1, \dotsc, (a+1)e + e-1 \}$.
Thus the $e$-abacus have $e$ infinitely long vertical runners, which we label as $0, 1,\dotsc, e-1$ going from left to right, and countably infinitely many horizontal rows, labelled by $\Z$ in an ascending order going from top to bottom, and the position on row $a$ and runner $i$ equals $ae + i$.
The {\em $e$-quotient of $\B$}, denoted $\quot_e(\B) = (\B_1,\dotsc, \B_{e})$, can be obtained as follows: for each $i \in \ZZ/e\ZZ$, treat each runner $i$ as an $\infty$-abacus, with the position $ae+i$ relabelled as $a$, and obtain the subset $\B_{i+1}$ from the beads in runner $i$ of the $e$-abacus of $\B$.  Formally,
$$
\B_{i+1} = \{ a \in \Z \mid ae+i \in \B \} = \{ \tfrac{x-i}{e} \mid x \in \B,\ x \equiv_e i \}.
$$
Its {\em $e$-core}, denoted $\core_e(\B)$, can be obtained by repeatedly sliding its beads in its $e$-abacus up their respectively runners to fill up the vacant positions above them, and its {\em $e$-weight}, denoted $\wt_e(\B)$, is the total number of times the beads move one position up their runners.
Note that
\begin{align*}
\sum_{i=1}^{e} \fs(\B_i) = \fs(\B) &= \fs(\core_e(\B)).
\end{align*}

Given $\beta$-set $\B$ and $k \in \Z$, define
$$
\B^{+k} :=  \{ x + k \mid x \in \B \}.
$$
Then $\B^{+k}$ is a $\beta$-set with $\fs(\B^{+k}) = \fs(\B) + k$.  Furthermore, if $\quot_e(\B) = (\B_1,\B_2,\dotsc, \B_{e})$, then $\quot_e(\B^{+1}) = ((\B_{e})^{+1}, \B_1,\dotsc, \B_{e-1})$.
Consequently,
\begin{equation} \label{E:quot}
\quot_e(\B^{+ke}) = ((\B_1)^{+k},(\B_2)^{+k},\dotsc, (\B_{e})^{+k}).
\end{equation}
In addition, we also have
\begin{align*}
\core_e(\B^{+k}) &= (\core_e(\B))^{+k};  \\
\wt_e(\B^{+k}) &= \wt_e(\B).
\end{align*}

\begin{lem} \label{L:hub}
Let $\B$ and $\C$ be $\beta$-sets.
Then
$$\fs(\B) - \fs(\C) = |\B \setminus \C| - |\C \setminus \B|.$$
In particular:
\begin{enumerate}
\item if $k := \fs(\B) - \fs(\C) > 0$, then there exist distinct $b_1,\dotsc, b_k \in \B \setminus \C$;
\item if $\B = \ZZ_{< m} \cup S$ for some $m \in \ZZ$ with $S \cap \ZZ_{< m} = \emptyset$, then $\fs(\B) = m + |S|$.
\end{enumerate}
\end{lem}

\begin{proof}
This is \cite[Lemma 2.2]{LT}.  Part (1) then follows since $|\B \setminus \C| \geq \fs(\B) - \fs(\C)$, and part (2) follows by putting $\C = \ZZ_{< m}$.
\end{proof}

For each $i \in \Z/e\Z$ and $\B \in \BB$, define $\hub_i(\B) = \hub_{i,e}(\B)$ by
\begin{align*}
\hub_i(\B) &:= |\{ x \in \B \mid x\equiv_e i,\, x-1 \notin \B \}| - |\{ x \in \B \mid x \equiv_e i-1,\, x+1 \notin \B \}|
\end{align*}
Furthermore, call the $e$-tuple $(\hub_0(\B), \dotsc, \hub_{e-1}(\B))$ the {\em $e$-hub of $\B$}, denoted $\Hub_e(\B)$.

We can deduce from Lemma \ref{L:hub} that if $\quot_e(\B) = (\B_1,\dotsc, \B_{e})$, then
\begin{equation}
\hub_j(\B) =
\fs(\B_{j+1}) - \fs(\B_{j}) - \delta_{j0}, \label{E:hub}
\end{equation}
for all $j \in \ZZ/e\ZZ$,
where
$\B_0$ is to be read as $\B_e$, and
$\delta_{i0}$ equals $1$ when $i=0$, and $0$ otherwise.
In particular,
\begin{equation}
\Hub_e(\B) = \Hub_e(\core_e(\B)). \label{E:hub-core}
\end{equation}

\begin{Def} \label{D:bfs}
Let $m \in \ZZ^+$.  Define a function $\bfs =\bfs_m : \BB^m \to \ZZ^m$ by
$$
\bfs(\B^{(1)}, \dotsc, \B^{(m)}) = (\fs(\B^{(1)}),\dotsc, \fs(\B^{(m)})).
$$
\end{Def}

\subsection{Partitions}
A partition is a weakly decreasing infinite sequence of nonnegative integers that is eventually zero.  Given partition $\lambda = (\lambda_1,\lambda_2,\dotsc)$, we write \begin{align*}
\len(\lambda) &:= \max \{ a \in \mathbb{Z}^+ : \lambda_a >0 \} ; \\
|\lambda| &:= \sum_{a=1}^{\len(\lambda)} \lambda_a.
\end{align*}
When $|\lambda| = n$, we say that $\lambda$ is a partition of $n$.
Its Young diagram $[\lambda]$ is defined as
$$ [\lambda] = \{ (i,j) \in \mathbb{Z} \mid 1 \leq i \leq \len(\lambda), 1 \leq j \leq \lambda_i \}.
$$
We write $\PP(n)$ for the set of partitions of $n$ and $\PP$ for the set of all partitions.  The unique partition of $0$, also known as the empty partition, shall be denoted as $\varnothing$.

Let $\lambda = (\lambda_1,\lambda_2,\dotsc)\in \PP$. For each $t \in \mathbb{Z}$,
$$
\beta_t(\lambda) := \{ \lambda_i+t-i \mid i \in \mathbb{Z}^+\}
$$
is a $\beta$-set, with $\fs(\beta_t(\lambda)) = t$.  We call $\beta_t(\lambda)$ the $\beta$-set of $\lambda$ with charge $t$.
The map $\beta : \PP \times \ZZ \to \BB$, defined by $\beta(\lambda,t) = \beta_t(\lambda)$, is a bijection, and for each $t \in \mathbb{Z}$, $\beta_t = \beta(-,t)$ gives a bijection between $\PP$ and $\BB(t)$.  In particular, when $\B$ is a $\beta$-set, we shall identify $\beta^{-1}(\B)$ with the unique partition $\lambda$ satisfying $\beta (\lambda, \fs(\B)) = \beta_{\fs(\B)}(\lambda) = \B$.

The $e$-weight $\wt_e(\lambda)$ and the $e$-core $\core_e(\lambda)$ of  $\lambda$ can be unambiguously defined to be $\wt_e(\beta_t(\lambda))$ and $\beta^{-1}(\core_e(\beta_t(\lambda)))$ respectively for any $t \in \mathbb{Z}$.

We note that if $\B$ is a $\beta$-set with $\quot_e(\B) = (\B_{1},\dotsc, \B_{e})$, then $\core_e(\B)$ is the unique $\beta$-set whose $e$-quotient is $(\beta_{\fs(\B_1)}(\varnothing), \dotsc, \beta_{\fs(\B_{e})} (\varnothing))$ and $\wt_e(\B) = \sum_{i=1}^{e} |\beta^{-1}(\B_{i})|$.

\subsection{Multipartitions}

An {\em $\ell$-partition of $n$} is an element $\bl = (\lambda^{(1)}, \dotsc, \lambda^{(\ell)}) \in \PP^{\ell}$ such that $\sum_{i=1}^{\ell} |\lambda^{(i)}| = n$.
We write $\PP^{\ell}(n)$ for the set of all $\ell$-partitions of $n$.
The unique $\ell$-partition of $0$ shall be denoted $\EP$.  Of course, $\EP$ depends on $\ell$, which will always be clear in the context in this paper.

Let $\bl \in \PP^{\ell}$.  The {\em Young diagram of $\bl$} is
$$[\bl] = \{ (a,b,i) \in (\Z^+)^3 \mid a \leq \len(\lambda^{(i)}),\ b \leq \lambda^{(j)}_a,\ i \leq \ell \}.$$
Elements of $[\bl]$ are called nodes.

We usually associate $\bl$ with an {\em $\ell$-charge} $\bt \in \Z^{\ell}$, and we write $(\bl; \bt)$ for such a pair.  The set of such pairs is then naturally identified with $\PP^{\ell} \times \Z^{\ell}$.

Let $\bt = (t_1,\dotsc, t_{\ell}) \in \Z^{\ell}$.
For $(a,b,i) \in \ZZ^3$ with $i \in [1,\,\ell]$, define its {\em $\bt$-content}, denoted $\cont_{\bt}(a,b,i)$, to be $b-a + t_i$.
Given $\bl \in \PP^{\ell}$, the rightmost nodes of $[\bl]$ are of the form $(a, \lambda^{(i)}_a, i)$ for $a \in \ZZ^+$ and $i \in [1,\,\ell]$.
Note that
$$\beta_{t_i}(\lambda^{(i)})= \{ \lambda^{(i)}_a - a + t_i \mid a \in \ZZ^+ \} = \{ \cont_{\bt}(a, \lambda^{(i)}_a, i) \mid a \in \ZZ^+ \},
$$
so that $\beta_{t_i}(\lambda^{(i)})$ is the set of $\bt$-contents of the rightmost nodes of $[\bl]$ coming from $\lambda^{(i)}$.

Given $(a,b,i) \in [\bl]$, its {\em $(e,\bt)$-residue}, denoted $\res_e^{\bt}(a,b,i)$,
is the residue class of its $\bt$-content modulo $e$; i.e.\ $\res_e^{\bt}(a,b,i) \in \ZZ/e\ZZ$ with $\res_e^{\bt}(a,b,i) \equiv_e \cont_{\bt}(a,b,i)$.
The {\em $(e,\bt)$-residues of $\bl$}, denoted $\res^{\bt}_e(\bl)$, is the multiset consisting precisely of $\res_e^{\bt}(a,b,i)$ for all $(a,b,i) \in [\bl]$.

If removing a node $\mathfrak{n}$ of $[\bl]$ from $[\bl]$ yields a Young diagram $[\bm]$ for some $\ell$-partition $\bm$, we call $\mathfrak{n}$ {\em a removable node of $\bl$} and {\em an addable node of $\bm$}, and write $\bl^{-\mathfrak{n}} := \bm$ and $\bm^{+\mathfrak{n}} := \bl$.
If $(a,b,i)$ is a removable node of $\bl$, then $b = \lambda^{(i)}_a > \lambda^{(i)}_{a+1}$, and so
$$
\cont_{\bt}(a, \lambda_a^{(i)}, i ) = \lambda_a^{(i)} - a + t_i
> \lambda_{(a+1)}^{(i)} - a + t_i
> \lambda_{(a+1)}^{(i)} - (a+1) + t_i = \cont_{\bt}(a+1,\lambda^{(i)}_{a+1}, i).
$$
Thus such a removable node corresponds bijectively to an $x \in \beta_{t_i}(\lambda^{(i)})$  with $x-1 \notin \beta_{t_i}(\lambda^{(i)})$ (namely, $x = \cont_{\bt}(a,b,i)$).
Note that if $\res_e^{\bt} (a,b,i) = j$, then $x = \cont_{\bt}(a,b,i) \equiv_e \res^{\bt}_e(a,b,i) = j$, so that $x$ lies in runner $j$ of the abacus display of $\beta_{t_i}(\lambda^{(i)})$.
In the same way, an addable node $(a',b',i')$ of $[\bl]$ corresponds bijectvely to an $x$ such that $x-1 \in \beta_{i'}(\lambda^{(i')})$ but $x \notin \beta_{i'}(\lambda^{(i')})$ (namely, $x = \cont_{\bt}(a',b',i')$).
In particular, this shows that each $\lambda^{(i)}$ can contribute at most one removable node and at most one addable node, and not both, with a given $\bt$-content.

We denote the set of addable (resp.\ removable) nodes of $[\bl]$ by $\AS(\bl)$ (resp. $\RS(\bl)$), and its subset consisting of nodes with $(e,\bt)$-residue $j$ ($j \in \ZZ/e\ZZ$) by $\AS^{e,\bt}_j(\bl)$ (resp.\ $\RS^{e,\bt}_j(\bl)$).
Furthermore, write $\AR(\bl) := \AS(\bl) \cup \RS(\bl)$ and
$\AR^{e,\bt}_j(\bl) := \AS^{e,\bt}_j(\bl) \cup \RS^{e,\bt}_j(\bl)$.



For $\bl = (\lambda^{(1)},\dotsc, \lambda^{(\ell)}) \in \PP^{\ell}$ and $\bt = (t_1,\dotsc, t_{\ell}) \in \Z^{\ell}$, we define
$$
\hub_j(\bl;\bt) := \sum_{i=1}^{\ell} \hub_j(\beta_{t_j}(\lambda^{(i)}))
$$
for each $j \in \Z/e\Z$.
This is the number of removable nodes of $\bl$ with $(e,\bt)$-residue $j$ minus the number of addable nodes of $\bl$ with $(e,\bt)$-residue $j$.
The {\em $e$-hub of $(\bl;\bt)$}, denoted $\Hub(\bl;\bt)$, is then
$$
\Hub(\bl;\bt) := (\hub_0(\bl;\bt), \dotsc, \hub_{e-1}(\bl;\bt)).
$$

For convenience, for $\bl = (\lambda^{(1)},\dotsc, \lambda^{(\ell)}) \in \PP^{\ell}$ and $\bt = (t_1,\dotsc, t_{\ell}) \in \Z^{\ell}$, we shall use the shorthand
$$\beta_{\bt}(\bl) := (\beta_{t_1}(\lambda^{(1)}), \dotsc, \beta_{t_{\ell}}(\lambda^{(\ell)})) \in \BB^{\ell}.$$


\subsection{Uglov's map} \label{SS:Uglov}

For each $i \in \{1,\dotsc,\ell\}$, define $\uu_i\ (= \uu_{e,\ell, i}) : \mathbb{Z} \to \mathbb{Z}$ by
$$
\uu_i(x) = (x - \bar{x})\ell + (\ell-i)e + \bar{x}.
$$
Here $\bar{x}$ is defined by $x \equiv_e \bar{x} \in \{ 0,1,\dotsc, e-1\}$.   In other words, if $x = ae + b$, where $a,b \in \mathbb{Z}$ with $0 \leq b < e$, then $$\uu_i(x) = ae\ell + (\ell-i)e + b.$$

We record some properties of these $\uu_i$'s which can be easily verified.

\begin{lem} \label{L:Uglov map}
\hfill
\begin{enumerate}
\item
$\uu_i(x) \equiv_e x$.
\item $\uu_i$ is injective.
\item
$
\mathbb{Z} = \bigcup_{i=1}^{\ell} \uu_i(\mathbb{Z})$ (disjoint union).
\item
$\uu_i(x) - e = \uu_{i+1}(x)$ for all $i \in [1,\, \ell-1]$, and $\uu_{\ell}(x) - e = \uu_1(x-e)$.
\end{enumerate}
\end{lem}

In \cite{Uglov}, Uglov defines a map $\UU\ (= \UU_{e,\ell}) : \BB^{\ell} \to \BB$ as follows:
$$
\UU(\B^{(1)},\dotsc,\B^{(\ell)}) := \bigcup_{j=1}^{\ell}
\uu_j(\B^{(j)}) \qquad (\B^{(1)},\dotsc, \B^{(\ell)} \in \BB).
$$
The best way to understand $\UU$ is via the abaci.
First we stack $\infty$-abacus displays of $\B^{(j)}$'s on top of each other, with the display of $\B^{(1)}$ at the bottom and that of $\B^{(\ell)}$ at the top.
Next, we cut up this stacked $\infty$-abaci into sections with positions $\{ ae, ae+1, \dotsc, ae + e-1 \}$ ($a \in \Z$), and putting the section with positions $\{ ae, ae+1, \dotsc, ae + e-1 \}$ on top of that with positions $\{ (a+1)e, (a+1)e+1, \dotsc, (a+1)e + e-1 \}$ (thus the section of the $\infty$-abacus of $\B^{(1)}$ with positions $\{ ae, ae+1, \dotsc, ae + e-1 \}$ is placed directly on top of the section of the $\infty$-abacus of $\B^{(\ell)}$ with positions $\{ (a+1)e, (a+1)e+1, \dotsc, (a+1)e + e-1 \}$), and relabel the rows of this $e$-abacus by $\Z$ in ascending order, with $0$ labelling the section of the $\infty$-abacus of $\B^{(\ell)}$ with positions $\{ 0, 1, \dotsc, e-1 \}$.
This yields the $e$-abacus display of $\UU(\B^{(1)},\dotsc,\B^{(\ell)})$.

Let $\bl = (\lambda^{(1)},\dotsc, \lambda^{(\ell)}) \in \PP^{\ell}$ and $\bt = (t_1,\dotsc, t_{\ell}) \in \ZZ^{\ell}$.
Recall that $|\bt| = \sum_{i=1}^{\ell} t_i$ and
$\beta_{\bt}(\bl) = (\beta_{t_1}(\lambda^{(1)}),\dotsc, \beta_{t_{\ell}}(\lambda^{(\ell)}))$.
It is not difficult to see from the description of $\UU$ that the addable nodes and removable nodes of $\bl \in \PP^{\ell}$ with nonzero $(e,\bt)$-residue correspond bijectively to the addable nodes and removable nodes of $\beta^{-1}(\UU(\beta_{\bt}(\bl)))$ with nonzero $(e,|\bt|)$-residue.
We record this formally below.

\begin{lem} \label{L:ARl1}
Let $\bt \in \ZZ^{\ell}$, and for all $\bnu \in \PP^{\ell}$, let $\Phi_{\bt}(\bnu) := \beta^{-1}(\UU(\beta_{\bt}(\bnu)))$.
For $\bl \in \PP^{\ell}$
and $j \in [1,\,e-1]$,
there is a unique function $\phi: \AR^{e,\bt}_j(\bl) \to \AR^{e,|\bt|}_j(\Phi_{\bt}(\bl))$ satisfying
$$\uu_i(\cont_{\bt}(\mathfrak{n})) = \cont_{|\bt|}(\phi(\mathfrak{n}))$$
for all $\mathfrak{n} = (a,b,i) \in \AR^{e,\bt}_j(\bl)$.
Furthermore:
\begin{enumerate}
\item $\phi$ sends removable (resp.\ addable) nodes of $\bl$ to removable (resp.\ addable) nodes of $\Phi_{\bt}(\bl)$;
\item $\phi$ is bijective;
\item $\Phi_{\bt}(\bl^{+\mathfrak{n}}) = (\Phi_{\bt}(\bl))^{+\phi(\mathfrak{n})}$ if $\mathfrak{n}$ is addable and
$\Phi_{\bt}(\bl^{-\mathfrak{n}}) = (\Phi_{\bt}(\bl))^{-\phi(\mathfrak{n})}$ if $\mathfrak{n}$ is removable.
\end{enumerate}
\end{lem}

\begin{proof}
Let $\bt = (t_1,\dotsc, t_{\ell})$ and $\bl = (\lambda^{(1)},\dotsc, \lambda^{(\ell)}) \in \PP^{\ell}$.
For each $\mathfrak{n} = (a,b,i) \in \AR^{e,\bt}_j(\bl)$,
exactly one of $\cont_{\bt}(\mathfrak{n})$ and $\cont_{\bt}(\mathfrak{n})-1$ lies in $\beta_{t_{i}}(\lambda^{(i)})$,
and hence
exactly one of $\uu_{i}(\cont_{\bt}(\mathfrak{n}))$ and $\uu_{i}(\cont_{\bt}(\mathfrak{n})-1)$ lies in
$\UU(\beta_{\bt}(\bl)) = \beta_{|\bt|}(\Phi_{\bt}(\bl))$.
Now,
$$\uu_{i}(\cont_{\bt}(\mathfrak{n})) \equiv_e \cont_{\bt}(\mathfrak{n})
\equiv_e \res^{\bt}_e(\mathfrak{n}) \equiv_e j \ne 0
$$
by Lemma \ref{L:Uglov map}(1), so that
$\uu_{i}(\cont_{\bt}(\mathfrak{n})-1) = \uu_{i}(\cont_{\bt}(\mathfrak{n})) -1$.
Thus $\uu_i(\cont_{\bt}(\mathfrak{n})) = \cont_{|\bt|}(\mathfrak{m}_{\mathfrak{n}})$ for a unique $\mathfrak{m}_{\mathfrak{n}} \in \AR(\Phi_{\bt}(\bl))$, with
$$
\res^{|\bt|}_e(\mathfrak{m}_{\mathfrak{n}}) \equiv_e \cont_{|\bt|}(\mathfrak{m}_{\mathfrak{n}})
= \uu_i(\cont_{\bt}(\mathfrak{n})) 
\equiv_e j.
$$
Thus, we have a function $\phi: \AR^{e,\bt}_j(\bl) \to \AR^{e,|\bt|}_j(\Phi_{\bt}(\bl))$, defined by $\phi(\mathfrak{n}) = \mathfrak{m}_{\mathfrak{n}}$, satisfying
$$
\uu_i(\cont_{\bt}(\mathfrak{n})) = \cont_{|\bt|}(\mathfrak{m}_{\mathfrak{n}}) = \cont_{|\bt|}(\phi(\mathfrak{n})).
$$
Note that the uniqueness of $\mathfrak{m}_{\mathfrak{n}}$ ensures that $f$ is both well-defined and unique satisfying $\uu_i(\cont_{\bt}(\mathfrak{n})) = \cont_{|\bt|}(\phi(\mathfrak{n}))$.
Furthermore,
\begin{align*}
\cont_{\bt}(\mathfrak{n}) \in \beta_{t_i}(\lambda^{(i)})
\ \Leftrightarrow\ \cont_{|\bt|}(\phi(\mathfrak{n})) = \uu_i(\cont_{\bt}(\mathfrak{n})) \in \UU(\beta_{\bt}(\bl))
= \beta_{|\bt|}(\Phi_{\bt}(\bl))
\end{align*}
so that $\mathfrak{n}$ is removable if and only if $\phi(\mathfrak{n})$ is removable.

If $\mathfrak{n} = (a,b,i), \mathfrak{n'} = (a',b',i') \in \AR^{e,\bt}_{j}(\bl)$ are such that $\phi(\mathfrak{n}') = \phi(\mathfrak{n})$, then
$$
\uu_{i'}(\cont_{\bt}(\mathfrak{n}') ) = \cont_{|\bt|}(\phi(\mathfrak{n}'))
= \cont_{|\bt|}(\phi(\mathfrak{n}))
= \uu_{i}(\cont_{\bt}(\mathfrak{n}))
$$
so that $i' = i$ since the $\uu_a(\ZZ)$'s are pairwise disjoint,
and hence $\cont_{\bt}(\mathfrak{n}') = \cont_{\bt}(\mathfrak{n})$ since $\uu_{i}$ is injective (see Lemma \ref{L:Uglov map}(2,3)).
This further implies that $\mathfrak{n}' = \mathfrak{n}$ as
$\lambda^{(i)}$ has at most one removable node and at most one addable node, but not both, with a given content.
Thus $f$ is injective.

If $\mathfrak{m} \in \AR^{e,|\bt|}_j(\Phi_{\bt}(\bl))$, then exactly one of $\cont_{|\bt|}(\mathfrak{m})$ or $\cont_{|\bt|}(\mathfrak{m}) -1$ lies in
$$\beta_{|\bt|}(\Phi_{\bt}(\bl)) = \UU(\beta_{\bt}(\bl)) = \bigcup_{i=1}^{\ell} \uu_i(\beta_{t_i}(\lambda^{(i)})),$$
and hence in
$\uu_i(\beta_{t_i}(\lambda^{(i)}))$ for a unique $i$ (since the $\uu_a(\ZZ)$'s are pairwise disjoint).
Thus exactly one of $\uu_i^{-1}(\cont_{|\bt|}(\mathfrak{m}))$ or $\uu_i^{-1}(\cont_{|\bt|}(\mathfrak{m})-1)$ lies in $\beta_{t_i}(\lambda^{(i)})$.
Since
\begin{align*}
\uu_i^{-1}(\cont_{|\bt|}(\mathfrak{m})) &\equiv_e \cont_{|\bt|}(\mathfrak{m})
\equiv_e \res^{|\bt|}_e(\mathfrak{m}) = j \ne 0,
\end{align*}
we have $\uu_i^{-1}(\cont_{|\bt|}(\mathfrak{m})-1) = \uu_i^{-1}(\cont_{|\bt|}(\mathfrak{m})) -1$.
Thus $\uu_i^{-1}(\cont_{|\bt|}(\mathfrak{m})) = \cont_{\bt}(\mathfrak{n})$ for some $\mathfrak{n} = (a,b,i) \in \AR(\bl)$.
Now
\begin{align*}
\res^{\bt}_e(\mathfrak{n}) \equiv_e \cont_{\bt}(\mathfrak{n})
&= \uu_i^{-1}(\cont_{|\bt|}(\mathfrak{m})) 
\equiv_e j.
\end{align*}
Thus $\mathfrak{n} \in \AR^{e,\bt}_j(\bl)$, and
$\cont_{|\bt|}(\mathfrak{m}) = \uu_i(\cont_{\bt}(\mathfrak{n}))$,
so that $\mathfrak{m} = \mathfrak{m}_{\mathfrak{n}} = \phi(\mathfrak{n})$.
Thus $f$ is surjective.

Let $\mathfrak{n} = (a,b,i) \in \AR^{e,\bt}_j(\bl)$.  If $\mathfrak{n}$ is addable, then
$$
\beta_{\bt}(\bl^{+\mathfrak{n}}) =
(\beta_{t_1}(\lambda^{(1)}), \dotsc, \beta_{t_{i-1}}(\lambda^{(i-1)}),
\beta_{t_i}(\lambda^{(i)}) \cup \{ \cont_{\bt}(\mathfrak{n}) \} \setminus \{ \cont_{\bt}(\mathfrak{n}) -1 \},
\beta_{t_{i+1}}(\lambda^{(i+1)}), \dotsc, \beta_{t_{\ell}}(\lambda^{(\ell)})),
$$
so that
\begin{align*}
\UU(\beta_{\bt}(\bl^{+\mathfrak{n}})) &=
\UU(\beta_{\bt}(\bl)) \cup \{ \uu_i(\cont_{\bt}(\mathfrak{n})) \} \setminus \{ \uu_i(\cont_{\bt}(\mathfrak{n})) -1 \} \\
&= \beta_{|\bt|}(\Phi_{\bt}(\bl)) \cup \{ \cont_{|\bt|}(\phi(\mathfrak{n})) \} \setminus \{ \cont_{|\bt|}(\phi(\mathfrak{n})) -1 \}
= \beta_{|\bt|}((\Phi_{\bt}(\bl))^{+\phi(\mathfrak{n})}),
\end{align*}
so that $\Phi_{\bt}(\bl^{+\mathfrak{n}}) = (\Phi_{\bt}(\bl))^{+\phi(\mathfrak{n})}$.
Similarly, one also gets $\Phi_{\bt}(\bl^{-\mathfrak{n}}) = (\Phi_{\bt}(\bl))^{-\phi(\mathfrak{n})}$ if $\mathfrak{n}$ is removable.
\end{proof}

For each $(\bl;\bt) \in \PP^{\ell} \times \ZZ^{\ell}$, define
$$
\core_e(\bl;\bt) := \core_e(\UU(\beta_{\bt}(\bl))) \qquad \text{and} \qquad
\wt_e(\bl;\bt) = \wt_e(\UU(\beta_{\bt}(\bl))).
$$
Note that while our $\wt_e(\bl;\bt)$ is the same as that in \cite{LT}, our $\core_e(\bl;\bt)$ is defined slightly differently---here $\core_e(\bl;\bt)$ is a $\beta$-set, while $\core_e(\bl;\bt)$ in \cite{LT} is defined to be the associated partition.

\subsection{The extended affine Weyl group $\EW_{m}$} \label{SS:Weyl}

Let $m \in \Z^+$.  The symmetric group on $m$ letters is denoted by $\sym{m}$.  We view $\sym{m}$ as a group of functions acting on $\{1,2,\dotsc,m\}$ so that we compose from right to left; for example, $(1,2)(2,3) = (1,2,3)$.  This is a Coxeter group of type $A_{m-1}$, with Coxeter generators $\cs_i = (i,i+1)$ for $1 \leq i \leq m-1$.

Given any nonempty set $X$, $\sym{m}$ has a natural right place permutation action on $X^{m}$ via
$$
(x_1,\dotsc, x_{m})^{\sigma} = (x_{\sigma(1)},\dotsc, x_{\sigma(m)}).
$$

Let $\{ \Be_1, \dotsc, \Be_m\}$ denote the standard $\Z$-basis for the free (left) $\Z$-module $\Z^m$.
The right place permutation action of $\sym{m}$ on $\Z^m$ induces the semidirect product $\EW_m = \sym{m} \ltimes \Z^m$,
where
\begin{equation}
(\sigma \bt)(\tau \bu) = (\sigma\tau) (\bt^{\tau} + \bu) = (\sigma\tau) (t_{\tau(1)} + u_1, \dotsc, t_{\tau(m)}+ u_m) \label{E:extendedaffineWeylgroup}
\end{equation}
for all $\sigma,\tau \in \sym{m}$ and $\bt,\bu \in \Z^m$ with $\bt = (t_1,\dotsc, t_m)$ and $\bu = (u_1,\dotsc, u_m)$.
This is an extended affine Weyl group of type $A^{(1)}_{m-1}$, and it contains the affine Weyl group $\AW_m$ of type $\widetilde{A}_{m-1}$ generated by $\{ \cs_0,\cs_1, \dotsc, \cs_{m-1}\}$, where $\cs_a = (a,a+1) \in \sym{m}$ for $1 \leq a \leq m-1$ as before, and $\cs_0 = (\Be_1 - \Be_m)(1,m)$.
Note that under this realisation,
\begin{equation}
\AW_m = \{ \sigma \bt \in \EW_m \mid \sigma \in \sym{m},\ \bt \in \ZZ^m,\ |\bt| = 0 \}. \label{E:affineWeylgroup}
\end{equation}

It is sometimes convenient to write an element of $\EW_m$ in the form $\bt \sigma$ where $\bt \in \ZZ^m$ and $\sigma \in \sym{m}$, which is equal to $\sigma \bt^{\sigma}$.

Given $w = \sigma \bt \in \EW_m$ where $\sigma \in \sym{m}$ and $\bt \in \ZZ^{m}$, we write $\overline {w}$ for $\sigma$ (thus $\overline{w}$ denotes the projection of $w$ onto $\sym{m}$).

We have a right action of $\EW_{\ell}$ on $\Z^{\ell}$ and $\PP^{\ell}$  via
\begin{align*}
\bt^{\sigma \bu} = \bt^{\sigma} + e\bu \quad \text {and} \quad
\bl^{\sigma \bu} = \bl^{\sigma}
\end{align*}
for $\bl \in \PP^{\ell}$, $\bt, \bu \in \Z^{\ell}$ and $\sigma \in \sym{\ell}$, which further induces a right action on $\PP^{\ell}\times \Z^{\ell}$.
We write $x^{\EW_{\ell}}$ for the orbit of $x$ (for $x$ in $\Z^{\ell}$ or $\PP^{\ell}$ or $\PP^{\ell}\times \Z^{\ell}$), and note that $\res_e(\bl;\bt)$ and $\Hub_e(\bl;\bt)$ are invariant under this action; i.e.\
\begin{equation*}
\res^{\bt^w}_e(\bl^w) = \res^{\bt}_e(\bl) \quad \text{and} \quad \Hub_e((\bl;\bt)^w) = \Hub_e(\bl;\bt) \label{E:res-hub}
\end{equation*}
for all $w \in \EW_{\ell}$.

Let
\begin{align*}
\AAF &:= \{ (a_1,\dotsc,  a_{\ell}) \in \ZZ^{\ell} \mid 0 \leq a_1 \leq a_2 \leq \dotsb \leq a_{\ell} < e \}; \\
\AAnobar &:= \{ (a_1,\dotsc, a_{\ell}) \in \ZZ^{\ell} \mid a_1 \leq a_2 \leq \dotsb \leq a_{\ell} < a_1 + e \}; \\
\AAbar &:= \{ (a_1,\dotsc, a_{\ell}) \in \ZZ^{\ell} \mid a_1 \leq a_2 \leq \dotsb \leq a_{\ell} \leq a_1 + e \}.
\end{align*}
Clearly, $\AAF \subseteq \AAnobar \subseteq \AAbar$, and $\AAF$ is a fundamental domain of the action of $\EW_{\ell}$ on $\ZZ^{\ell}$.

\begin{lem} \label{L:smallest}
Let $(\bl;\bt) \in \PP^{\ell} \times \ZZ^{\ell}$.  If $\hub_i(\bl;\bt) \leq 0$ for all $i \in \ZZ/e\ZZ$, then
$\core_e(\bl;\bt) = \UU(\beta_{\bt^*}(\EP))$ for a unique $\bt^* \in \AAbar$.
\end{lem}

\begin{proof}
Let $\quot_e(\UU(\beta_{\bt}(\bl))) = (\C_1,\dotsc, \C_{e})$.
Then
$$
\fs(\C_{j+1}) - \fs(\C_{j}) - \delta_{j0} = \hub_j(\UU(\beta_{\bt}(\bl))) = \hub_j(\bl;\bt) + (\ell-1)\delta_{j0}$$
for all $j \in \ZZ/e\ZZ$
by \eqref{E:hub} and \cite[Lemma 2.6(3)]{LT}, where $\C_0$ is to be read as $\C_e$.
Thus
$$
\fs(\C_1) \geq \fs(\C_2) \geq \dotsb \geq \fs(\C_{e}) \geq \fs(\C_1) - \ell.
$$
Consequently, since $\quot_e(\core_e(\bl;\bt)) = \beta_{\bs}(\EP)$ where $\bs = (\fs(\C_1), \dotsc, \fs(\C_{e}))$, this shows that for each $x \in \core_e(\bl;\bt)$ with $x\equiv_e j \in \ZZ/e\ZZ$, we have $x-1-\delta_{j0} \ell \in \core_e(\bl;\bt)$.  Thus $\core_e(\bl;\bt) = \UU(\beta_{\bt^*}(\EP))$ for some unique $\bt^* \in \ZZ^{\ell}$.
That $\bt^* \in \AAbar$ follows from \cite[Proposition 2.14]{JL-cores}.
\end{proof}

\begin{lem} \label{L:AA}
Let $\bt \in \ZZ^{\ell}$, and let $\rr_{\ell} = (1,2,\dotsc, \ell) \in \sym{\ell}$.
\begin{enumerate}
\item For $a \in \ZZ$ and $b \in \ZZ/\ell\ZZ$, $(\rr_{\ell}\Be_{\ell})^{a\ell + b} = \rr_{\ell}^b \bv_{\ell}(a\ell + b)$ (recall from Subsection \ref{SS:notations} that
$$
\bv_{\ell}(a\ell+b) = (a,\dotsc, a, \overbrace{a+1,\dotsc, a+1}^{b \text{ times}}) \in \ZZ^{\ell}).
$$
In particular, $\rr_{\ell}\Be_{\ell}$ has infinite order as an element of $\EW_{\ell}$.

\item We have $\bt \in \AAbar$ if and only if $\bt^{\rr_{\ell}\Be_{\ell}} \in \AAbar$.

\item If $\bt \in \AAnobar$ and $\bt^{\sigma \bv} \in \AAbar$ for some $\sigma \in \sym{\ell}$ and $\bv \in \ZZ^{\ell}$, then $\bv = \bv_{\ell}(a\ell+b)$ and $\sigma\rr_{\ell}^{-b} \in \left< \cs_i \mid t_i = t_{i+1} \right>$ for some $a \in \ZZ$ and $b \in \ZZ/\ell\ZZ$.

\item For $\bl \in \PP^{\ell}$,
$$\UU(\beta_{\bt^{\rr_{\ell}\Be_{\ell}}}(\bl^{\rr_{\ell}\Be_{\ell}})) =
(\UU(\beta_{\bt}(\bl)))^{+e}.
$$

\item
Restricting the right action of $\EW_{\ell}$ on $\ZZ^{\ell}$ to $\left< \rr_{\ell}\Be_{\ell} \right>$, we get a transitive and faithful action of $\left< \rr_{\ell}\Be_{\ell} \right>$ on $\bt^{\EW_{\ell}} \cap \AAbar$.

In particular, if $w,w' \in \EW_{\ell}$ such that $\bt^w,\bt^{w'} \in \AAbar$,
then $\bt^{w'} = (\bt^{w})^{(\rr_{\ell} \Be_{\ell})^k}$ for a unique $k \in \ZZ$.
\end{enumerate}
\end{lem}

\begin{proof}
Parts (1)--(3) is \cite[Proposition 3.5]{LT}, while part (4) follows from \cite[Proposition 3.3(1)]{LT}.

For part (5), that $\left< \rr_{\ell}\Be_{\ell} \right>$ acts on $\bt^{\EW_{\ell}} \cap \AAbar$ follows from part (3).
Let $\bt_{\mathsf{F}} = (u_1,\dotsc, u_{\ell}) \in \bt^{\EW_{\ell}} \cap \AAF$, and let $w_0 \in \EW_{\ell}$ be such that $\bt_{\mathsf{F}} = \bt^{w_0}$.
To show that the action of $\left< \rr_{\ell}\Be_{\ell} \right>$ on $\bt^{\EW_{\ell}} \cap \AAbar$ is transitive, let $w\in \EW_{\ell}$ be such that $\bt^w \in \AAbar$.
By part (4), since $\bt_{\mathsf{F}} \in \AAnobar$ and $\bt^w= (\bt^{w_0})^{w_0^{-1}w} = \bt_{\mathsf{F}}^{w_0^{-1}w}\in \AAbar$, we have that
$w_0^{-1} w = \sigma \bv_{\ell}(a\ell+b)$ for some $\sigma \in \sym{\ell}$, $a\in \ZZ$ and $b \in \ZZ/\ell\ZZ$ such that $\sigma \rr_{\ell}^{-b} \in \left< \cs_i \mid u_i = u_{i+1} \right>$.
Consequently,
$$
\bt^w = \bt_{\mathsf{F}}^{w_0^{-1}w} = (\bt_{\mathsf{F}}^{\sigma\rr_{\ell}^{-b}})^{\rr_{\ell}^b \bv_{\ell}(a\ell+b)} =
    \bt_{\mathsf{F}}^{(\rr_{\ell} \Be_{\ell})^{a\ell + b}} ,
$$
as desired.
That the action is faithful now follows immediately from the fact that $\bt^{\EW_{\ell}} \cap \AAbar$ is an infinite set.
\end{proof}

The $e$-cores and $e$-weights of $\ell$-partitions with their associated charges in the same $\EW_{\ell}$-orbit are related as follows:

\begin{thm}[{\cite[Theorem 3.6]{LT}}] \label{T:core-n-weight}
Let $(\bl;\bt) \in \PP^{\ell} \times \ZZ^{\ell}$, and let $(\bm;\bu) \in (\bl;\bt)^{\EW_{\ell}}$.  Then
\begin{enumerate}
  \item $\beta^{-1}(\core_e(\bm;\bu)) = \beta^{-1}(\core_e(\bl;\bt))$;
 \item $\wt_e(\bm;\bu) = \min( \wt_e( (\bl;\bt)^{\EW_{\ell}} ) )$ if and only if $\bu \in \AAbar$.
\end{enumerate}
\end{thm}


For each $k \in \ZZ$, we have a left action $\DDot{k}$ of $\EW_e$ on $\Z$ defined by
\begin{alignat*}{2}
\cs_i \DDot{k} x &=
\begin{cases}
x-1, &\text{if } x \equiv_e i \\
x+1, &\text{if } x \equiv_e i-1 \\
x, &\text{otherwise}
\end{cases}
&\qquad & (i \in [ 1,\, e-1]); \\
\Be_j \DDot{k} x &=
\begin{cases}
x+k e, &\text{if } x \equiv_e j-1 \\
x, &\text{otherwise}
\end{cases}
&& (j \in [1,\, e]).
\end{alignat*}

We extend this left action $\DDot{k}$ to the set $\BB$ of $\beta$-sets by $w \DDot{k} \B = \{ w \DDot{k} b \mid b \in \B \}$ for $w \in \EW_e$ and $\B \in \BB$.
For each $j \in \ZZ/e\ZZ$, we have
\begin{align*}
\fs(\cs_j \DDot{k} \B) &= \fs(\B);  \\
\wt_e(\cs_j \DDot{k} \B) &= \wt_e(\B);  \\
\core_e(\cs_j \DDot{k} \B) &= \cs_j \DDot{k} \core_e(\B).
\end{align*}
We extend $\DDot{k}$ further to $\BB^{\ell}$ by acting on each and every component.
Then
\begin{equation}
\UU(\cs_j \DDot{1} (\B^{(1)},\dotsc, \B^{(\ell)})) = \cs_j \DDot{\ell} \UU(\B^{(1)},\dotsc, \B^{(\ell)}) \label{E:cs_i-U}
\end{equation}
for all $j \in \ZZ/e\ZZ$ by \cite[(2.16)]{LT}.

For $\bl = (\lambda^{(1)},\dotsc, \lambda^{(\ell)}) \in \PP^{\ell}$ and $\bt = (t_1,\dotsc, t_{\ell}) \in \Z^{\ell}$, we have
a left action $\DDot{\bt}$ of $\AW_e$  on $\PP^{\ell}$,
where for each $j \in \ZZ/e\ZZ$, $\cs_j \DDot{\bt} \bl$ is the $\ell$-partition obtained from $\bl$ by removing all its removable nodes of $(e,\bt)$-residue $j$ and adding all its addable nodes of $(e,\bt)$-residue $j$.
This satisfies
\begin{equation}
\beta_{\bt}(\cs_j \DDot{\bt} \bl) = \cs_j \DDot{1} \beta_{\bt}(\bl)
\label{E:t-action}
\end{equation}
by \cite[(2.17)]{LT},
and further induces a left action $\DDot{}$ of $\AW_e$ on $\PP^{\ell} \times \Z^{\ell}$ via $\cs_j \DDot{} (\bl;\bt) = (\cs_j \DDot{\bt} \bl; \bt)$ for all $j \in \ZZ/e\ZZ$,
which commutes with the right action of $\EW_{\ell}$ described above (interested readers please see last paragraph before Section 3 of \cite{LT} for details).

We also have a natural left action $\CDot{k}$ of $\EW_e$ on $\BB^e$ and $\ZZ^e$ as follows:
\begin{align*}
\cs_i \CDot{k} (\B_1,\dotsc, \B_{e}) &= (\B_1,\dotsc,\B_{i-1}, \B_{i+1}, \B_{i}, \B_{i+2},\dotsc, \B_{e}), \\
\Be_j \CDot{k} (\B_1,\dotsc, \B_{e}) &= (\B_1,\dotsc,\B_{j-1}, \B_j^{+k}, \B_{j+1},\dotsc, \B_{e}), \\
\cs_i \CDot{k} (a_1,\dotsc, a_{e}) &= (a_1,\dotsc,a_{i-1}, a_{i+1}, a_{i}, a_{i+2},\dotsc, a_{e}), \\
\Be_j \CDot{k} (a_1,\dotsc, a_{e}) &= (a_1,\dotsc,a_{j-1}, a_j +k, a_{j+1},\dotsc, a_{e}),
\end{align*}
for all $i \in [1,\,e-1]$, $j \in [1,\,e]$, $(\B_1,\dotsc, \B_e) \in \BB^e$ and $(a_1,\dotsc, a_e) \in \ZZ^e$.
Note that $\CDot{k}$ and $\DDot{k}$ are related in the following way:
\begin{align}
w \CDot{k} \quot_e(\B) &= \quot_e( w \DDot{k} \B) \label{E:CDot-tuple-beta-sets}, \\
w \CDot{k} \bfs(\quot_e(\B)) &= \bfs(w \CDot{k} \quot_e(\B)) = \bfs(\quot_e(w \DDot{k} \B)) \label{E:CDot-tuple-integers}
\end{align}
for all $w \in \EW_e$ and $\B \in \BB$ (recall the function $\bfs$ in Definition \ref{D:bfs}).

Note that when $k = 0$, we have
\begin{equation*}
\mathbf{t}\sigma \CDot{0} (a_1,\dotsc, a_e) = (a_{\sigma^{-1}(1)},\dotsc, a_{\sigma^{-1}(e)}) \label{E:CDot0}
\end{equation*}
for all $\mathbf{t} \in \ZZ^e$ and $\sigma \in \sym{e}$, so that $\CDot{0}$ is really just the left conjugation action of $\EW_{e}$ on its Abelian normal subgroup $\ZZ^e$.

\begin{lem} \label{L:Weyl action} \hfill
\begin{enumerate}
\item
Let $(\bm;\bu) \in \PP^e \times \ZZ^e$ and $k \in \ZZ$.
Then
$$w \CDot{k} \beta_{\bu}(\bm) = \beta_{\smash[t]{w \CDot{k} \bu}}(\bm^{\overline{w}^{-1}})$$ for all $w \in \EW_{e}$.
(Recall that $\overline{w}$ denotes the projection of $w$ onto $\sym{e}$.)

\item
Let $(\bl;\bt) \in \PP^{\ell} \times \ZZ^{\ell}$.
Then $(w \DDot{\bt} \bl)^{w'} = w \DDot{\bt^{w'}} \bl^{w'}$ for all $w \in \AW_e$ and $w' \in \EW_{\ell}$.
\end{enumerate}
\end{lem}

\begin{proof}
It is straightforward to verify the statement for $w= \Be_i$ ($i \in [1,e]$) and $w=\cs_j$ ($j \in [1,e-1]$), and the general statement follows immediately.

For part (2), we have
$$
(w \DDot{\bt^{w'}} \bl^{w'}; \bt^{w'}) = w \DDot{} (\bl^{w'};\bt^{w'})
=  w \DDot{} ((\bl;\bt)^{w'})
= ( w \DDot{} (\bl;\bt))^{w'}
= (w \DDot{\bt} \bl;\bt)^{w'}
= ( (w\DDot{\bt} \bl)^{w'}; \bt^{w'}),
$$
so that $w \DDot{\bt^{w'}} \bl^{w'} = (w\DDot{\bt} \bl)^{w'}$ as desired.
\end{proof}

\subsection{Ariki-Koike algebras} \label{SS:AK-alg}

From now on, $\FF$ denotes a fixed field of characteristic $p$, where we allow $p=0$, and we assume that $\FF$ contains a primitive $e$-th root of unity if $e \ne p$.
Let $q \in \FF$ be such that $q = 1$ if $e = p$, and $q$ is a primitive $e$-th root of unity otherwise.
Fix $\br = (r_1,\dotsc, r_{\ell}) \in \Z^{\ell}$, and let $n \in \Z^+$.
The {\em Ariki-Koike algebra} $\HH_n = \HH_{\FF, q, \br}(n)$ is the unital $\FF$-algebra generated by $\{ T_0,\dotsc, T_{n-1} \}$ subject to the following relations:
\begin{alignat*}{2}
(T_0 - q^{r_1})(T_0 - q^{r_2}) \dotsm (T_0 - q^{r_{\ell}}) &= 0; \\
(T_a - q)(T_a+1) &= 0 &\qquad &(a \in [1,\, n-1] ); \\
T_0T_1T_0T_1 &= T_1T_0T_1T_0; && \\
T_aT_{a+1}T_a &= T_{a+1}T_a T_{a+1} &\qquad &(a \in [1,\, n-2]); \\
T_aT_b &= T_aT_b &\qquad &(|a-b| \geq 2).
\end{alignat*}
It is clear from the definition that $\HH_n$ only depends on the orbit $\br^{\EW_{\ell}}$ and not on $\br$.

The set of {\em Specht modules} $\{\Sp^{\bl} = \Sp^{\bl}_{\br} \mid \bl \in \PP^{\ell}(n) \}$ is a distinguished class of $\HH_n$-modules.  Here, for our purposes, we use the Specht modules defined by \cite{Du-Rui} instead of the original ones constructed by Dipper, James and Mathas \cite{DJM}.
We note that, unlike $\HH_n$, these Specht modules actually depend on the order of the $r_j$'s.

For $\bl \in \PP^{\ell}$, define its
{\em $\HH$-core}, {\em $\HH$-weight} and {\em $\HH$-hub}, denoted $\core_{\HH}(\bl)$, $\wt_{\HH}(\bl)$ and $\Hub_{\HH}(\bl) = (\hub^{\HH}_0(\bl), \dotsc, \hub^{\HH}_{e-1}(\bl)) $,
by
\begin{align*}
\core_{\HH}(\bl) &= \beta^{-1} (\core_e((\bl;\br)^w)), \\
\wt_{\HH}(\bl) &= \wt_e((\bl;\br)^w), \\
\Hub_{\HH}(\bl) & = \Hub_e((\bl;\br)^w),
\end{align*}
where $w$ is any element of $\EW_{\ell}$ such that $\br^w \in \AAbar$.
%
By Theorem \ref{T:core-n-weight}, $\core_{\HH}(\bl)$ and $\wt_{\HH}(\bl)$ are well-defined, i.e.\ do not depend on the choice of $w$ such that $\br^w \in \AAbar$, while $\Hub_{\HH}(\bl)$ is well-defined by the definition of $\Hub_e$.

Given a block $B$ of $\HH_n$, we say that an $\ell$-partition $\bl$ {\em lies in $B$} if $\Sp^{\bl}$ lies in $B$.

\begin{thm}[{\cite[Corollary 4.5]{LT}}] \label{T:Naka}
Let $\bl,\bm \in \PP^{\ell}$.  Then $\bl$ and $\bm$ lie in the same block of $\HH_n$ if and only if $\core_{\HH}(\bl) = \core_{\HH}(\bm)$ and $\wt_{\HH}(\bl) = \wt_{\HH}(\bm)$.
\end{thm}

By Theorem \ref{T:Naka},
we may unambiguously define the core and weight of a block $B$, denoted $\core(B)$ and $\wt(B)$, as the common $\HH$-core and the common $\HH$-weight of the $\ell$-partitions lying in $B$. Moreover, these two attributes completely determine the block.

By \eqref{E:hub-core}, $\HH$-hub is another block invariant of the Ariki-Koike algebras, and we write $\Hub(B) = (\hub_0(B),\dotsc,\hub_{e-1}(B))$ for the common $\HH$-hub of the $\ell$-partitions lying in $B$.

Recall that we have a left action $\DDot{\br}$ of $\AW_e$ on $\PP^{\ell}$.
This induces a left action $\DDot{}$ of $\AW_e$ on the set of blocks of Ariki-Koike algebras (with common $\ell$-charge $\br$) such that if $B$ is a block of $\HH_n$, then $\cs_j \DDot{} B$ is the block of $\HH_{n-\hub_j(B)}$ with
$$
\core(\cs_j \DDot{} B ) = \beta^{-1}(\cs_j \DDot{\ell} \beta_{|\br|}(\core(B))) \quad \text{ and } \quad \wt(\cs_j \DDot{} B) = \wt(B)
$$
(see, for example, \cite[Proposition 4.8]{LT}).

\section{Moving Vectors} \label{S:moving-vectors}

The concept of moving vectors was first introduced by the first and second authors in \cite{LQ-movingvector} for $\ell$-partitions with associated $\ell$-charge lying in $\AAbar$ to classify the blocks of Ariki-Koike algebras with finite representation type.

In this section, we will use moving vectors to classify the core blocks of Ariki-Koike algebras, and obtain some immediate consequences for the weight graphs for $\ell$-partitions.

\subsection{Definition and properties}
We present in this subsection a slightly different perspective of moving vectors, generalising it to all $\ell$-charges.

Recall the Uglov map $\UU$ in Subsection \ref{SS:Uglov}.  Post-composing $\UU$ with $\quot_e$ gives a bijection from $\BB^{\ell}$ to $\BB^{e}$.  Together with the bijection $\beta$ between $\PP \times \ZZ$ and $\BB$, this further induces a bijection $\bij_{\ell,e} : \PP^{\ell} \times \ZZ^{\ell} \to \PP^e \times \ZZ^e$ defined by $\bij_{\ell,e}(\bl;\bt) = (\bm;\bu)$ where $\quot_e(\UU(\beta_{\bt}(\bl))) = \beta_{\bu}(\bm)$.  We like to think of $\bij_{\ell,e}$ as the {\em rank-level duality map}.

\begin{lem} \label{L:ell-residue}
Let $(\bl;\bt) \in \PP^{\ell} \times \ZZ^{\ell}$, and let $\bij_{\ell,e}(\bl;\bt) = (\bm;\bu)$ (thus $(\bm;\bu) \in \PP^e \times \ZZ^e$ with $\beta_{\bu}(\bm) = \quot_e(\UU(\beta_{\bt}(\bl)))$). Then
\begin{enumerate}
\item $\quot_e(\core_e(\bl;\bt)) = \beta_{\bu}(\EP)$; 
\item $\res^{\bu}_{\ell}(\bm)$ is completely determined by $\bt$, $\core_e(\bl;\bt)$ and $\wt_e(\bl;\bt)$;
\item $\bij_{\ell,e}((\bl;\bt)^{\rr_{\ell} \Be_{\ell}}) = (\bm;\bu + \bone)$, where $\rr_{\ell} = (1,\dotsc, \ell) \in \sym{\ell}$.
\end{enumerate}
\end{lem}

\begin{proof}
Part (1) is clear.

For part (2), let $(\bl^*;\bt^*) \in \PP^{\ell} \times \ZZ^{\ell}$ be such that $\UU(\beta_{\bt^*}(\bl^*)) = \core_e(\bl;\bt)$.
Note that a removable node of $\bm$ with $(\ell,\bu)$-residue $i$ corresponds to an $x \in \UU(\beta_{\bt}(\bl))$ such that $\lfloor x/e \rfloor \equiv_{\ell} i$ and $x-e \notin \UU(\beta_{\bt}(\bl))$,
i.e.\ an $x \in \uu_{\ell-i}(\beta_{t_{\ell-i}}(\lambda^{(\ell-i)}))$ such that $x-e \notin \uu_{\ell-i+1}(\beta_{t_{\ell-i+1}}(\lambda^{(\ell-i+1)}))$ (where $\ell-i+1$ is to be read as $1$ if $i = 0$).
Removing this node from $\bm$ thus corresponds to removing $\uu_{\ell-i}^{-1}(x)$ from $\beta_{t_{\ell-i}}(\lambda^{(\ell-i)})$ and adding $\uu_{\ell-i+1}^{-1}(x-e)$ to $\beta_{t_{\ell-i+1}}(\lambda^{(\ell-i+1)})$.
Consequently, this results in subtracting 1 from the charge of $\beta_{t_{\ell-i}}(\lambda^{(\ell-i)})$ and adding 1 to the charge of $\beta_{t_{\ell-i+1}}(\lambda^{(\ell-i+1)})$.
Since repeatedly removing the removable nodes from $\bm$ eventually arrives at $\EP$, and
$$
\beta_{\bu}(\EP) = \quot_e(\core_e(\bl;\bt)) = \quot_e(\UU(\beta_{\bt^*}(\bl^*))),
$$
we have $t^*_{\ell-i} = t_{\ell -i} - C_{i} + C_{i+1}$, where $C_k$ is the total number of nodes of $\bm$ with $(\ell,\bu)$-residue $k$ for each $k \in \ZZ/\ell\ZZ$ and $\bt^* = (t^*_1,\dotsc, t^*_{\ell})$. This is true for all $i \in \ZZ/\ell\ZZ$, so that we get a system of $\ell$ linear equations in $C_i$'s which has a one-dimensional solution space.  However, we further require that
$$\sum_{i = 1}^{\ell} C_i = \wt_e(\UU(\beta_{\bt}(\bl))) = \wt_e(\bl;\bt),$$
which is a linear equation in $C_i$'s independent of the earlier linear equations.
Thus the $C_i$'s, and hence $\res^{\bu}_{\ell}(\bm)$, have a unique solution when $\bt$, $\wt_e(\bl;\bt)$ and $\core_e(\bl;\bt)$ are fixed.

For part (3), we have
\begin{alignat*}{2}
\quot_e(\UU(\beta_{\bt^{\rr_{\ell} \Be_{\ell}}}(\bl^{\rr_{\ell}\Be_{\ell}})))
&= \quot_e((\UU(\beta_{\bt}(\bl)))^{+e}) &&(\text{by Lemma \ref{L:AA}(4)}) \\
&= ((\beta_{u_1}(\mu^{(1)}))^{+1},\dotsc, (\beta_{u_e}(\mu^{(e)}))^{+1}) &&(\text{by \eqref{E:quot}})\\
&= (\beta_{u_1+1}(\mu^{(1)}),\dotsc, \beta_{u_e+1}(\mu^{(e)})) = \beta_{\bu+\bone}(\bm).
\end{alignat*}
Thus $\bij_{\ell,e}((\bl;\bt)^{\rr_{\ell}\Be_{\ell}}) = (\bm;\bu + \bone)$.
\end{proof}

\begin{Def}
Let $(\bl;\bt) \in \PP^{\ell} \times \ZZ^{\ell}$.  The {\em $e$-moving vector of $(\bl;\bt)$}, denoted $\mv_e(\bl;\bt)$, is defined as
$$
\mv_e(\bl;\bt) := (C_{\ell-1}, C_{\ell-2}, \dotsc, C_0),
$$
where $C_i$ is the total number of nodes of $\bm$ with $(\ell,\bu)$-residue $i$ and $\bij_{\ell,e}(\bl;\bt) = (\bm;\bu)$.
\end{Def}

\begin{eg} \label{Eg:mv}
Let $e = 5$, $\ell = 4$, $\bt = (t_1,t_2,t_3,t_4) = (1,3,3,6)$, and $\bl = (\lambda^{(1)}, \lambda^{(2)}, \lambda^{(3)}, \lambda^{(4)}) = ((3,2,1^4), (4,2,1), (2^2,1), (1))$.

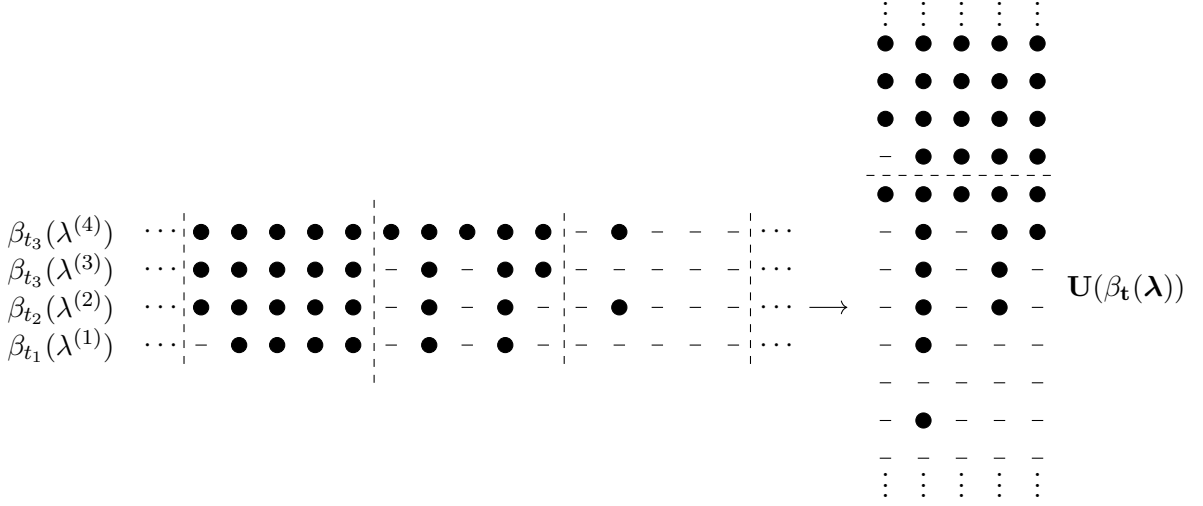
\begin{center}
\begin{tikzpicture}[scale = 0.5]
\node [left] at (-8,1) {$\beta_{t_1}(\lambda^{(1)})$};
\node [left] at (-8,2) {$\beta_{t_2}(\lambda^{(2)})$};
\node [left] at (-8,3) {$\beta_{t_3}(\lambda^{(3)})$};
\node [left] at (-8,4) {$\beta_{t_3}(\lambda^{(4)})$};

\foreach \y in {1,2,3,4}
\foreach \x in {-7,9.2}
\node at (\x,\y) {$\dotsm$};

\foreach \y in {2,3,4}
\foreach \x in {-2,-3,-4,-5,-6}
\draw [radius=2mm, fill=black] (\x,\y) circle;
\foreach \x in {-2,-3,-4,-5}
\draw [radius=2mm, fill=black] (\x,1) circle;
\draw (-6.15, 1) -- (-5.85, 1);

\foreach \x in {-1,0,1,2,3}
\draw [radius=2mm, fill=black] (\x,4) circle;

\foreach \x in {0,2,3}
\draw [radius=2mm, fill=black] (\x,3) circle;
\foreach \x in {-1,1}
\draw (\x-0.15, 3) -- (\x+ 0.15, 3);

\foreach \x in {0,2}
\draw [radius=2mm, fill=black] (\x,2) circle;
\foreach \x in {-1,1,3}
\draw (\x-0.15, 2) -- (\x+ 0.15, 2);

\foreach \x in {0,2}
\draw [radius=2mm, fill=black] (\x,1) circle;
\foreach \x in {-1,1,3}
\draw (\x-0.15, 1) -- (\x+ 0.15, 1);

\foreach \x in {4,5,6,7,8}
\foreach \y in {1,2,3,4}
\draw (\x-0.15, \y) -- (\x+ 0.15, \y);
\foreach \y in {4,2}
\draw [radius=2mm, fill=black] (5,\y) circle;

\draw [dashed] (-1.45,0) -- (-1.45, 5);
\foreach \x in {-6.45,3.55, 8.45}
\draw [dashed] (\x,0.5) -- (\x, 4.5);

\draw[->] (10,2) -- (11,2);

\foreach \x in {12,13,14,15,16}
\foreach \y in {10,-2.5}
\node at (\x,\y) {$\vdots$};

\foreach \y in {7,8,9}
\foreach \x in {12,13,14,15,16}
\draw [radius=2mm, fill=black] (\x,\y) circle;
\foreach \x in {13,14,15,16}
\draw [radius=2mm, fill=black] (\x,6) circle;
\draw (11.85, 6) -- (12.15, 6);

\foreach \x in {12,13,14,15,16}
\draw [radius=2mm, fill=black] (\x,5) circle;

\foreach \x in {13,15,16}
\draw [radius=2mm, fill=black] (\x,4) circle;
\foreach \x in {12,14}
\draw (\x-0.15, 4) -- (\x+ 0.15, 4);

\foreach \x in {13,15}
\draw [radius=2mm, fill=black] (\x,3) circle;
\foreach \x in {12,14,16}
\draw (\x-0.15, 3) -- (\x+ 0.15, 3);

\foreach \x in {13,15}
\draw [radius=2mm, fill=black] (\x,2) circle;
\foreach \x in {12,14,16}
\draw (\x-0.15, 2) -- (\x+ 0.15, 2);

\foreach \x in {12,13,14,15,16}
\foreach \y in {-2,-1,0,1}
\draw (\x-0.15, \y) -- (\x+ 0.15, \y);
\foreach \y in {-1,1}
\draw [radius=2mm, fill=black] (13,\y) circle;

\draw[dashed] (11.5,5.5) -- (16.5,5.5);

\node [right] at (16.5, 2.5) {$\UU(\beta_{\bt}(\bl))$};

\end{tikzpicture}
\end{center}
Thus $\bij_{4,5}(\bl;\bt) = (((1), (1), \varnothing, \varnothing,\varnothing); (0,6,1,4,2))$, and $\mv_e(\bl;\bt) = (0,1,0,1)$.
\end{eg}

\begin{rem}
Let $(\bl;\bt) \in \PP^{\ell} \times \ZZ^{\ell}$ with $\bij_{\ell,e}(\bl;\bt) = (\bm;\bu)$.
A node $(a,b,j)$ of $[\bm]$ corresponds bijectively to a triple $(x,y,z) \in \ZZ^3$ with $x < y \leq z$ and $x\equiv_e y \equiv_e z$ such that $x \notin \UU(\beta_{\bt}(\bl))$, $z \in \UU(\beta_{\bt}(\bl))$, and
$y = \cont_{\bu}(a,b,j) e + j-1$.
Since  $\res^{\bu}_{\ell}(a,b,j) \equiv_{\ell} \cont_{\bu}(a,b,j)$,
we see that $(a,b,j)$ has $(\ell,\bu)$-residue $i$ if and only if $y \in \uu_{\ell-i}(\ZZ)$.
From this description of nodes of $[\bm]$ with $(\ell,\bu)$-residue $i$, one can see that our definition of the moving vector is equivalent to that first defined in \cite{LQ-movingvector}, even though the latter only does so for $\bt \in \AAbar$.
\end{rem}

We record the following result which is already obtained in the proof of Lemma \ref{L:ell-residue}(2).

\begin{cor} \label{C:t*}
Let $(\bl;\bt) \in \PP^{\ell} \times \ZZ^{\ell}$, and let $\UU(\beta_{\bt^*}(\bl^*)) = \core_e(\bl;\bt)$.  Then
$\bt^* = (t^*_1,\dotsc, t^*_{\ell})$ satisfies $$t^*_i = t_i - C_{\ell-i} + C_{\ell-i+1}$$ for all $i \in [1,\, \ell]$, where $C_k$ is the number of nodes of $\bm$ with $(\ell,\bu)$-residue $k$ and $\bij_{\ell,e}(\bl;\bt) = (\bm;\bu)$, and $\ell-i+1$ is to be read as $0$ if $i = 1$.

In particular, $\bt^*$ is completely determined by $\bt$ and $\mv_e(\bl;\bt)$.
\end{cor}

The following lemma generalises \cite[Lemma 4.3.1]{LQ-movingvector}:

\begin{lem} \label{L:mv-construction}
Let $\bl = (\lambda^{(1)},\dotsc, \lambda^{(\ell)}) \in \PP^{\ell}$ and let $\bt = (t_1,\dotsc, t_{\ell}) \in \ZZ^{\ell}$.
Let $k_1,\dotsc, k_{\ell-1} \in \mathbb{Z}_{\geq 0}$ and define
  $\bu = (u_1,\dotsc, u_{\ell}) \in \ZZ^{\ell}$ by
  $u_i := t_i + k_i - k_{i-1}$
  for all $i \in [1,\,\ell]$, where $k_0 = k_{\ell} = 0$.
If $t_1\leq \dotsb \leq t_{\ell}$ and $u_1 \leq \dotsb \leq u_{\ell}$, then there exists $\bm \in \PP^{\ell}$ such that $\core_e(\bm;\bu) = \core_e(\bl;\bt)$ and $\mv_e(\bm;\bu) = \mv_e(\bl;\bt) + (k_1,\dotsc, k_{\ell-1},0)$.
\end{lem}

\begin{proof}

Note first that $\sum_{i=1}^{\ell} u_i = \sum_{i=1}^{\ell} t_i$.

We prove by induction on $K = \sum_{i=1}^{\ell-1} k_i$.
If $K = 0$, then $k_i = 0$ for all $i \in [1,\,\ell-1]$, so that $\bu = \bt$, and hence we can let $\bm = \bl$.
If $K > 0$, let $M = \max\{ i \in [1,\, \ell-1] \mid k_i > 0\}$.
Then $u_{M+1} = t_{M+1} - k_M < t_{M+1}$, and $u_i = t_i$ for all $i > M+1$.
Since $\sum_{i=1}^{\ell} u_i = \sum_{i=1}^{\ell} t_i$, there exists $j \in [1,\,\ell]$ such that $u_j > t_j$, and we let $m$ be the largest such index.
Then $m \leq M$,
so that
$t_m < u_m \leq u_{M+1} < t_{M+1}$, hence
$$t_m + 1 \leq u_m \leq u_{M+1} \leq t_{M+1}-1.$$
Let
\begin{align*}
a &:= \max\{ i \in [1,\ell] \mid t_i = t_m \}; \\
b &:= \min\{ i \in [1,\ell] \mid t_i = t_{M+1} \}.
\end{align*}
Then $m \leq a < b \leq M+1$, and $t_a < t_i < t_b$ for all $i \in [a+1,\,b-1]$.
Let $\beta_{\bt}(\bl) = (\B^{(1)},\dotsc, \B^{(\ell)})$.
Since $\fs(\B^{(a)}) = t_a = t_m < t_{M+1} = t_b = \fs(\B^{(b)})$,
there exists $x \in \B^{(b)}$ such that $x \notin \B^{(a)}$ by Lemma \ref{L:hub}(1).
For $i \in [1,\,\ell]$, let
$$
\C^{(i)} =
\begin{cases}
\B^{(a)} \cup \{x \}, &\text{if } i = a; \\
\B^{(b)} \setminus \{x \}, &\text{if } i = b; \\
\B^{(i)}, &\text{otherwise}.
\end{cases}
$$
Let $(\bnu; \bv) \in \PP^{\ell} \times \ZZ^{\ell}$ be such that $\beta_{\bv}(\bnu) = (\C^{(1)},\dotsc, \C^{(\ell)})$.
Then
\begin{align*}
\core_e(\bnu;\bv) &= \core_e(\bl;\bt),\qquad
\mv_e(\bnu;\bv) = \mv_e(\bl;\bt) + \smash{\sum_{i=a}^{b-1}} \Be_{i}.
\end{align*}
If $\bv = (v_1,\dotsc, v_{\ell})$, then $v_i = t_i + \delta_{ia} - \delta_{ib}$ for all $i \in [1,\ell]$, and so
$$
v_{i+1} - v_i = t_{i+1} - t_i + \delta_{i,a-1} - \delta_{ia} + \delta_{ib} - \delta_{i,b-1}
\geq 0
$$
for all $i \in [1,\,\ell-1]$,
since $t_{a+1} > t_a$ and $t_b > t_{b-1}$ and $t_b - t_a = t_M - t_m \geq 2$.
Moreover, if $i > a$, then $i > m$, so that $t_i \geq u_i$ by definition of $m$.
Since $u_i = t_i + k_i - k_{i-1}$, this yields $k_{i-1} \geq k_i$ when $i > a$.
In particular, we have
$$
k_a \geq k_{a+1} \geq \dotsb \geq k_{M} > 0.
$$
For each $i \in [1,\,\ell-1]$,
let $k'_i := k_i - \bbone_{i \in [a,\,b-1]}$.
Then $k'_i \in \ZZ_{\geq 0}$ and
\begin{align*}
u_i = t_i + k_i - k_{i-1}
= (v_i - \delta_{ia}+ \delta_{ib}) + (k'_i - k'_{i-1} + \delta_{ia} - \delta_{ib})
= v_i + k'_i - k'_{i-1}
\end{align*}
for all $i \in [1,\,\ell]$, where $k'_0 = k'_{\ell} = 0$.
Since $\sum_{i=1}^{\ell-1} k'_i = \sum_{i=1}^{\ell-1} k_i - (b-a) < \sum_{i=1}^{\ell-1} k_i$, by induction, there exists $\bm \in \PP^{\ell}$ such that
\begin{align*}
\core_e(\bm;\bu) &= \core_e(\bnu;\bv) = \core_e(\bl;\bt); \\
\mv_e(\bm;\bu) &= \mv_e(\bnu;\bv) + (k'_1,\dotsc, k'_{\ell-1},0)
= \mv_e(\bl;\bt) + (k_1,\dotsc, k_{\ell-1}, 0),
\end{align*}
as desired.
\end{proof}

\begin{lem} \label{L:mv}
Let $(\bl;\bt) \in \PP^{\ell} \times \ZZ^{\ell}$, and let $\rr_{\ell} = (1,2,\dotsc, \ell) \in \sym{\ell}$.
\begin{enumerate}
\item
Then
$$\mv_e((\bl;\bt)^{\rr_{\ell}\Be_{\ell}}) = (\mv_e(\bl;\bt))^{\rr_{\ell}}.$$

\item If $w, w' \in \EW_{\ell}$ such that $\bt^w, \bt^{w'} \in \AAbar$, then $$\mv_e((\bl;\bt)^{w'}) = (\mv_e((\bl;\bt)^{w}))^{\rr_{\ell}^k}$$ for some $k \in \ZZ/\ell\ZZ$.

\item If $w, w' \in \EW_{\ell}$ such that $\bt^w = \bt^{w'} \in \AAbar$, then
$$\mv_e((\bl;\bt)^w) = \mv_e((\bl;\bt)^{w'}).$$
\end{enumerate}
\end{lem}

\begin{proof} \hfill
\begin{enumerate}
\item
Let $\bij_{\ell,e}(\bl;\bt) = (\bm;\bu)$.
By Lemma \ref{L:ell-residue}(3), $\bij_{\ell,e}((\bl;\bt)^{\rr_{\ell}\Be_{\ell}}) = (\bm;\bu + \bone)$.
A node of $[\bm]$ has $(\ell;\bu)$-residue $i$ if and only if it has $(\ell; \bu+\bone)$-residue $i+1 \pmod{\ell}$.
Consequently, if $C_i(\bnu;\bv)$ denotes the number of nodes of $[\bnu]$ of $(\ell,\bv)$-residue $i$, then $C_i(\bm;\bu) = C_{i+1}(\bm;\bu+\bone)$.
Thus,
\begin{align*}
\mv_e(\bl;\bt) &= (C_{\ell-1}(\bij_{\ell,e}(\bl;\bt)), C_{\ell-2}(\bij_{\ell,e}(\bl;\bt)), \dotsc, C_{0}(\bij_{\ell,e}(\bl;\bt))) \\
&= (C_{\ell-1}(\bm;\bu), C_{\ell-2}(\bm;\bu), \dotsc, C_{0}(\bm;\bu)) \\
&= (C_{0}(\bm;\bu+\bone), C_{\ell-1}(\bm;\bu+\bone), \dotsc, C_{1}(\bm;\bu+\bone)) \\
&= (C_{\ell-1}(\bm;\bu+\bone), C_{\ell-2}(\bm;\bu+\bone), \dotsc, C_{0}(\bm;\bu+\bone))^{\rr_{\ell}^{-1}} \\
&= (C_{\ell-1}(\bij_{\ell,e}((\bl;\bt)^{\rr_{\ell}\Be_{\ell}})), C_{\ell-2}(\bij_{\ell,e}((\bl;\bt)^{\rr_{\ell}\Be_{\ell}})), \dotsc, C_{0}(\bij_{\ell,e}((\bl;\bt)^{\rr_{\ell}\Be_{\ell}})))^{\rr_{\ell}^{-1}} \\
&= (\mv_e((\bl;\bt)^{\rr_{\ell}\Be_{\ell}}))^{\rr_{\ell}^{-1}}
\end{align*}
as desired.

\item
By Lemma \ref{L:AA}(5), $\bt^{w'} = (\bt^w)^{(\rr_{\ell}\Be_{\ell})^k}$ for some unique $k \in \ZZ$.
Consequently,
$$\mv_e((\bl;\bt)^{w'}) =
\mv_e(((\bl;\bt)^w)^{(\rr_{\ell}\Be_{\ell})^k}) =
(\mv_e((\bl;\bt)^w))^{\rr_{\ell}^k}
$$
by part (1).  Since $\rr_{\ell}$ has order $\ell$, $\mv_e((\bl;\bt)^{w'}) = (\mv_e((\bl;\bt)^w))^{\rr_{\ell}^{k \textrm{ mod } \ell}}$ as desired.

\item
By Theorem \ref{T:core-n-weight},
\begin{align*}
 \core_e((\bl;\bt)^w) &= \core_e((\bl;\bt)^{w'}), \\
\wt_e((\bl;\bt)^w) &= \min\{ \wt_e((\bl;\bt)^{\EW_{\ell}}) \} = \wt_e((\bl;\bt)^{w'}).
\end{align*}
By Lemma \ref{L:ell-residue}, $\mv_e((\bl;\bt)^w)$ and $\mv_e((\bl;\bt)^{w'})$ are both completely determined by $\bt^w = \bt^{w'}$, $\core_e((\bl;\bt)^w) = \core_e((\bl;\bt)^{w'})$ and $\wt_e((\bl;\bt)^w) = \wt_e((\bl;\bt)^{w'})$.
Thus, $$\mv_e((\bl;\bt)^w) = \mv_e((\bl;\bt)^{w'}).$$
\end{enumerate}
\end{proof}

Lemmas \ref{L:mv-construction} and \ref{L:mv}(1) give us the following partial converse to Corollary \ref{C:t*}, which is a generalisation of \cite[Lemma 4.3.2 and Remark 4.3.3]{LQ-movingvector}.

\begin{cor} \label{C:mv-construction}
Let $(\bl^*;\bt^*) \in \PP^{\ell} \times \ZZ^{\ell}$ be such that $\UU(\beta_{\bt^*}(\bl^*))$ is an $e$-core.
Let $\mathbf{m} = (m_1,\dotsc, m_{\ell}) \in (\ZZ_{\geq 0})^{\ell}$, and suppose that $\bt = (t_1,\dotsc, t_{\ell}) \in \ZZ^{\ell}$ satisfies
$t_i = t^*_i + m_i - m_{i-1}$ for all $i\in [1,\, \ell]$, where $\bt^* = (t^*_1,\dotsc, t^*_{\ell})$ and $m_0 = m_{\ell}$.
If $t_1 \leq \dotsb \leq t_{\ell}$, then there exists $\bl = (\lambda^{(1)},\dotsc, \lambda^{(\ell)}) \in \PP^{\ell}$ such that
$\core_e(\bl;\bt) = \UU(\beta_{\bt^*}(\bl^*))$ and
$\mv_e(\bl;\bt) = \mathbf{m}$.  Furthermore, $\wt_e(\lambda^{(i)}) \geq \min\{m_1,\dotsc, m_{\ell}\}$ for some $i \in [1,\,\ell]$.
\end{cor}

\begin{proof}
Since $\UU(\beta_{\bt^*}(\bl^*))$ is an $e$-core, $\bt^* \in \AAbar$ by \cite[Proposition 2.14]{JL-cores}.
In particular, $t^*_1 \leq \dotsb \leq t^*_{\ell}$.
We note also that $|\bt| = |\bt^*|$.

We assume first that $m_{\ell} \leq m_i$ for all $i \in [1,\ell]$.
Let $\mathbf{m}' = (m'_1,\dotsc, m'_{\ell}) := \mathbf{m} - m_{\ell}\bone$.
Then $m'_{\ell} = 0$, and $m'_i \in \ZZ_{\geq 0}$ and
$t_i = t^*_i +m'_i - m'_{i-1}$ for all $i \in [1,\,\ell]$.
Thus, there exists $\bm = (\mu^{(1)},\dotsc, \mu^{(\ell)})\in \PP^{\ell}$ such that
\begin{alignat*}{2}
\core_e(\bm;\bt) &= \core_e(\bl^*;\bt^*) = \core_e(\UU(\beta_{\bt^*}(\bl^*))) =  \UU(\beta_{\bt^*}(\bl^*)), \\
\mv_e(\bm;\bt) &= \mv_e(\bl^*;\bt^*) + \mathbf{m}' = \mathbf{m}'
\end{alignat*}
by Lemma \ref{L:mv-construction}.
Let $\nu$ be the partition obtained from $\mu^{(\ell)}$ by adding $m_{\ell}e$ nodes to its first row.
Then $\wt_e(\nu) \geq m_{\ell}$.
Let $\bl = (\mu^{(1)}, \dotsc, \mu^{(\ell-1)}, \nu)$.
Then $\UU(\beta_{\bt}(\bl))$ can be obtained from $\UU(\beta_{\bt}(\bm))$ by replacing $\uu_{\ell}(x)$ with $\uu_{\ell}(x+ m_{\ell} e)  = \uu_{\ell}(x) + m_{\ell} e \ell$, where $x= \max(\beta_{t_{\ell}}(\mu^{(\ell)}))$.
Consequently, 
\begin{align*}
\core_e(\bl;\bt) &= \core_e(\bm;\bt) 
= \UU(\beta_{\bt^*}(\bl^*)); \\
\mv_e(\bl;\bt) &= \mv_e(\bm;\bt) + m_{\ell} \bone = \mathbf{m'} + m_{\ell}\bone = \mathbf{m}.
\end{align*}

In general, there exists $i \in [1,\ell]$ such that $m_i \leq m_j$ for all $j \in [1,\,\ell]$.
Then $(\bt^*)^{(\rr_{\ell}\Be_{\ell})^i} \in \AAbar$ by Lemma \ref{L:AA}(2).
Thus
$$\wt_e((\bl^*;\bt^*)^{(\rr_{\ell}\Be_{\ell})^i}) = \min(\wt_e((\bl^*;\bt^*)^{\EW_{\ell}})) = \wt_e(\bl^*;\bt^*) = 0$$
by Theorem \ref{T:core-n-weight}(2),
so that
$\UU(\beta_{(\bt^*)^{(\rr_{\ell}\Be_{\ell})^i}}( (\bl^*)^{\rr_{\ell}^i}))$ is an $e$-core.
Let
\begin{alignat*}{2}
\mathbf{m}' &= (m'_1,\dotsc, m'_{\ell}) &&:= \mathbf{m}^{\rr_{\ell}^i}; \\
\bu   &= (u_1,\dotsc, u_{\ell}) &&:= \bt^{(\rr_{\ell}\Be_{\ell})^i}; \\
\bu^* &= (u^*_1,\dotsc, u^*_{\ell}) &&:= (\bt^*)^{(\rr_{\ell}\Be_{\ell})^i}.
\end{alignat*}
Then $m'_j = m_{j+i \textrm{ mod } \ell}$,
$u_j = t_{j+i \textrm{ mod } \ell} + e\bbone_{j+i > \ell}$
and $u^*_j = t^*_{j+i \textrm{ mod } \ell} + e\bbone_{j+i > \ell}$ for all $j \in [1,\,\ell]$.
In particular, $m'_{\ell} = m_i \leq m'_j$  and
\begin{align*}
u_j &= t_{j+i \textrm{ mod } \ell} + e\bbone_{j+i > \ell} \\
&= t^*_{j+i \textrm{ mod } \ell} + m_{j+i \textrm{ mod } \ell} - m_{j+i-1 \textrm{ mod } \ell} + e\bbone_{j+i > \ell} \\
&= u^*_j + m'_{j} - m'_{j-1}
\end{align*}
for all $j \in [1,\,\ell]$.
Consequently, we can conclude from above that there exists $\bm = (\mu^{(1)},\dotsc, \mu^{(\ell)}) \in \PP^{\ell}$ such that $\core_e(\bm;\bu) = \UU(\beta_{\bu^*}((\bl^*)^{\rr_{\ell}^i}))$,
$\mv_e(\bm;\bu) = \mathbf{m}'$ and $\wt_e(\mu^{(\ell)}) \geq m'_{\ell} = m_i$.
Let $\bl = \bm^{\rr_{\ell}^{-i}}$.
Then
$$
\mv_e(\bl;\bt) = \mv_e((\bm;\bu)^{(\rr_{\ell}\Be_{\ell})^{-i}}) = (\mv_e(\bm;\bu))^{\rr_{\ell}^{-i}} = (\mathbf{m}')^{\rr_{\ell}^{-i}} = \mathbf{m},
$$
where the second equality follows from Lemma \ref{L:mv}(1).
Furthermore,
\begin{align*}
\beta^{-1}(\core(\bl;\bt)) &= \beta^{-1}(\core_e((\bm;\bu)^{(\rr_{\ell}\Be_{\ell})^{-i}})) = \beta^{-1}(\core_e(\bm;\bu)) = \beta^{-1}(\UU(\beta_{\bu^*}((\bl^*)^{\rr_{\ell}^i}))) \\
&= \beta^{-1}(\core_e((\bl^*;\bt^*)^{(\rr_{\ell}\Be_{\ell})^i}))
= \beta^{-1}(\core_e(\bl^*;\br^*))
= \beta^{-1}(\UU(\beta_{\bt^*}(\bl^*)))
\end{align*}
by Theorem \ref{T:core-n-weight}(1),
so that $\core(\bl;\bt) = \UU(\beta_{\bt^*}(\bl^*))$ since $\fs(\core(\bl;\bt)) = |\bt| = |\bt^*| = \fs(\UU(\beta_{\bt^*}(\bl^*)))$.
In addition, $\wt_e(\lambda^{(i)}) = \wt_e(\mu^{(\rr_{\ell}^{-i}(i))}) = \wt_e(\mu^{(\ell)}) \geq m_i$, where $\bl = (\lambda^{(1)},\dotsc, \lambda^{(\ell)})$.
\end{proof}

Recall the left action $\DDot{}$ of $\AW_e$ on $\PP^{\ell} \times \ZZ^{\ell}$ as described in Subsection \ref{SS:Weyl}.

\begin{lem} \label{L:cs_i-residue}
Let $(\bl; \bt) \in \PP^{\ell} \times \ZZ^{\ell}$.
\begin{enumerate}
\item  There exists $(\bm;\bt) \in \AW_e \DDot{} (\bl;\bt)$ such that $\hub_j(\bm;\bt) \leq 0$ for all $j \in \ZZ/e\ZZ$.  
\item
For each $j \in \ZZ/e\ZZ$,
$\mv_e(\cs_j \DDot{} (\bl;\bt)) = \mv_e(\bl;\bt)$.
\end{enumerate}
\end{lem}

\begin{proof} \hfill
\begin{enumerate}
\item Note first that $\AW_e \DDot{} (\bl;\bt) = (\AW_e\DDot{\bt} \bl; \bt)$.  Let $\bm$ be an $\ell$-partition in the orbit $\AW_e \DDot{\bt} \bl$ with the least size.  If $\hub_j(\bm;\bt) > 0$, then $\cs_j \DDot{\bt} \bm \in \AW_e \DDot{\bt} \bl$, and $\cs_j \DDot{\bt} \bm$ has size $|\bm| - \hub_j(\bm;\bt) < |\bm|$, contradicting our choice of $\bm$.

\item
Let $\bij_{\ell,e}(\bl;\bt) = (\bm;\bu)$ and $\bij_{\ell,e}(\cs_j \DDot{} (\bl;\bt)) = (\bnu;\bv)$.  Then
\begin{alignat*}{2}
\quot_e(\UU(\beta_{\bt}(\bl))) &= \beta_{\bu}(\bm)& & = (\beta_{u_1}(\mu^{(1)}),\dotsc, \beta_{u_e}(\mu^{(e)})), \\
\quot_e(\UU(\beta_{\bt}(\cs_i \DDot{\bt} \bl))) &= \beta_{\bv}(\bnu) && = (\beta_{v_1}(\nu^{(1)}),\dotsc, \beta_{v_e}(\nu^{(e)})).
\end{alignat*}
Since $\cs_j \DDot{} (\bl;\bt) = (\cs_j \DDot{\bt} \bl; \bt)$, and
$
\UU(\beta_{\bt}(\cs_j \DDot{\bt} \bl))=
\UU(\cs_j \DDot{1} \beta_{\bt}(\bl))=
\cs_j \DDot{\ell} \UU(\beta_{\bt}(\bl))
$ by \eqref{E:t-action} and \eqref{E:cs_i-U} respectively,
we have
$$
\beta_{\bv}(\bnu) = \quot_e(\UU(\beta_{\bt}(\cs_j \DDot{\bt} \bl))) = \quot_e(\cs_j \DDot{\ell} \UU(\beta_{\bt}(\bl)))
= \cs_j \CDot{\ell} \quot_e(\UU(\beta_{\bt}(\bl))) = \cs_j \CDot{\ell} \beta_{\bu}(\bm),
$$
by \eqref{E:CDot-tuple-beta-sets}, so that
$\beta_{v_{j'}}(\nu^{(j')}) = \beta_{u_{j'}}(\mu^{(j')})$ for all $j' \in [1,\, e] \setminus \{ j,j+1\}$, while
\begin{alignat*}{2}
\beta_{v_j}(\nu^{(j)}) &= (\beta_{u_{j+1}}(\mu^{(j+1)}))^{+(-\delta_{j0}\ell)} &&= \beta_{u_{j+1}-\delta_{je}\ell} (\mu^{(j+1)}),
\\
\beta_{v_{j+1}}(\nu^{(j+1)}) &= (\beta_{u_{j}}(\mu^{(j)}))^{+\delta_{j0}\ell} &&= \beta_{u_j+\delta_{je}\ell}(\mu^{(j)}),
\end{alignat*}
where $u_j$, $v_j$ and $\mu^{(j)}$ and $\nu^{(j)}$ are to be read as $u_e$, $v_e$ and $\mu^{(e)}$ and $\nu^{(e)}$ respectively when $j = 0$.
Consequently, $\res^{\bu}_{\ell}(\bm) = \res^{\bv}_{\ell}(\bnu)$, and hence $$\mv_e(\bij_{\ell,e}(\bl;\bt)) = \mv_e(\bm;\bu) = \mv_e(\bnu;\bv) = \mv_e(\bij_{\ell,e}(\cs_j \DDot{} (\bl;\bt))$$ as desired.
\end{enumerate}
\end{proof}

Recall the Ariki-Koike algebra $\HH_n = \HH_{\FF, q, \br}(n)$.

\begin{lem} \label{L:mv-block}
Let $w_0 \in \EW_{\ell}$ such that $\br^{w_0} \in \AAF \cap \br^{\EW_{\ell}}$, and let $\bl \in \PP^{\ell}$.
\begin{enumerate}
\item Then
$$\{ \mv_e((\bl;\br)^{w}) \mid w \in \EW_{\ell},\, \br^{w} \in \AAbar \} = \{ (\mv_e((\bl;\br)^{w_0}))^{\rr_{\ell}^i} \mid i \in \ZZ/\ell\ZZ \}.$$

\item If $\bl, \bm$ lie in the same block $B$ of $\HH_n$ and $\br' \in \br^{\EW_{\ell}} \cap \AAbar$, then $\mv_e((\bl;\br)^{w}) = \mv_e((\bm;\br)^{w'})$ for all $w,w'\in \EW_{\ell}$ such that $\br^w =\br^{w'} = \br'$.

\end{enumerate}
\end{lem}

\begin{proof}
Part (1) follows from Lemma \ref{L:mv}(2).

For part (2), if $\bl,\bm$ lie in the same block, then
\begin{alignat*}{3}
\core_e((\bl;\br)^w) &= \beta_{|\br^w|} (\core_{\HH}(\bl)) &&= \beta_{|\br^w|} (\core_{\HH}(\bm)) &&= \core_e((\bm;\br)^w) \\
\wt_e((\bl;\br)^w) &= \wt_{\HH}(\bl) &&= \wt_{\HH}(\bm) &&= \wt_e((\bm;\br)^w)
\end{alignat*}
by Theorem \ref{T:Naka},
so that $\mv_e((\bl;\br)^w) = \mv_e((\bm;\br)^w)$ by Lemma \ref{L:ell-residue}(2).
Consequently,
$$
\mv_e((\bl;\br)^w) = \mv_e((\bm;\br)^w) = \mv_e((\bm;\br)^{w'})$$
by Lemma \ref{L:mv}(3).
\end{proof}

\begin{Def}
Given a block of $B$ of $\HH_n$,
we shall denote the common moving vector $\mv_e((\bl;\br)^w)$ for all $\bl \in \PP^{\ell}$ lying in $B$ and all $w\in \EW^{\ell}$ with $\br^w = \br'$ in Lemma \ref{L:mv-block}(2) as
$$\mv^{\br'}(B) = (\mvi^{\br'}_1(B), \dotsc, \mvi^{\br'}_{\ell}(B) ).$$
\end{Def}

\begin{cor}
Let $B$ be a block of $\HH_n$, and let $\br^{\HH} \in \AAF \cap \br^{\EW_{\ell}}$.  Then
$$
\{ \mv^{\br'}(B) \mid \br' \in \br^{\EW^{\ell}} \cap \AAbar \} = \{ (\mv^{\br^{\HH}}(B))^{\rr_{\ell}^i} \mid i \in \ZZ/\ell\ZZ \}.
$$
\end{cor}



\begin{prop} \label{P:moving-vector-Weyl-orbit}
Let $B$ and $C$ be two blocks of Ariki-Koike algebras with the same $\ell$-charge $\br$.
Let $\hat{w} \in \EW_{\ell}$ such that $\br^{\hat{w}} \in \AAbar$.
Then $\mv^{\br^{\hat{w}}}(B) = \mv^{\br^{\hat{w}}}(C)$ if and only if $B = w \DDot{} C$ for some $w \in \AW_e$.
\end{prop}

\begin{proof}
If $B = w \DDot{} C$ where $w \in \AW_e$, then for $\bl \in \PP^{\ell}$ lying in $B$, we have $w \DDot{\br} \bl$ lies in $C$, and so
$$
\mv^{\br^{\hat{w}}}(C) = \mv_e( (w\DDot{\br} \bl;\br)^{\hat{w}})
= \mv_e( w \DDot{} ((\bl;\br)^{\hat{w}}))
= \mv_e((\bl;\br)^{\hat{w}})
= \mv^{\br^{\hat{w}}}(B)
$$
by Lemma \ref{L:cs_i-residue}(2).

Conversely, suppose that $\mv^{\br^{\hat{w}}}(B) = \mv^{\br^{\hat{w}}}(C)$.
By Lemma \ref{L:cs_i-residue}(1), we can find $B_0 \in \AW_e \DDot{} B$ and $C_0 \in \AW_e \DDot{} C$ such that $\hub_j(B_0), \hub_j(C_0) \leq 0$ for all $j \in \ZZ/e\ZZ$.

Let $\bl \in \PP^{\ell}$ be lying in $B_0$.
Then
$ \hub_j((\bl;\br)^{\hat{w}}) = \hub^{\HH}_j(\bl) = \hub_j(B_0) \leq 0$
for all $j\in \ZZ/e\ZZ$.
Consequently $
\core_e((\bl;\br)^{\hat{w}}) = \UU(\beta_{\bt^*}(\EP))$ for some unique $\bt^* \in \AAbar$ by Lemma \ref{L:smallest}.
By Corollary \ref{C:t*}, $\bt^*$ is completely determined by $\br^{\hat{w}}$ and $\mv_e((\bl;\br)^{\hat{w}}) = \mv^{\br^{\hat{w}}}(B_0) = \mv^{\br^{\hat{w}}}(B)$.  Consequently $\core_e(B_0) = \beta^{-1}(\core_e((\bl;\br)^{\hat{w}})) = \beta^{-1}(\UU(\beta_{\bt^*}(\EP)))$ is completely determined by $\br^{\hat{w}}$ and $\mv^{\br^{\hat{w}}}(B)$.

Similar statements hold for $C_0$, so that $\core_e(B_0) = \core_e(C_0)$ since $\mv^{\br^{\hat{w}}}(B) = \mv^{\br^{\hat{w}}}(C)$.
Since
$\wt_e(B_0) = |\mv^{\br^{\hat{w}}}(B_0)| = |\mv^{\br^{\hat{w}}}(B)|
= |\mv^{\br^{\hat{w}}}(C)| = |\mv^{\br^{\hat{w}}}(C_0)| = \wt_e(C_0)$, we see that $B_0$ and $C_0$ have the same core and the same weight, and hence $B_0 = C_0$.
Consequently, $B$ and $C$ lie in the same $\AW_e$-orbit.
\end{proof}

Recall the function $\bfs$ in Definition \ref{D:bfs}.

\begin{Def} \label{D:r*}
Let $B$ be a block of $\HH_n$, and let $w \in \EW_{\ell}$ such that $\br^w \in
\AAbar$.  Define
$$
\br^{w*}_B := \bfs(\quot_e(\beta_{|\br^w|}(\core(B)))).
$$
\end{Def}

\begin{lem} \label{L:bij-*}
Let $\bl \in \PP^{\ell}$ be lying in a block $B$ of $\HH_n$, and let $w \in \EW_{\ell}$ such that $\br^w \in \AAbar$.
\begin{enumerate}
\item We have $\bij_{\ell,e}((\bl;\br)^{w}) = (\bl^*; \br^{w*}_B)$ for some $\bl^* \in \PP^e$.
\item Let $\bl^* \in \PP^e$ be that satisfying part (1).  If $j \in \ZZ/e\ZZ$, then
$$
\bij_{\ell,e}((\cs_j \DDot{\br} \bl;\br)^{w}) = ((\bl^*)^{\overline{\cs_j}}; \cs_j \CDot{\ell} \br^{w*}_B).
$$
(Recall that $\overline{\cs_j}$ denotes the projection of $\cs_j$ onto $\sym{e}$.)

In particular, $\br^{w*}_{\cs_j \DDot{} B} = \cs_j \CDot{\ell} \br^{w*}_B$.
\end{enumerate}
\end{lem}

\begin{proof}
\hfill
\begin{enumerate}
\item
We have $\core_e((\bl;\br)^{w}) = \beta_{|\br^w|}(\core(B))$, so that
$$\quot_e(\core_e((\bl;\br)^{w})) = \quot_e(\beta_{|\br^w|}(\core(B))) = \beta_{\br^{w*}_B}(\EP).$$
Now apply Lemma \ref{L:ell-residue}(1).

\item
By part (1), $\quot_e(\UU(\beta_{\br^{w}}(\bl^{w}))) = \beta_{\br^{w*}_B}(\bl^*)$.
Furthermore,
\begin{align*}
\cs_j \DDot{\ell} \UU(\beta_{\br^{w}}(\bl^{w}))
= \UU(\cs_j \DDot{1} \beta_{\br^{w}}(\bl^{w}))
= \UU(\beta_{\br^{w}}(\cs_j\DDot{\br^{w}} \bl^{w}))
= \UU(\beta_{\br^{w}}((\cs_j\DDot{\br} \bl)^{w}))
\end{align*}
by \eqref{E:cs_i-U}, \eqref{E:t-action} and Lemma \ref{L:Weyl action}(2) respectively.
Consequently,
\begin{align*}
\quot_e(\UU(\beta_{\br^{w}}((\cs_j\DDot{\br} \bl)^{w})))
&= \quot_e(\cs_j \DDot{\ell} \UU(\beta_{\br^{w}}(\bl^{w})))
= \cs_j \CDot{\ell} \quot_e(\UU(\beta_{\br^{w}}(\bl^{w}))) \\
&= \cs_j \CDot{\ell} \beta_{\br^{w*}_B}(\bl^*)
=
\beta_{\cs_j \smash{\CDot{\ell}} \br^{w*}_B}((\bl^*)^{\overline{\cs_j}})
\end{align*}
by \eqref{E:CDot-tuple-beta-sets} and Lemma \ref{L:Weyl action}(1).  Part (2) thus follows.

The last assertion now follows from part (1), since $\cs_j \DDot{\br} \bl$ lies in $\cs_j \DDot{} B$.
\end{enumerate}
\end{proof}

\subsection{Core blocks}

The notion of core blocks was first introduced by Fayers in \cite{Fayers-Coreblock}.
These are blocks where all multipartitions lying in it are multicores, and are believed to be the most elementary non-simple blocks of Ariki-Koike algebras.

We first note the following lemma:

\begin{lem} \label{L:core-blocks-invariant}
The set of core blocks of Ariki-Koike algebras (associated with the same $\ell$-charge) is invariant under the left action of $\AW_e$.
\end{lem}

\begin{proof}
Let $\bl = (\lambda^{(1)},\dotsc, \lambda^{(\ell)}) \in \PP^{\ell}$.
Then $\cs_j \DDot{\br} \bl = (\cs_j \DDot{r_1} \lambda^{(1)},\dotsc, \cs_j \DDot{r_{\ell}} \lambda^{(\ell)})$ for each $j \in \ZZ/e\ZZ$,
where $\cs_j \DDot{r_i} \lambda^{(i)}$ is the partition obtained from $\lambda^{(i)}$ by adding all its addable nodes of $(e,r_i)$-residue $j$ and removing all its removable nodes of $(e,r_i)$-residue $j$.
It is well known that adding all addable nodes of $(e,t)$-residue $j$ and removing all removable nodes of $(e,t)$-residue $j$ from a partition preserves the $e$-weight of the partition, for all $j \in [1,\,e]$ and $t\in \ZZ$. (Alternatively, one can also use Lemma \ref{L:cs_i-residue}(2).)
Thus $\bl$ is a multicore if and only if $\cs_j \DDot{\br} \bl$ is.
\end{proof}

\begin{thm} \label{T:moving-vector-of-core-block}
Let $B$ be a block of $\HH_n$, and let $w \in \EW_{\ell}$ such that $\br^w \in \AAbar$.  Then $B$ is a core block if and only if $\mvi^{\br^w}_i(B) = 0$ for some $i \in [1,\, \ell]$.
\end{thm}

\begin{proof}
Let $w = \bu \sigma$ where $\bu = (u_1,\dotsc, u_{\ell}) \in \ZZ^{\ell}$ and $\sigma \in \sym{\ell}$.
Then for any $\bl = (\lambda^{(1)},\dotsc, \lambda^{(\ell)}) \in \PP^{\ell}$, we have $(\bl;\br)^{w} = ((\lambda^{(\sigma(1))}, \dotsc, \lambda^{(\sigma(\ell))});
(r_{\sigma(1)} + eu_{\sigma(1)}, \dotsc, r_{\sigma(\ell)} + eu_{\sigma(\ell)}))$.

Suppose that $\bl$ lies in $B$ with $\wt_e(\beta_{r_i}(\lambda^{(i)})) > 0$ for some $i$.
Let $\bl^w = \bm = (\mu^{(1)},\dotsc, \mu^{(\ell)})$ and let $\br^w = (r'_1,\dotsc, r'_{\ell})$.
Then
$\lambda^{(i)} = \mu^{(a)}$ and $r_i + u_i e = r'_a$, where $a = \sigma^{-1}(i)$, and
so
$$\wt_e( \beta_{r'_a}(\mu^{(a)})) =
\wt_e(\beta_{r_i +  u_i e}(\lambda^{(i)})) =
\wt_e((\beta_{r_i}(\lambda^{(i)}))^{+u_i e}) =
\wt_e(\beta_{r_i}(\lambda^{(i)})) > 0.$$
Thus
there exists an $x \in \beta_{r'_a}(\mu^{(a)})$ such that $x-e \notin \beta_{r'_a}(\mu^{(a)})$.
Consequently, $\uu_{a}(x) \in \UU(\beta_{\br^w}(\bm))$ and $\uu_{a}(x)- e\ell = \uu_{a}(x-e) \notin \UU(\beta_{\br^w}(\bm))$.
Hence all components of $\mv_e(\bm;\br^w) = \mv_e((\bl;\br)^{w}) = \mv^{\br^w}(B)$ would be positive.

Conversely, suppose that $\mvi^{\br^w}_i(B) >0$ for all $i \in [1,\,\ell]$.
Let $\bl \in \PP^{\ell}$ be lying in $B$, and let $\core_e((\bl;\br)^{\br^w}) = \UU(\beta_{\bt^*}(\bl^*))$.
Since $\mv_e((\bl;\bt)^{\br^w}) = \mv^{\br^w} (B)$,
if $\br^w = (r'_1,\dotsc, r'_{\ell})$ and $\bt^* = (t^*_1, \dotsc, t^*_{\ell})$
then
$ r'_i = t^*_i + \mvi^{\br^w}_i(B) - \mvi^{\br^w}_{i-1}(B)$ for all $i \in [1,\,\ell]$ by Corollary \ref{C:t*}, where $\mvi^{\br^w}_{0}(B) := \mvi^{\br^w}_{\ell}(B)$.
By Corollary \ref{C:mv-construction},
there exists $\bm = (\mu^{(1)},\dotsc, \mu^{(\ell)}) \in \PP^\ell$ such that $\core_e(\bm;\br^w) = \UU(\beta_{\bt^*}(\bl^*))$, $\mv_e(\bm;\br^w) = \mv^{\br^w}(B)$ and $\wt_e(\mu^{(i)}) \geq \min \{ \mvi^{\br^w}_{i'}(B) \mid i' \in [1,\,\ell] \} > 0$ for some $i \in [1,\ell]$.
Let $\bnu := \bm^{{w}^{-1}}$.
Then
\begin{gather*}
\core_e((\bnu;\br)^w) = \core_e(\bm;\br^w) = \UU(\beta_{\bt^*}(\bl^*)) = \core_e((\bl;\br)^w); \\
\mv_e((\bnu;\br)^w) = \mv_e(\bm;\br^w) = \mv^{\br^w}(B) = \mv_e((\bl;\br)^w).
\end{gather*}
In particular $\wt_e((\bnu;\br)^w) = |\mv_e((\bnu;\br)^w)| = |\mv_e((\bl;\br)^w)| = \wt_e((\bl;\br)^w)$.
Thus $\core_{\HH}(\bnu) = \core_{\HH}(\bl)$ and $\wt_{\HH}(\bnu) = \wt_{\HH}(\bl)$, so that $\bnu$ and $\bl$ lie in the same block, namely $B$.
If $\bnu = (\nu^{(1)},\dotsc, \nu^{(\ell)})$, then $\wt_e(\nu^{(\overline{w}(i))}) = \wt_e(\mu^{(i)}) > 0$, so that $B$ is not a core block.
\end{proof}

We end this subsection with some necessary and sufficient conditions, in terms of moving vectors, for a block to be a core block and for a partition to lie in a core block, which follow immediately from Theorem \ref{T:moving-vector-of-core-block} and Lemma \ref{L:mv}.

\begin{cor} \label{C:mv} \hfill
\begin{enumerate}
\item
Let $B$ be a block of $\HH_n$.
The following statements are equivalent:
\begin{enumerate}
\item $B$ is a core block.
\item For all $\br' \in \br^{\EW_{\ell}} \cap \AAbar$, $\mvi_i^{\br'}(B) = 0$ for some $i \in [1,\,\ell]$.
\item For all $i \in [1,\,\ell]$, there exists $\br' \in \br^{\EW_{\ell}} \cap \AAbar$ such that $\mvi_i^{\br'}(B) = 0$.
\item There exists $\br' \in \br^{\EW_{\ell}} \cap \AAbar$ such that $\mvi^{\br'}_{\ell}(B) = 0$.
\end{enumerate}

\item
Let $\bl \in \PP^{\ell}(n)$.
The following statements are equivalent:
\begin{enumerate}
\item $\bl$ lies in a core block.
\item For all $w \in \EW_{\ell}$ such that $\br^w \in \AAbar$, $\mv_e((\bl;\br)^w)$ has a zero component.
\item For all $i \in [1,\,\ell]$, there exists $w \in \EW_{\ell}$ such that $\br^w \in \AAbar$ and the $i$-th component of $\mv_e((\bl;\br)^w)$ equals zero.
\item There exists $w \in \EW_{\ell}$ such that $\br^w \in \AAbar$ and the last component of $\mv_e((\bl;\br)^w)$ equals zero.
\end{enumerate}
\end{enumerate}
\end{cor}

\subsection{Weight graphs of multipartitions lying in a core block}

In \cite[para.\ after Proposition 2.6]{LyleRuff-Decompositionnumber}, Lyle and Ruff define the {\em weight graph} $G(\bl)$ for $\bl = (\lambda^{(1)},\dotsc, \lambda^{(\ell)})$ lying in a core block of $\HH_n$ as follows:
$G(\bl)$ has vertex set $\{1,\dotsc, \ell\}$ and no loops, and the number of (undirected) edges between $i$ and $j$ ($i \ne j$) equals the $\HH^{ij}$-weight $\wt_{\HH^{ij}}(\lambda^{(i)}, \lambda^{(j)})$ of the bipartition $(\lambda^{(i)}, \lambda^{(j)})$ in the Ariki-Koike algebra $\HH^{ij} = \HH_{\FF,q,(r_i,r_j)}(|\lambda^{(i)}| + |\lambda^{(j)}|)$.
It follows from \cite[Subsection 3.5]{Fayers-Coreblock} that $\bl$ is `indecomposable' if and only if $G(\bl)$ is connected, and that $\bl$ can be decomposed into constituents naturally associated with the connected components of $G(\bl)$.
Furthermore, $G(\bl)$ and $G(\bm)$ have the same connected components when $\bl$ and $\bm$ lie in the same block.

In this subsection, we define a graph $\Gamma_{\br'}(B)$ that is naturally associated to the moving vector $\mv^{\br'}(B)$ of a core block $B$ of $\HH_n$ (where $\br' \in \br^{\EW_{\ell}} \cap \AAbar$), and show that there is a natural bijective correspondence between its connected components and those of $G(\bl)$ for any $\bl \in \PP^{\ell}$ lying in $B$.

\begin{Def}
Let $B$ be a core block of $\HH_n$ and let $\br' \in \br^{\EW_{\ell}} \cap \AAbar$.  Define a undirected graph $\Gamma_{\br'}(B)$ as follows:  Its vertex set $V(\Gamma_{\br'}(B)) = \{1,\dotsc, \ell\}$ and there is an edge between $i$ and $j$ if and only if $j \equiv_{\ell} i + 1$ and $\mvi^{\br'}_i(B) \ne 0$.
\end{Def}

We begin with a lemma to understand the situation when there is an edge between two vertices in $G(\bl)$.

\begin{lem} \label{L:weight-bipartition}
Let $(\B^{(1)},\dotsc, \B^{(\ell)}) \in \BB^{\ell}$.
Let $i \in [1,\,\ell]$ and $k \in \ZZ^+$, and write $k = a\ell + b$, where $a \in \ZZ$ and $b \in [1-i,\,\ell-i]$.
Let $y \in \ZZ$, and recall the maps $\uu_{e,\ell,i} = \uu_i$ and $\UU_{e,\ell} = \UU$ in Subsection \ref{SS:Uglov}.
\begin{enumerate}
\item We have
\begin{align*}
\uu_{e,\ell,i}(y) - ke &= \uu_{e,\ell,i+b}(y-ae), \\
\uu_{e,2,1}(y) - (2a+1)e &= \uu_{e,2,2}(y-ae), \\
\uu_{e,2,2}(y) - (2a-1)e &= \uu_{e,2,1}(y-ae).
\end{align*}

\item The following statements are equivalent:
\begin{enumerate}
\item $\uu_{e,\ell,i}(y) - ke \notin \UU_{e,\ell}(\B^{(1)},\dotsc, \B^{(\ell)})$.
\item $y-ae \notin \B_{i+b}$.
\item $\uu_{e,2,1}(y) - (2a+1)e \notin \UU_{e,2}(\B^{(i)}, \B^{(i+b)})$.
\item $\uu_{e,2,2}(y) - (2a-1)e \notin \UU_{e,2}(\B^{(i+b)}, \B^{(i)})$.
\end{enumerate}

\item If $y \in \B^{(i)}$ and any of the statements in part (2) holds, then
$$
\wt_e(\UU_{e,2}(\B^{(\min\{i,i+b\})}, \B^{(\max\{i,i+b\})})) > 0.
$$
\end{enumerate}
\end{lem}

\begin{proof}
Part (1) follows immediately from the definition of $\uu_{e,\ell,i}$ and Lemma \ref{L:Uglov map}(4), and part (2) follows from part (1) and the definition of $\UU_{e,2}$ and $\UU_{e,\ell}$.

For part (3), if $i < i+b$, then $b >0$, and so $a \geq 0$ since $k = a\ell + b > 0$.  Thus
$$
\UU_{e,2}(\B^{(i)},\B^{(i+b)}) \ni \uu_{e,2,1}(y) > \uu_{e,2,1}(y) - (2a+1)e \notin \UU_{e,2}(\B^{(i)},\B^{(i+b)})
$$
and hence $\wt_e(\UU_{e,2}(\B^{(i)},\B_{(i+b)})) > 0$.
On the other hand, if $i \leq i+b$, then $b \leq 0$, so that $a>0$.
Thus
$$\UU_{e,2}(\B^{(i+b)},\B^{(i)}) \ni \uu_{e,2,2}(y) > \uu_{e,2,2}(y) - (2a-1)e \notin \UU_{e,2}(\B^{(i+b)}, \B^{(i)}),$$
and hence $\wt_e(\UU_{e,2}(\B^{(i+b)},\B^{(i)})) > 0$.
\end{proof}


\begin{thm} \label{T:weight-graph}
Let $B$ be a core block of $\HH_n$, and let $w\in \EW_{\ell}$ such that $\br^w \in \AAbar$. Let $\bl \in \PP^{\ell}$ be lying in $B$, and let $i, j \in \{1,\dotsc, \ell\}$ with $i \ne j$.
Then $i$ and $j$ are connected in $\Gamma_{\br^w}(B)$ if and only if $\overline{w}(i)$ and $\overline{w}(j)$ are connected in $G(\bl)$.
(Recall that $\overline{w}$ denotes the projection of $w$ onto $\sym{\ell}$.)
\end{thm}

\begin{proof}
Let $\beta_{\br^{w}}(\bl^{w}) = (\B^{(1)},\dotsc, \B^{(\ell)})$.
Let $w = \bt \sigma$ where $\bt = (t_1,\dotsc, t_{\ell}) \in \ZZ^{\ell}$ and $\sigma \in \sym{\ell}$.
Then $\overline{w} = \sigma$ and
$\B^{(a)} = \beta_{r_{\sigma(a)} + et_{\sigma(a)}}(\lambda^{(\sigma(a))})$
for all $a \in [1,\,\ell]$.

Without loss of generality, we assume that $i < j$.

Let
\begin{align*}
\br^{ij} = (r_{\sigma(i)},r_{\sigma(j)}), \qquad
\bt^{ij} = (t_{\sigma(i)},t_{\sigma(j)}), \qquad
\bl^{ij} &= (\lambda^{(\sigma(i))},\lambda^{(\sigma(j))}).
\end{align*}
Then
$
\beta_{(\br^{ij})^{\bt^{ij}}}(\bl^{ij}) = (\B^{(i)}, \B^{(j)}).
$
Since $\bfs(\B^{(1)},\dotsc, \B^{(\ell)}) = \bfs(\beta_{\br^w}(\bl^w)) = \br^w \in \AAbar$, we see that $(\br^{ij})^{\bt^{ij}} = \bfs(\B^{(i)},\B^{(j)}) \in \overline{\mathcal{A}}_e^2$.
Thus,
\begin{align}
\wt_e(\UU_{e,2}(\B^{(i)},\B^{(j)}))
= \wt_e(\UU_{e,2}(\beta_{(\br^{ij})^{\bt^{ij}}}(\bl^{ij})))
= \wt_e((\bl^{ij}; \br^{ij})^{\bt^{ij}})
= \wt_{\HH^{\sigma(i),\sigma(j)}}(\bl^{ij}). \label{E:wt-bipartition}
\end{align}
Since $(\bl^{ji};\br^{ji}) = (\bl^{ij};\br^{ij})^{\cs_1} \in (\bl^{ij};\br^{ij})^{\EW_2}$, we have
\begin{equation}
  \wt_{\HH^{\sigma(j),\sigma(i)}}(\bl^{ji}) = \min(\wt_e((\bl^{ji};\br^{ji})^{\EW_2})) = \min(\wt_e((\bl^{ij};\br^{ij})^{\EW_2})) = \wt_{\HH^{\sigma(i),\sigma(j)}}(\bl^{ij}). \label{E:swap-bipartition}
\end{equation}

Suppose that $\wt_{\HH^{\sigma(i),\sigma(j)}}(\lambda^{(\sigma(i))},\lambda^{(\sigma(j))})
> 0$.
Then $\wt_e(\UU_{e,2}(\B^{(i)},\B^{(j)})) > 0$.
Thus there exists $x \in \UU_{e,2}(\B^{(i)},\B^{(j)}) = \uu_{e,2,1}(\B^{(i)}) \cup \uu_{e,2,2}(\B^{(j)})$ such that $x - e \notin \UU_{e,2}(\B^{(i)},\B^{(j)})$.

If $x \in \uu_{e,2,1}(\B^{(i)})$, say $x = \uu_{e,2,1}(y)$ where $y \in \B^{(i)}$,
then $\uu_{e,2,1}(y)-e = x - e \notin \UU_{e,2}(\B^{(i)},\B^{(j)})$, so that
$\uu_{e,\ell, i}(y) \in \uu_{e,\ell,i}(\B^{(i)}) \subseteq \UU_{e,\ell}(\B^{(1)},\dotsc, \B^{(\ell)})$ and $\uu_{e,\ell, i}(y) - (j-i)e \notin \UU_{e,\ell}(\B^{(1)},\dotsc, \B^{(\ell)})$ by Lemma \ref{L:weight-bipartition} (with $k = j-i$, $a = 0$ and $b = j-i$).
Consequently, $\mvi^{\br^w}_i(B), \mvi^{\br^w}_{i+1}(B),\dotsc, \mvi^{\br^w}_{j-1}(B) >0$ so that $i$ is connected to $j$ in $\Gamma_{\br^w}(B)$.

On the other hand, if $x \in \uu_{e,2,2}(\B^{(j)})$, say $x = \uu_{e,2,2}(y)$ where $y \in \B^{(j)}$,
then $\uu_{e,2,2}(y)-e = x - e \notin \UU_{e,2}(\B^{(i)},\B^{(j)})$, so that
$\uu_{e,\ell, j}(y) \in \uu_{e,\ell,j}(\B^{(j)}) \subseteq \UU_{e,\ell}(\B^{(1)},\dotsc, \B^{(\ell)})$ and $\uu_{e,\ell, j}(y) - (\ell-j+i)e \notin \UU_{e,\ell}(\B^{(1)},\dotsc, \B^{(\ell)})$ by Lemma \ref{L:weight-bipartition} (with $k = \ell - j+i$, $a = 1$ and $b = i-j$).
Consequently,
$$\mvi^{\br^w}_j(B), \mvi^{\br^w}_{j+1}(B), \dotsc, \mvi^{\br^w}_{\ell}(B), \mvi^{\br^w}_1(B),\dotsc, \mvi^{\br^w}_{i-1}(B) >0$$ so that $j$ is connected to $i$ in $\Gamma_{\br^w}(B)$.

Now suppose that $\mvi_i^{\br^w}(B) > 0$, and let $i^+ \in [1,\,\ell]$ be such that $i^+ \equiv_{\ell} i+1$.
Then the $i$-th component of $\mv_e((\bl;\br)^w)$ is positive.
Thus there exists $x \in \ZZ$ such that
\begin{align*}
C &:= \{ c \in \ZZ_{\geq 0} \mid \uu_{e,\ell,i}(x) + ce \in \UU_{e,\ell}(\B^{(1)},\dotsc, \B^{(\ell)}) \}, \\
D &:= \{ d \in \ZZ^+ \mid \uu_{e,\ell,i}(x) - de \notin \UU_{e,\ell}(\B^{(1)},\dotsc, \B^{(\ell)}) \}
\end{align*}
are both nonempty.
Let $c_{\min} := \min(C)$ and $d_{\min} := \min(D)$,
and let $y_C,y_D \in \ZZ$ and $i_C,i_D \in [1,\,\ell]$ be such that
\begin{align*}
\uu_{e,\ell,i_C}(y_C) &= \uu_{e,\ell,i}(x) + c_{\min}e, \\
\uu_{e,\ell,i_D}(y_D) &= \uu_{e,\ell,i}(x) - d_{\min}e.
\end{align*}
Then $\uu_{e,\ell,i_C}(y_C) \in \UU_{e,\ell}(\B^{(1)},\dotsc,\B^{(\ell)})$, so that $y_C \in \B_{i_C}$,
while $\uu_{e,\ell,i_D}(y_D) \notin \UU_{e,\ell}(\B^{(1)},\dotsc,\B^{(\ell)})$.
Consequently, $\wt_{\HH^{\sigma(i_C),\sigma(i_D)}} (\bl^{i_C,i_D}) > 0$ by Lemma \ref{L:weight-bipartition}(3), \eqref{E:wt-bipartition} and \eqref{E:swap-bipartition}, so that $\sigma(i_C)$ and $\sigma(i_D)$ are connected in $G(\bl)$.

If $c_{\min} = 0$, then $i = i_C$ and trivially $\sigma(i)$ is connected to $\sigma(i_C)$ in $G(\bl)$.
On the other hand, if $c_{\min} > 0$, then $\uu_{e,\ell,i}(x) \notin \UU_{e,\ell}(\B^{(1)},\dotsc, \B^{(\ell)})$.  Since $\uu_{e,\ell,i_C}(y_C) \in \UU_{e,\ell}(\B^{(1)},\dotsc, \B^{(\ell)})$, by the same argument used in the last paragraph, $\sigma(i)$ and $\sigma(i_C)$ are connected in $G(\bl)$.

If $d_{\min} = 1$, then $i_D = i^+$.  Thus, $\sigma(i^+)$ is connected to $\sigma(i_D)$ trivially.
On the other hand, if $d_{\min} > 1$, then $\uu_{e,\ell,i^+}(x) = \uu_{e,\ell,i}(x) - e \in \UU_{e,\ell}(\B^{(1)},\dotsc, \B^{(\ell)})$.  Since $\uu_{e,\ell,i_D}(y_D) \notin \UU_{e,\ell}(\B^{(1)},\dotsc, \B^{(\ell)})$, we see that $\sigma(i^+)$ and and $\sigma(i_D)$ are connected in $G(\bl)$ by the same argument.

Thus, $\sigma(i)$, $\sigma(i_C)$, $\sigma(i_D)$ and $\sigma(i^+)$ all lie in the same connected component of $G(\bl)$, so that $\sigma(i)$ and $\sigma(i^+)$ are connected in $G(\bl)$.
\end{proof}

Since $\Gamma_{\br^w}(B)$ is disconnected if and only if there exist distinct $i,j \in [1,\,\ell]$ such that $\mvi^{\br^w}_i(B) = \mvi^{\br^w}_j(B) = 0$, we get the following immediate corollary, where the equivalence of the second and third statements follows from Lemma \ref{L:mv-block}(1).

\begin{cor} \label{C:decomposable}
Let $\bl \in \PP^{\ell}$.  For each $w \in\EW_{\ell}$ with $\br^w \in \AAbar$, let $\mv_e((\bl;\br)^w) = (m^w_1,\dotsc, m^w_{\ell})$.
The following statements are equivalent:
\begin{enumerate}
\item $\bl$ is decomposable (in the sense of Fayers).
\item There exist $w \in\EW_{\ell}$ with $\br^w \in \AAbar$ and distinct $i,j \in [1,\,\ell]$ such that $m^w_i = m^w_{j} = 0$.
\item For all $w \in\EW_{\ell}$ with $\br^w \in \AAbar$, there exist distinct $i,j \in [1,\,\ell]$ such that $m^w_i = m^w_{j} = 0$.
\end{enumerate}
\end{cor}

\section{Scopes equivalence between core blocks} \label{S:Scopes}

Let $B$ be a block of $\HH_n$ and $j \in \ZZ/e\ZZ$.
Following \cite{Lyle-RoCK, Webster}, we say that $B$ and $\cs_j \DDot{} B$ are {\em Scopes
equivalent} if all $\ell$-partitions lying in $B$ have no addable node with $(e,\br)$-residue $j$ or all $\ell$-partitions lying in $B$ have no removable node with $(e,\br)$-residue $j$. Clearly, Scopes equivalence is symmetric by definition, and we further extend it to
an equivalence relation on all blocks by taking its reflexive and transitive closure.

In this section, we obtain a necessary and sufficient condition for two core blocks to be Scopes equivalent.

The importance of Scopes equivalence is due to the following theorem which is well known among experts.

\begin{thm}[see {\cite[Theorem 6.4]{CR}} and {\cite[Lemma 3.1]{Webster}}] \label{T:CR}
Let $B$ be a block of $\HH_n$ and $j \in \Z/e\Z$.  Suppose that for all $\ell$-partitions $\bl$ lying in $B$, $\bl$ has no addable node with $(e,\br)$-residue $j$.  Then $B$ and $\cs_j \DDot{} B$ are Morita equivalent. Under this equivalence, the Specht module $S^{\bl}$ lying in $B$ corresponds to $S^{\cs_j \DDot{\br} \bl}$.
\end{thm}

Recall $\br_B^{w*}$ for a block $B$ and $w \in \EW_{\ell}$ with $\br^w \in \AAbar$ in Definition \ref{D:r*}.

\begin{lem} \label{L:core-block-abacus}
Let $w \in \EW_{\ell}$ such that $\br^w \in \AAbar$.
Let $B$ be a block of $\HH_n$ with $\br_B^{w*} = (x^B_1, \dotsc, x^B_{e})$.
If $\mvi^{\br^w}_i(B) = 0$, then
for any $\bl \in \PP^{\ell}$ lying in $B$ with $\quot_e(\UU(\beta_{\br^{w}}(\bl^{w}))) = (\C_1,\dotsc, \C_e)$, we have
\begin{alignat*}{2}
\frac{\max(\C_j) +i}{\ell} < \left\lceil \frac{x^B_j+i}{\ell} \right\rceil  \qquad \text{and} \qquad
\frac{\min(\ZZ \setminus \C_j) +i}{\ell} \geq \left\lfloor \frac{x^B_j+i}{\ell} \right\rfloor
\end{alignat*}
for all $j \in [1,\,e]$.
\end{lem}

\begin{proof}
If $\mvi^{\br^w}_i(B) = 0$, then for all $\bl \in \PP^{\ell}$ lying in $B$, $\bij_{\ell,e}((\bl;\br)^{w})$ has no node with $(\ell,\br^{w*}_B)$-residue $\ell-i$.
Now, 
$\bij_{\ell,e}((\bl;\br)^{w})$ will have a node with $(\ell,\br^{w*}_B)$-residue $\ell-i$ if
there exist $x, y,z \in \ZZ$ such that $x < y \leq z$, $x \notin \C_j$, $y \equiv_{\ell} \ell - i$ and $z \in \C_j$ for some $j \in [1,\,e]$.

Since $\fs(\C_j) = x^B_j$, we have $\max(\C_j) \geq x^B_j$ if and only if $\min(\ZZ \setminus \C_j) < x^B_j$.
Thus
if $\max(\C_j) \geq \left\lceil \frac{x^B_j+ i}{\ell} \right\rceil \ell - i$,
then $\max(\C_j) \geq x^B_j$ and hence $\min(\ZZ \setminus \C_j) < x^B_j$.
Consequently,
$$
\min(\ZZ \setminus \C_j) < x^B_j \leq \left\lceil \frac{x^B_j+ i}{\ell} \right\rceil \ell - i \leq \max(\C_j).$$
On the other hand, if $\min(\ZZ \setminus \C_j) < \left\lfloor \frac{x^B_j+ i}{\ell} \right\rfloor \ell - i$, then $\min(\ZZ \setminus \C_j) < x^B_j$ and hence $\max(C_j) \geq x^B_j$.
Thus
$$
\min(\ZZ \setminus \C_j) < \left\lfloor \frac{x^B_j+ i}{\ell} \right\rfloor \ell - i
\leq x^B_j \leq \max(\C_j).$$
This shows that in either of these cases, we have found $x,y,z \in \ZZ$ with $x = \min(\ZZ\setminus \C_j) \notin \C_j$, $y \in \left\{\left\lceil \frac{x^B_j+ i}{\ell} \right\rceil \ell - i, \left\lfloor \frac{x^B_j+ i}{\ell} \right\rfloor \ell - i \right\}$ and $z = \max(\C_j) \in \C_j$,
so that $\bij_{\ell,e}((\bl;\br)^{w})$ has a node with $(\ell,\br^{w*}_B)$-residue $\ell-i$.
\end{proof}

\begin{cor}
The weight of a core block of an Ariki-Koike algebra is bounded above by
$$
\left\lfloor \ell/2 \right\rfloor \left\lceil \ell/2 \right\rceil e.
$$
\end{cor}

\begin{proof}
Let $w \in \EW_{\ell}$ such that $\br^w \in \AAbar$.
Let $B$ be a core block with $\br^{w*}_B = (x^B_1,\dotsc, x^B_e)$.
By Theorem \ref{T:moving-vector-of-core-block}, there exists $i \in [1,\ell]$ such that  $\mvi^{\br^w}_i(B) = 0$.
Let $\bl \in \PP^{\ell}$ be lying in $B$, and
let $\quot_e(\UU(\beta_{\br^w}(\bl^w))) = (\C_1,\dotsc, \C_e)$.
By Lemma \ref{L:core-block-abacus},
we have $\C_j = \ZZ_{< m} \cup L_j$ for some $L_j \subseteq [m,\, m+\ell -1]$, where $m = \lfloor \frac{x^B_j+i}{\ell} \rfloor \ell - i$,
so that $|\beta^{-1}(\C_j)| \leq |L_j|(\ell - |L_j|) \leq \lfloor \ell/2 \rfloor \lceil \ell/2 \rceil$.
Consequently,
$$
\wt(B) = \wt_{\HH}(\bl) = \wt_e((\bl;\br)^w)
= \wt_e(\UU(\beta_{\br^w}(\bl^w)))
=
\sum_{j = 1}^e |\beta^{-1}(\C_j)| \leq \lfloor \ell/2 \rfloor \lceil \ell/2 \rceil e.
$$
\end{proof}

\begin{cor} \label{C:Scopes}
Let $w \in \EW_{\ell}$ with $\br^w \in \AAbar$.
Let $B$ be a block of $\HH_n$ with $\br^{w*}_B = (x^B_1,\dotsc, x^B_e)$.
If there exist $i \in [1,\,\ell]$ and $j \in [1,\,e]$ such that:
\begin{itemize}
\item $x^B_{j+1} > x^B_j + \delta_{je} \ell$;
\item $\mvi^{\br^w}_i(B) = 0$;
\item there exists $x \in [x^B_j,\, x^B_{j+1} - \delta_{je}\ell]$ such that $\ell \mid (x+i)$,
\end{itemize}
then for all $\bl \in \PP^{\ell}$ lying in $B$, $\bl$ does not have any addable node with $(e,\br)$-residue $\equiv_e j$;
in particular, $B$ and $\cs_{j} \DDot{} B$ are Scopes equivalent.

Here, $j+1$ and $\cs_j$ are to be read as $1$ and $\cs_0$ respectively when $j = e$.
\end{cor}

\begin{proof}
Let $x + i = k\ell$.
Let $\bl \in \PP^{\ell}$ be lying in $B$, and let $\quot_e(\UU(\beta_{\br^{w_{\br}}}(\bl^{w_{\br}}))) = (\C_1,\dotsc, \C_{e})$.
For each $j \in [1,\,e]$ and $b \in \C_j$, we have, by Lemma \ref{L:core-block-abacus},
$$
b \leq \max(\C_j) < \left\lceil \tfrac{x^B_j+i}{\ell} \right\rceil \ell - i
\leq \left\lceil \frac{x+i}{\ell} \right\rceil \ell - i = k\ell -i
\leq \left\lfloor \tfrac{x^B_{j+1}-\delta_{je} \ell +i}{\ell} \right\rfloor \ell -i
\leq \min (\ZZ \setminus \C_{j+1}) - \delta_{je} \ell.
$$
Thus $b \in \C_{j+1}$ if $j \ne e$ and $b+\ell \in \C_{j+1}$ if $j = e$.
This shows that $\bl^w$ does not have any addable node with $(e,\br^w)$-residue $\equiv_e j$, and hence $\bl$ does not have any addable node with $(e,\br)$-residue $\equiv_e j$.
\end{proof}

For the remainder of this section, we shall be working towards providing a necessary and sufficient conditions for two core blocks of Ariki-Koike algebras (with common $\ell$-charge $\br$) to be Scopes equivalent.
As two Scopes equivalent blocks must lie in the same $\AW_e$-orbit, they have the same moving vectors by
Proposition \ref{P:moving-vector-Weyl-orbit}.
As such, we shall fix a $w_0 \in \EW_{\ell}$ such that $\br^{w_0} \in \AAbar$ and $\mvi^{\br^{w_0}}_{\ell}(B) = 0$ for all such core blocks considered, and there is no loss of generality in doing so by Corollary \ref{C:mv}(1).
We can then simplify our notations and write $\mv(B)$ for $\mv^{\br^{w_0}}(B)$, and $\br^*_B$ for $\br^{w_0*}_B$.

We shall also fix $x^B_1,\dotsc, x^B_e \in\ZZ$ satisfying $\br^*_B = (x^B_1,\dotsc, x^B_e)$.

\begin{Def} \label{D:yz}
Let $B$ be a core block of $\HH_n$.
Define
$\by^B = (y^B_1,\dotsc, y^B_e)\in \ZZ^e$ and $\bz^B = (z^B_1,\dotsc, z^B_e) \in (\ZZ/\ell\ZZ)^e$ by
    $$\br^{*}_B = \by^B \ell + \bz^B.$$
Furthermore, define $\cy_B \in \ZZ$ and $j_B \in \ZZ/e\ZZ$ by
$$
|\by^B| = \cy_B e + j_B.
$$
\end{Def}

\begin{eg} \label{Eg:y^Bz^B}
Let $e = 5$, $\ell = 4$, $\br = (1,3,3,6)$ and $\bl = ((3,2,1^4), (4,2,1), (2^2,1), (1))$.
We have seen in Example \ref{Eg:mv} that $\bij_{\ell,e}(\bl;\br) = (((1), (1), \varnothing, \varnothing,\varnothing); (0,6,1,4,2))$, and $\mv_e(\bl;\br) = (0,1,0,1)$.
Since $\br \in \AAbar$, we see that the block $B$ in which $\bl$ lies is a core block, with $\mv^{\br}(B) = (0,1,0,1)$.

We choose $w_0 = \rho_4 \Be_4$.  Then
\begin{align*}
(\bl;\br)^{w_0} &= (((4,2,1), (2^2,1), (1), (3,2,1^4)), (3,3,6,6)), \\
\bij_{\ell,e}((\bl;\br)^{w_0}) &= (((1), (1), \varnothing, \varnothing,\varnothing);(1,7,2,5,3)), \\
\mv^{\br^{w_0}}(B) &= (1,0,1,0).
\end{align*}
Thus $\br^*_B = (1,7,2,5,3)$, $\by^B = (0,1,0,1,0)$, $\bz^B = (1,3,2,1,3)$, $\cy_B = 0$ and $j_B = 2$.
\end{eg}

\begin{lem} \label{L:yz}
Let $B$ be a core block of $\HH_n$.
For each $j \in \ZZ/e\ZZ$, we have
\begin{align*}
\by^{\cs_j \DDot{} B} &= \cs_j \CDot{1} \by^B \qquad \text{and} \qquad
\bz^{\cs_j \DDot{} B} = \cs_j \CDot{0} \bz^B = (\bz^B)^{\overline{\cs_j}}.
\end{align*}
In particular, $|\by^{\cs_j \DDot{} B}| = |\by^B|$, and so $\cy_{\cs_j \DDot{} B} = \cy_B$ and $j_{\cs_j \DDot{} B} = j_B$.
(Recall that $\overline{\cs_j}$ denotes the projection of $\cs_j$ onto $\sym{e}$.)
\end{lem}

\begin{proof}
By Lemma \ref{L:bij-*}(2), $\br^*_{\cs_j \DDot{} B} = \cs_j \CDot{\ell} \br^{*}_B$.  The lemma thus follows directly from Definition \ref{D:yz}.
\end{proof}

The following is a reformulation of Corollary \ref{C:Scopes}.

\begin{lem} \label{L:Scopes}
Let $B$ be a core block of $\HH_n$.
Let $j \in [1,\,e]$.
Suppose that one of the following holds:
\begin{itemize}
\item $y^B_{j+1} > y^B_j + \delta_{je}$;
\item $y^B_{j+1} = y^B_j + \delta_{je}$, $z^B_{j+1} > z^B_j$ and there exists $i\in [1,\,\ell]$ with $\mvi_i(B) =0$ such that
    $z^B_{j+1} \geq \ell - i \geq z^B_j$.
\end{itemize}
Then $\hub_j(B) >0$, and for all $\bl \in \PP^{\ell}$ lying in $B$, $\bl$ has no addable node with $(e,\br)$-residue $\equiv_e j$.
In particular, $B$ and $\cs_{j} \DDot{} B$ are Scopes equivalent.

Here, $j+1$, $\cs_j$ and $\hub_j(B)$ are to be read as $1$, $\cs_0$ and $\hub_0(B)$ respectively when $j = e$.
\end{lem}

\begin{proof}
Note first that $\hub_j(B) = x^B_{j+1} - x^B_j - \delta_{je}\ell$, so that $\hub_j(B) >0$ if and only if $x^B_{j+1} > x^B_j + \delta_{je}\ell$, which is equivalent to
$$
y^B_{j+1} > y^B_j + \delta_{je} \qquad \text{or} \qquad (y^B_{j+1} = y^B_j + \delta_{je} \text{ and } z^B_{j+1} > z^B_j).
$$

If $y^B_{j+1} > y^B_j + \delta_{je}$, then
$$
x^B_j = y^B_j \ell + z^B_j <  
(y^B_j + 1) \ell
\leq  y^B_{j+1}\ell - \delta_{je} \ell \leq x^B_{j+1} - \delta_{je}\ell,$$
so that $x = (y^B_{j+1}+1)\ell$ and $i = \ell$ satisfy the conditions in Corollary \ref{C:Scopes}.

On the other hand, if $y^B_{j+1} = y^B_j + \delta_{je}$, $z^B_{j+1} > z^B_j$, and $\mvi_i(B) = 0$ with $z^B_j \leq \ell -i \leq z^B_{j+1}$, 
then 
$$
x^B_j = y^B_j\ell + z^B_j \leq y^B_j\ell + (\ell - i) \leq y^B_j\ell + z^B_{j+1} = y^B_{j+1}\ell - \delta_{je}\ell + z^B_{j+1} = x^B_{j+1} - \delta_{je}\ell.
$$
Thus the conditions in Corollary \ref{C:Scopes} are satisfied with $x= y^B_j\ell + (\ell - i)$.

Consequently, the Lemma follows from Corollary \ref{C:Scopes}.
\end{proof}

\begin{Def} \hfill
\begin{enumerate}
\item For each $\by = (y_1,\dotsc, y_e) \in \ZZ^e$, define $\sigma_{\by} \in \sym{e}$ by: for each $j \in [1,\,e-1]$,
$$ y_{\sigma_{\by}(j)} \leq y_{\sigma_{\by}(j+1)},$$
 with equality only if $\sigma_{\by}(j) < \sigma_{\by}(j+1)$.
 (In other words, $\sigma_{\by}$ is the element of $\sym{e}$ with minimal length such that $\by^{\sigma_{\by}}$ is weakly increasing.)

\item Let $I \subseteq \ZZ_{\geq 0}$ with $0 \in I$.  For $b \in \ZZ_{\geq 0}$, define
        $$\Ht_I(b) := \max\{ i \in I \mid b \geq i \}.$$
      Furthermore, for $\bz = (z_1,\dotsc, z_e) \in (\ZZ_{\geq 0})^e$, 
       define
      $\tau^I_{\bz} \in \sym{e}$ by: for each $j \in [1,\,e-1]$,
    $$
    \Ht_I(z_{\tau^I_{\bz}(j)}) \geq \Ht_I(z_{\tau^I_{\bz}(j+1)}),
    $$
    with equality only if 
    one of the following holds:
\begin{itemize}
     \item $z_{\tau^I_{\bz}(j)} > z_{\tau^I_{\bz}(j+1)} = \Ht_I(z_{\tau^I_{\bz}(j+1)})$;
     \item $\tau^I_{\bz}(j) < \tau^I_{\bz}(j+1)$ and ($z_{\tau^I_{\bz}(j)} = z_{\tau^I_{\bz}(j+1)}$ or
     $z_{\tau^I_{\bz}(j)}, z_{\tau^I_{\bz}(j+1)} > \Ht_I(z_{\tau^I_{\bz}(j+1)})$).
     \end{itemize}
     (It is an easy exercise to show that $\tau^I_{\bz}$ is unique/well-defined.)
\end{enumerate}
\end{Def}

\begin{eg} \label{Eg:sigmatau}
Let $\by = (0,1,0,1,0)$.  Then $\by^{\sigma_{\by}} = (0,0,0,1,1)$ and $\sigma_{\by} = \cs_2\cs_4\cs_3$.

Let $\bz = (1,2,3,3,1)$, and $I=\{0,2\}$.  Then $\bz^{\tau^I_{\bz}} = (3,3,2,1,1)$ and $\tau^I_{\bz} = \cs_2\cs_1\cs_3\cs_2\cs_3$.
\end{eg}

\begin{lem} \label{L:tau}
Let $\bz = (z_1,\dotsc, z_e) \in (\ZZ_{\geq 0})^e$ and $I \subseteq \ZZ_{\geq 0}$ with $0 \in I$.
If $j, j' \in [1,\, e]$ are such that $z_{\tau^I_{\bz}(j)} > z_{\tau^I_{\bz}(j')}$ and $
\Ht_I(z_{\tau^I_{\bz}(j)}) \geq z_{\tau^I_{\bz}(j')}$, then $j < j'$.
\end{lem}

\begin{proof}
Since $z_{\tau^I_{\bz}(j)} > z_{\tau^I_{\bz}(j')}$, we see that $j \ne j'$.  Furthermore, $\Ht_I(z_{\tau^I_{\bz}(j)}) \geq \Ht_I(z_{\tau^I_{\bz}(j')})$ by the definition of $\Ht_I$.
If $j > j'$, then by the definition of $\tau^I_{\bz}$, we have 
$$
\Ht_I(z_{\tau^I_{\bz}(j')}) \geq \Ht_I(z_{\tau^I_{\bz}(j'+1)}) \geq \dotsb \geq \Ht_I(z_{\tau^I_{\bz}(j-1)}) \geq \Ht_I(z_{\tau^I_{\bz}(j)}).
$$
Combining with the inequality $\Ht_I(z_{\tau^I_{\bz}(j)}) \geq \Ht_I(z_{\tau^I_{\bz}(j')})$ obtained above, we get
$$
\Ht_I(z_{\tau^I_{\bz}(j')}) = \Ht_I(z_{\tau^I_{\bz}(j'+1)}) = \dotsb = \Ht_I(z_{\tau^I_{\bz}(j-1)}) = \Ht_I(z_{\tau^I_{\bz}(j)}).
$$
Now, the condition $\Ht_I(z_{\tau^I_{\bz}(j)}) \geq z_{\tau^I_{\bz}(j')}$ implies that
$$
\Ht_I(z_{\tau^I_{\bz}(j')}) = \Ht_I(z_{\tau^I_{\bz}(j)}) \geq z_{\tau^I_{\bz}(j')} \geq \Ht_I(z_{\tau^I_{\bz}(j')}) ,
$$
forcing $z_{\tau^I_{\bz}(j')} = \Ht_I(z_{\tau^I_{\bz}(j')})$.
Let $m = \max\{ b \in [j',j-1] \mid z_{\tau^I_{\bz}(b)} = \Ht_I(z_{\tau^I_{\bz}(b)}) \}$.
Then 
$$
z_{\tau^I_{\bz}(m)} = \Ht_I(z_{\tau^I_{\bz}(m)}) = \Ht_I(z_{\tau^I_{\bz}(m+1)}) < 
z_{\tau^I_{\bz}(m+1)},
$$
contradicting the definition of $\tau^I_{\bz}$.  Hence, $j < j'$.
\end{proof}

\begin{cor} \label{C:tau=1}
Let $\bz = (z_1,\dotsc, z_e) \in (\ZZ_{\geq 0})^e$ and $I \subseteq \ZZ_{\geq 0}$ with $0 \in I$.
The following statements are equivalent:
\begin{enumerate}
\item $\tau^I_{\bz} = 1_{\sym{e}}$.
\item Whenever $b,b' \in [1,\,e]$ such that $z_b > z_{b'}$ and $\Ht_{I}(z_b) \geq z_{b'}$, we have $b < b'$.
\item There does not exist $b \in [1,\,e-1]$ such that $z_{b+1} > z_{b}$ and $\Ht_{I}(z_{b+1}) \geq z_{b}$.
\item For all $b \in [1,\, e-1]$, we have $\Ht_{I}(z_{b}) \geq \Ht_{I}(z_{b+1})$,
      and if $\Ht_{I}(z_{b}) = \Ht_{I}(z_{b+1}) \in \{ z_{b}, z_{b+1} \}$, then $z_{b} \geq z_{b+1}$.
\end{enumerate}
\end{cor}

\begin{proof}
That $(1) \Rightarrow (2)$ follows from Lemma \ref{L:tau}, while $(2) \Rightarrow (3)$ is trivial.

If $\Ht_{I}(z_{b}) < \Ht_{I}(z_{b+1})$, then $z_{b} < \Ht_{I}(z_{b+1}) \leq z_{b+1}$ by the definition of $\Ht_I$, so that (3) does not hold.
Also, if $z_{b} < z_{b+1}$ and $\Ht_{I}(z_{b}) = \Ht_{I}(z_{b+1}) \in \{ z_{b}, z_{b+1} \}$, then
$$\{ z_{b}, z_{b+1} \} \ni  \Ht_{I}(z_{b+1}) = \Ht_{I}(z_{b}) \leq z_b < z_{b+1},$$ 
so that $z_b = \Ht_I(z_b) = \Ht_I(z_{b+1}) < z_{b+1}$, and once again (3) does not hold.
Thus $(3) \Rightarrow (4)$.

If (4) holds, then it is straightforward to verify that $1_{\sym{e}}$ satisfies the defining condition for $\tau^I_{\bz}$, so that $\tau^I_{\bz} = 1_{\sym{e}}$, i.e.\ (1) holds.
\end{proof}

\begin{prop} \label{P:sigma-tau}
Let $j \in [1,\, e]$.
\begin{enumerate}
\item Let $\by = (y_1,\dotsc, y_e) \in \ZZ^e$. If $y_{j+1} > y_j + \delta_{je}$, then
    $$\sigma_{\smash[t]{\cs_j \CDot{1} \by}} = \overline{\cs_j} \sigma_{\by}.$$

\item
Let $\bz = (z_1,\dotsc, z_e) \in (\ZZ_{\geq 0})^e$ and $I \subseteq \ZZ_{\geq 0}$ with $0 \in I$.
If $j < e$, $z_{j+1} > z_j$ and $\Ht_I(z_{j+1}) \geq z_j$,
then
$$
\tau^I_{\bz^{\cs_j}} = \cs_j \tau^I_{\bz}.
$$
\end{enumerate}
Again, $j+1$ and $\cs_j$ are to be read as $1$ and $\cs_0$ respectively when $j =e$.
\end{prop}

\begin{proof} \hfill
\begin{enumerate}
\item
Let $\by' = (y'_1,\dotsc, y'_e) := \cs_j \CDot{1} \by$, 
$\sigma_{\by}(a) = j$ and $\sigma_{\by}(b) = j+1$.
Then
\begin{gather*}
y'_c =
\begin{cases}
y_{j+1}- \delta_{je}, &\text{if } c = j; \\
y_j + \delta_{je}, &\text{if } c = j+1; \\
y_c, &\text{otherwise},
\end{cases} \\
y_{\sigma_{\by}(a)} = y_j < y_{j+1} = y_{\sigma_{\by}(b)},
\end{gather*}
so that $a < b$.
To show part (1), we show that $y'_{\overline{\cs_j} \sigma_{\by}(c)} \leq y'_{\overline{\cs_j} \sigma_{\by}(c+1)}$, with equality only if $\overline{\cs_j} \sigma_{\by}(c) < \overline{\cs_j} \sigma_{\by}(c+1)$ for all $c \in [1,\,e-1]$.  We do this by considering six cases separately.
\begin{description}
\item[Case 1. $\{c, c+1\} \cap \{ a,b \} = \emptyset$] this follows directly from the properties of $\sigma_{\by}$.

\item[Case 2. $c+1 = a$]  Since $a < b$, we have $c \notin \{a,b\}$, or equivalently, $\sigma_{\by}(c) \notin \{j,j+1\}$.  Thus
$$
y'_{\overline{\cs_j}\sigma_{\by}(c)} = y_{\sigma_{\by}(c)} \leq y_{\sigma_{\by}(c+1)} = y_{\sigma_{\by}(a)} = y_j = y'_{j+1} - \delta_{je} \leq y'_{j+1} = y'_{\overline{\cs_j}\sigma_{\by}(a)} = y'_{\overline{\cs_j}\sigma_{\by}(c+1)},
$$
with equality only if $y_{\sigma_{\by}(c)} = y_{\sigma_{\by}(c+1)}$ and $\delta_{je} = 0$, which yields $\sigma_{\by}(c) < \sigma_{\by}(c+1) = j < e$, in which case $\overline{\cs_j}\sigma_{\by}(c) < \overline{\cs_j}\sigma_{\by}(c+1)$.

\item[Case 3. $c = b$]  This is similar to Case 2, and we omit the details here.

\item[Case 4. $c = a$ and $c+1 \ne b$]
By the property of $\sigma_{\by}$, we have
either $y_{\sigma_{\by}(c)} = y_{\sigma_{\by}(c+1)}$ and $\sigma_{\by}(c) < \sigma_{\by}(c+1)$ --- in which case, $j = \sigma_{\by}(c) < e$ and so $\delta_{je} = 0$ --- or $y_{\sigma_{\by}(c)} \leq y_{\sigma_{\by}(c+1)}-1$.
In the former case, we get
$$
y'_{\overline{\cs_j} \sigma_{\by}(c)} = y'_{j+1} = y_j = y_{\sigma_{\by}(c)} = y_{\sigma_{\by}(c+1)} = y'_{\overline{\cs_j} \sigma_{\by}(c+1)}
$$
and $\overline{\cs_j}\sigma_{\by}(c) = j +1 < \sigma_{\by}(c+1) = \overline{\cs_j}\sigma_{\by}(c+1)$.
In the latter case, we have
$$
y'_{\overline{\cs_j}\sigma_{\by}(c)} = y'_{j+1} = y_j + \delta_{je} = y_{\sigma_{\by}(c)} +\delta_{je} \leq y_{\sigma_{\by}(c+1)} - 1 + \delta_{je} \leq  y_{\sigma_{\by}(c+1)} = y'_{\overline{\cs_j}\sigma_{\by}(c+1)},
$$
with equality only if $j=e$ in which case $\overline{\cs_j}\sigma_{\by}(c) = 1 < \overline{\cs_j}\sigma_{\by}(c+1)$.

\item[Case 5. $c \ne a$ and $c+1 = b$] This is similar to Case 4, and we omit the details here.

\item[Case 6. $c = a$ and $c+1 = b$]
We have
$$
y'_{\overline{\cs_j}\sigma_{\by}(c)} = y'_{j+1} = y_j + \delta_{je} < y_{j+1} = y'_j + \delta_{je}.
$$
Thus
$$
y'_{\overline{\cs_j}\sigma_{\by}(c)} \leq y'_j + \delta_{je} - 1 \leq y'_j = y'_{\overline{\cs_j}\sigma_{\by}(c+1)}
$$
with equality only if $j=e$, which case $\overline{\cs_j}\sigma_{\by}(c) = 1 < e = \overline{\cs_j}\sigma_{\by}(c+1)$.
\end{description}

\item
Let $(z'_1, \dotsc, z'_{e}) = \bz^{\cs_j}$.  Then $z'_b = z_{\cs_j(b)}$ for all $b \in [1,\, e]$.
Thus for $b \in [1,\, e-1]$, we have
$$
\Ht_I(z'_{\cs_j\tau^I_{\bz}(b)}) = \Ht_I(z_{\tau^I_{\bz}(b)}) \geq \Ht_I(z_{\tau^I_{\bz}(b+1)}) = \Ht_I(z'_{\cs_j\tau^I_{\bz}(b+1)}),
$$
with equality only if 
one of the following holds:
\begin{itemize}
     \item $z_{\tau^I_{\bz}(b)} > z_{\tau^I_{\bz}(b+1)} = \Ht(z_{\tau^I_{\bz}(b+1)})$;
     \item $\tau^I_{\bz}(b) < \tau^I_{\bz}(b+1)$ and ($z_{\tau^I_{\bz}(b)} = z_{\tau^I_{\bz}(b+1)}$ or
     $z_{\tau^I_{\bz}(b)}, z_{\tau^I_{\bz}(b+1)} > \Ht(z_{\tau^I_{\bz}(b+1)})$).
     \end{itemize}
These conditions can easily be seen to be equivalent to the following, except when $\{ j, j+1 \} = \{ \tau^I_{\bz}(b), \tau^I_{\bz}(b+1) \}$:
\begin{itemize}
     \item $z'_{\cs_j \tau^I_{\bz}(b)} > z'_{\cs_j \tau^I_{\bz}(b+1)} = \Ht(z'_{\cs_j \tau^I_{\bz}(b+1)})$;
     \item $\cs_j \tau^I_{\bz}(b) < \cs_j \tau^I_{\bz}(b+1)$ and ($z'_{\cs_j \tau^I_{\bz}(b)} = z'_{\cs_j \tau^I_{\bz}(b+1)}$ or
     $z'_{\cs_j \tau^I_{\bz}(b)}, z'_{\cs_j \tau^I_{\bz}(b+1)} > \Ht_I(z'_{\cs_j \tau^I_{\bz}(b)})$).
\end{itemize}
Now if $\{\tau^I_{\bz}(b), \tau^I_{\bz}(b+1)\} = \{j,j+1\}$ and $\Ht_I(z_{\tau^I_{\bz}(b)}) = \Ht_I(z_{\tau^I_{\bz}(b+1)})$, then $\Ht_I(z_j) = \Ht_I(z_{j+1})$.
Since we are given that $z_{j+1} > z_j$ and $\Ht_I(z_{j+1}) \geq z_j$, this implies that $$\Ht_I(z_{j+1}) \geq z_j \geq \Ht_I(z_j) = \Ht_I(z_{j+1}),$$ forcing equality throughout, yielding
$$
z_{j+1} > z_j = \Ht_I(z_j) = \Ht_I(z_{j+1}).
$$
Consequently none of the conditions for $\Ht_I(z_{\tau^I_{\bz}(b)}) = \Ht_I(z_{\tau^I_{\bz}(b+1)})$ holds when $\tau^I_{\bz}(b) = j$ and $\tau^I_{\bz}(b+1) = j+1$.
Thus $\tau^I_{\bz}(b) = j+1$ and $\tau^I_{\bz}(b+1) = j$, and so $z_{j+1} > z_j = \Ht_I(z_j)$ yields
$$
z'_{\cs_j \tau^I_{\bz}(b)} > z'_{\cs_j \tau^I_{\bz}(b+1)} = \Ht(z'_{\cs_j \tau^I_{\bz}(b+1)}).
$$
Hence $\cs_j \tau^I_{\bz}$ satisfies the conditions for $\tau^I_{\bz^{\cs_j}}$ so that
$\cs_j \tau^I_{\bz} = \tau^I_{\bz^{\cs_j}}$ as desired.
\end{enumerate}
\end{proof}

\begin{Def} \label{D:initial}
Let $B$ be a core block of $\HH_n$. 
Recall $\by^B = (y^B_1, \dotsc, y^B_e)$ and $\bz^B = (z^B_1,\dotsc, z^B_e)$ in Definition \ref{D:yz}, and let
$$I_B := \{ \ell-i \mid \mvi_{i}(B) = 0 \}.$$
\begin{enumerate}
\item
Define $\sigma_B,\tau_B \in \sym{e}$ by
$$
\sigma_B := \sigma_{\by^B} \qquad \text{and} \qquad \tau_B := \tau^{I_B}_{(\bz^B)^{\sigma_B}}.
$$

\item The {\em Scopes vector of $B$}, denoted $\Sc(B)$, is defined to be
$$
\Sc(B) := (\bz^B)^{\sigma_B\tau_B} = (z^B_{\sigma_B\tau_B(1)}, \dotsb, z^B_{\sigma_B\tau_B(e)}). $$

\item We call $B$ an {\em initial} core block if
    \begin{enumerate}
    \item[(I)] $y^B_{\sigma_B(e)} \leq y^B_{\sigma_B(1)} + \bbone_{\sigma_B(e) < \sigma_B(1)}$;
    \item[(II)] $\tau_B = 1_{\sym{e}}$.
    \end{enumerate}
\end{enumerate}
\end{Def}

\begin{eg}
Let $e = 5$, $\ell = 4$, $\br = (1,3,3,6)$, and let $B$ be the block of $\HH_{22}$ containing the $4$-partition $\bl = ((3,2,1^4), (4,2,1), (2^2,1), (1))$.
Let $w_0 = \rho_4\Be_4$.
Then $\mv(B) = (1,0,1,0)$, $\by^B = (0,1,0,1,0)$ and $\bz^B = (1,3,2,1,3)$ by Example \ref{Eg:y^Bz^B}.
Thus $I_B = \{0,2\}$.

By Example \ref{Eg:sigmatau}, $\sigma_B = \sigma_{\by^B} = \cs_2\cs_4\cs_3$, so that $(\bz^B)^{\sigma_B} = (1,2,3,3,1)$.
Thus, by the same example, $\tau_B = \tau^{I_B}_{(\bz^B)^{\sigma_B}} = \cs_2\cs_1\cs_3\cs_2\cs_3$.
Hence $\Sc(B) = (\bz^B)^{\sigma_B\tau_B} = (3,3,2,1,1)$.
\end{eg}

\begin{lem} \label{L:y_{i+1}=y_i}
Let $B$ be a core block of $\HH_n$. 
If $y^B_{j+1} = y^B_j + \delta_{je}$ for some $j \in [1,\,e]$,
then $\sigma_B \cs_{\sigma_B^{-1}(j)} = \overline{\cs_j} \sigma_B$.
(Again, $j+1$ and $\cs_j$ are to be read as $1$ and $\cs_0$ respectively when $j =e$.)
\end{lem}

\begin{proof}
let $\sigma_B(a) = j$. If $y^B_{j+1} = y^B_j + \delta_{je}$, then $\sigma_B(a+1) = j+1$ by the definition of $\sigma_B = \sigma_{\by^B}$.  Consequently, $\sigma_B \cs_{a} = \overline{\cs_j} \sigma_B$, as desired.
\end{proof}

\begin{prop} \label{P:Scopes-vector}
Let $B$ be a core block of $\HH_n$, and 
let $j \in [1,\, e]$.
\begin{enumerate}
\item If $y^B_{j+1} > y^B_j + \delta_{je}$, then
\begin{align*}
\sigma_{\cs_j \DDot {} B} = \overline{\cs_j} \sigma_B \qquad \text{and} \qquad
\tau_{\cs_j \DDot {} B} = \tau_B.
\end{align*}

\item If $y^B_{j+1} = y^B_{j} + \delta_{je}$, $z^B_{j+1} > z^B_j$ and
     $\Ht(z^B_{j+1}) \geq z^B_{j}$,
     then
     \begin{align*}
     \sigma_{\cs_{j} \DDot {} B} = \sigma_B \qquad \text{and} \qquad
         \tau_{\cs_{j} \DDot{} B} = \cs_{\sigma_B^{-1}(j)} \tau_B.
     \end{align*}
\end{enumerate}
Furthermore, we have $\Sc(\cs_j \DDot{} B) = \Sc(B)$ in both of these cases. (Once again, $j+1$ and $\cs_j$ are to be read as $1$ and $\cs_0$ respectively when $j =e$.)
\end{prop}

\begin{proof}
Let $\cs_j \DDot{} B = C$.
Since $\mv(B) = \mv(C)$ by Proposition \ref{P:moving-vector-Weyl-orbit}, we have $I_{B} = I_C$.
Furthermore, we have $\by^{C} = \cs_j \CDot{1} \by^B$ and $\bz^{C} = (\bz^B)^{\overline{\cs_j}}$ by Lemma \ref{L:yz}.

\begin{enumerate}
\item
  We have
  $$\sigma_{C} = \sigma_{\by^{C}} = \sigma_{\smash[t]{\cs_j \CDot{1} \by^B}} = \overline{\cs_j} \sigma_{\by^B} = \overline{\cs_j} \sigma_B,$$
  where the third equality follows from Proposition \ref{P:sigma-tau}(1).
  Thus
  $$\tau_{C} = \tau^{I_{C}}_{(\bz^{C})^{\sigma_{C}}}
  = \tau^{I_{B}}_{(\bz^{B})^{\overline{\cs_j}\sigma_{C}}}
  = \tau^{I_{B}}_{(\bz^{B})^{\sigma_{B}}} = \tau_B.
  $$
  Hence,
  $$
  \Sc(C) = (\bz^C)^{\sigma_C\tau_C} = (\bz^B)^{\overline{\cs_j} \sigma_C \tau_C} = (\bz^B)^{\sigma_B \tau_B} = \Sc(B).
  $$

\item
  Let $\sigma_B(a) = j$.  Then $\sigma_B {\cs_a} = \overline{\cs_j} \sigma_B$ by Lemma \ref{L:y_{i+1}=y_i}.
  Furthermore, we have
  $\by^{C} = \cs_j \CDot{1} \by^B = \by^B$ so that $\sigma_{C} = \sigma_{\by^{C}} = \sigma_{\by^B} = \sigma_B$.
  Thus
  $$
  \tau_{C} = \tau^{I_{C}}_{(\bz^{C})^{\sigma_{C}}}
  = \tau^{I_B}_{(\bz^B)^{\overline{\cs_j} \sigma_B}}
  = \tau^{I_B}_{(\bz^B)^{\sigma_B\cs_a}}
  = \cs_a \tau^{I_B}_{(\bz^B)^{\sigma_B}} = \cs_a \tau_B
  $$
  where the penultimate equality follows from Proposition \ref{P:sigma-tau}(2) (note that the $a$-th and $(a+1)$-th component of $(\bz^B)^{\sigma_B}$ are $z^B_{\sigma_B(a)} = z^B_j$ and $z^B_{\sigma_B(a+1)} = z^B_{j+1}$ respectively).  Now
  $$
  \Sc(C) = (\bz^C)^{\sigma_C\tau_C} = (\bz^B)^{\overline{\cs_j} \sigma_C \tau_C} = (\bz^B)^{\overline{\cs_j} \sigma_B \cs_a \tau_B} = (\bz^B)^{\sigma_B \tau_B} = \Sc(B).
  $$
\end{enumerate}
\end{proof}

\begin{lem} \label{L:equivalent}
Let $B$ be a core block of $\HH_n$.
The following statements are equivalent:
\begin{enumerate}
  \item $B$ satisfies Condition (I) of being initial, i.e.\ $y^B_{\sigma_B(e)} \leq y^B_{\sigma_B(1)} + \bbone_{\sigma_B(e) < \sigma_B(1)}$.
  \item
  For all $j \in [1,\,e]$, $y^B_j =
  \cy_B + \bbone_{j \leq j_B}.$

  \item For all $j \in [1,\,e] \setminus \{ \overline{j_B} \}$, $y^B_{\overline{j+1}} = y^B_j + \delta_{je}$.

  \item For all $j \in [1,\,e]$, $y^B_{\overline{j+1}} \leq y^B_{j} + \delta_{je}$.

  \item $y^B_{\sigma_B(e)} = y^B_{\sigma_B(1)} + \bbone_{\sigma_B(e) < \sigma_B(1)}$.
\end{enumerate}
Here, and hereafter, for $x \in \ZZ$, set $\overline{x}$ to be the unique element in $[1,\,e]$ satisfying $\overline{x} \equiv_e x$.
%
\end{lem}

\begin{proof} 
  For (1) $\Rightarrow$ (2), let $m := \max\{j \in [1,\,e] \mid y^B_{\sigma_B(j)} = y^B_{\sigma_B(1)} \}$.
Then
$y^B_{\sigma_B(1)} = y^B_{\sigma_B(2)} = \dotsb = y^B_{\sigma_B(m)}$, and so $\sigma_B(1) < \sigma_B(2) < \dotsb < \sigma_B(m)$.
If $m=e$, then $y^B_j = y^B_{\sigma_B(1)}$ for all $j \in [1,\,e]$, so that
$\cy_B e + j_B = |\by^B| = y^B_{\sigma_B(1)} e$.  Hence $y^B_{\sigma_B(1)} = \cy_B$ and $j_B = 0$, and so $y^B_j = \cy_B + \bbone_{j \leq j_B}$ for all $j \in [1,\,e]$.

On the other hand, if $m < e$, then
$$y^B_{\sigma_B(1)} = y^B_{\sigma_B(m)} < y^B_{\sigma_B(m+1)} \leq y^B_{\sigma_B(m+2)}  \leq \dotsb \leq y^B_{\sigma_B(e)} \leq y^B_{\sigma_B(1)} + \bbone_{\sigma_B(e) < \sigma_B(1)} \leq y^B_{\sigma_B(1)} +1,$$
forcing
$$
y^B_{\sigma_B(m+1)} = y^B_{\sigma_B(m+2)}  = \dotsb =  y^B_{\sigma_B(e)} = y^B_{\sigma_B(1)} + \bbone_{\sigma_B(e) < \sigma_B(1)} = y^B_{\sigma_B(1)} +1.
$$
Hence
$$ \cy_B e + j_B = |\by^B| = m y^B_{\sigma_B(1)} + (e-m)(y^B_{\sigma_B(1)} + 1) = y^B_{\sigma_B(1)} e + (e-m)
$$
so that $\cy_B = y^B_{\sigma_B(1)}$ and $j_B = e-m$.
Furthermore,
$\sigma_B(m+1) < \sigma_B(m+2) < \dotsb < \sigma_B(e)$ and $\sigma_B(e) <\sigma_B(1)$.
We thus have
$$\sigma_B(m+1) < \dotsb < \sigma_B(e) < \sigma_B(1) < \dotsb < \sigma_B(m).$$
Consequently,
\begin{alignat*}{4}
&\sigma_B(m+1) = 1, &\quad& \sigma_B(m+2) = 2, & & \dotsc,\ &\quad& \sigma_B(e) = e-m, \\
&\sigma_B(1) = e-m+1, && \sigma_B(2) = e-m+2, &\quad &\dotsc,\ && \sigma_B(m) = e.
\end{alignat*}
Thus $y^B_j = \cy_B + \bbone_{j \leq j_B}$.

It is straightforward to verify (2) $\Rightarrow$ (3) $\Rightarrow$ (4), and (5) $\Rightarrow$ (1) is trivial.

For (4) $\Rightarrow$ (5), we have
$$
y^B_1 \geq \dotsb \geq y^B_e \geq y^B_1 - 1,$$
by (4).
Let $M = \max\{ j \mid y^B_j = y^B_1\}$.
Then
$$y^B_1 = \dotsb = y^B_M = y^B_{M+1} +1 = \dotsb = y^B_e + 1.$$
If $M = e$, then $\sigma_B(1) = 1$ and $\sigma_B(e) = e$, and
$$y^B_{\sigma_B(e)} = y^B_e = y^B_1 = y^B_{\sigma_B(1)} = y^B_{\sigma_B(1)} + \bbone_{\sigma_B(1) > \sigma_B(e)}.$$
On the other hand, if $M < e$, then $\sigma_B(1) = M+1$ and $\sigma_B(e) = M$,
$$y^B_{\sigma_B(e)} = y^B_M = y^B_{M+1} + 1 = y^B_{\sigma_B(1)} +1 = y^B_{\sigma_B(1)} + \bbone_{\sigma_B(1) > \sigma_B(e)}.$$
\end{proof}


We record the following corollary, which is immediate from Lemma \ref{L:equivalent}(2).

\begin{cor} \label{C:I-sigma_B}
Let $B$ be a core block of $\HH_n$
satisfying Condition (I) of being initial.
Then
\begin{align*}
\sigma_B &= \rr_e^{j_B} = \rr_e^{|\by^B|}; \\
\by^B &= (\bv_e(|\by^B|))^{\rho_e^{-j_B}} = (\smash[t]{\overbrace{\cy+1, \dotsc, \cy+1}^{j_B \text{ times}}, \overbrace{\cy, \dotsc, \cy}^{e-j_B \text{ times}}}),
\end{align*}
where $\rho_e = (1,2,\dotsc,e) \in\sym{e}$.
(Recall the notation $\bv_e(x)$ in Subsection \ref{SS:notations}.)
\end{cor}

\begin{lem} \label{L:initial-r^*}
Let $B$ be an initial core block of $\HH_n$.  Then
$$\br^*_B = (\bv_e(|\by^B|)\ell + \Sc(B))^{\rho_e^{-j_B}}.$$
\end{lem}

\begin{proof}
By Corollary \ref{C:I-sigma_B}, $\by^B = (\bv_e(|\by^B|))^{\sigma_B^{-1}}$, and $\sigma_B = \rho_e^{j_B}$.
Since $\Sc(B) = (\bz^B)^{\sigma_B\tau_B} = (\bz^B)^{\sigma_B}$,
we have
\begin{align*}
\br^*_B &= \by^B \ell + \bz^B = (\bv_e(|\by^B|))^{\sigma_B^{-1}} \ell + (\Sc(B))^{\sigma_B^{-1}} \\
&= (\bv_e(|\by^B|)\ell + \Sc(B))^{\sigma_B^{-1}} = (\bv_e(|\by^B|)\ell + \Sc(B))^{\rho_e^{-j_B}}.
\end{align*}
\end{proof}

\begin{lem} \label{L:unique-j-initial}
Let $\bv \in [0,\, \ell-1]^e$ and $k \in \ZZ_{\geq 0}$.  There is at most one initial core block $B$ with $\wt(B) = k$ and $\Sc(B) = \bv$.
\end{lem}

\begin{proof}
Let $B$ be an initial core block $B$ with $\wt(B) = k$ and $\Sc(B) = \bv$.
Then $|\bv| = |\Sc(B)| = |(\bz^B)^{\sigma_B\tau_B}| = |\bz^B|$.
Thus,
\begin{align*}
\cy_B = \frac{|\by^B| - j_B}{e} = \frac{(|\br^*_B| - |\bz^B|)/\ell - j_B}{e} =
\frac{|\br^*_B| - |\bv| - j_B\, \ell}{e\ell}.
\end{align*}
The same would also hold for another initial core block $B'$ with $\wt(B') = k$ and $\Sc(B') = \bv$.
Thus,
\begin{align*}
\ZZ \ni \cy_B - \cy_{B'} = \frac{j_{B'} - {j_B}}{e},
\end{align*}
forcing ${j_{B'}} = {j_B}$, since $|{j_{B'}} - {j_B}| < e$, and hence $\cy_B = \cy_{B'}$.
Thus $\sigma_B = \sigma_{B'}$ and $\by^{B} = \by^{B'}$ by Corollary \ref{C:I-sigma_B}.
Hence $\br^*_B = \br^*_{B'}$  by Lemma \ref{L:initial-r^*},
so that
$\core(B) = \core(B')$ (since $\quot_e(\beta_{|\br^{w_0}|}(\core(B))) = \beta_{\br^*_B}(\EP)$ and similarly for $\core(B')$ by Lemma \ref{L:ell-residue}(1)).
Since $\wt(B) = k = \wt(B')$, we see that $B = B'$.
\end{proof}

\begin{prop} \label{P:Scopes}
Let $B$ be a core block of $\HH_n$.
There exists a sequence $B_0, \dotsc, B_k$ of blocks such that:
\begin{enumerate}
\item $B_0 = B$, and $B_k$ is an initial core block with $\br^*_{B_k} = (\bv_e(|\by^B|)\ell + \Sc(B))^{\rho_e^{-j_B}}$  (recall the notation $\bv_e(x)$ in Subsection \ref{SS:notations} and $\rho_e = (1,2,\dotsc, e) \in \sym{e}$);
\item for each $a \in [0, k-1]$, there exists $j_a \in \ZZ/e\ZZ$ such that $\hub_{j_a}(B_a) >0$ and $\bl$ has no addable nodes of $(e,\br)$-residue $j_a$ for all $\bl \in \PP^{\ell}$ lying in $B_a$ ;
\item $\Sc(B_0) = \Sc(B_1) = \dotsb = \Sc(B_k)$.
\end{enumerate}
In particular, $B$ and $B_k$ are Scopes equivalent.

Furthermore, under the Scopes equivalence between $B$ and $B_k$, the Specht module $\Sp^{\bl}$ of $B$ with $\bij_{\ell,e}((\bl;\br)^{w_0}) = (\bl^*;\br^*_B)$ corresponds to the Specht module $\Sp^{\bm}$ of $B_k$ with
$$
\bij_{\ell,e}((\bm;\br)^{w_0}) = ((\bl^*)^{\sigma_B\tau_B\rho_e^{-j_B}} ; \br^*_{B_k}).
$$
\end{prop}

\begin{proof}
We prove by induction on $n$.
If $B$ is an initial core block, then all statements hold by Lemma \ref{L:initial-r^*} and Corollary \ref{C:I-sigma_B}.
Assume thus that $B$ is not initial.


If $B$ does not satisfy Condition (I) of being initial, then there exists $j \in [1,\,e]$ such that $y^B_{\overline{j+1}} > y^B_j + \delta_{je}$ by Lemma \ref{L:equivalent}(4).

On the other hand, if $B$ satisfies Condition (I) but not Condition (II) of being initial,
then $y^B_{\overline{j+1}} = y^B_j + \delta_{je}$ for all $j \in [ 1,\, e] \setminus \{ \overline{j_B} \}$ by Lemma \ref{L:equivalent}(3), $\sigma_B = \rho_e^{j_B}$ by Corollary \ref{C:I-sigma_B},
and there exists $b \in [1,\,e-1]$ such that $z^B_{\sigma_B(b+1)}  > z^B_{\sigma_B(b)}$ and $\Ht_{I_B}(z^B_{\sigma_B(b+1)}) \geq z^B_{\sigma_B(b)}$ by Corollary \ref{C:tau=1}(3).
Let $j = \sigma_B(b)$.
Then $j \ne \sigma_B(e) = j_B$ and
$$\sigma_B(b+1) = \rho_e^{j_B}(\rho_e(b)) = \rho_e(\rho_e^{j_B}(b)) = \rho_e(\sigma_B(b)) = \rho_e (j) = \overline{j+1}.$$
Thus we have $y^B_{\overline{j+1}} = y_j + \delta_{je}$, $z^B_{\overline{j+1}} > z^B_j$ and $\Ht_{I_B}(z^B_{\overline{j+1}}) \geq z^B_j$.

Consequently, for both cases, we have $\hub_j(B) >0$, and $\bl$ has no addable nodes of $(e,\br)$-residue $j$ for all $\bl \in \PP^{\ell}$ lying in $B$ by Lemma \ref{L:Scopes}.
Let $B_1 = \cs_j \DDot{} B$.
Then $B_1$ is a core block of $\HH_{n-\hub_j(B)}$, with $\mv(B) = \mv(B_1)$ by Proposition
\ref{P:moving-vector-Weyl-orbit},
$|\by^{B_1}| = |\by^B|$ by Lemma \ref{L:yz}
and $\Sc(B) = \Sc(B_1)$ by Proposition \ref{P:Scopes-vector}.
Since $n- \hub_j(B) < n$, by induction, parts (1)--(3) hold for $B_1$, and hence for $B$.

We now show the correspondence of $\ell$-partitions lying $B$ and $B_k$ under the Scopes equivalence between them.
For each $\bl \in \PP^{\ell}$ with $\bij_{\ell,e}((\bl;\br)^{w_0}) = (\bl^*;\br^*_B)$, the Specht module $\Sp^{\bl}$ lying in $B$ corresponds to $\Sp^{\cs_j\DDot{\br} \bl}$ lying in $B_1$ by Theorem \ref{T:CR}, and
$\bij_{\ell,e} ((\cs_j \DDot{\br} \bl; \br)^{w_{0}}) = ((\bl^*)^{\overline{\cs_j}}; \br^*_{\cs_j \DDot{} B})$
by Lemma \ref{L:bij-*}(2).
By induction, $\Sp^{\cs_j\DDot{\br} \bl}$ corresponds to $\Sp^{\bm}$ lying in the initial core block $B_k$ with
\begin{align*}
\bij((\bm;\br)^{w_{0}}) = (((\bl^*)^{\overline{\cs_j}})^{\sigma_{B_1}\tau_{B_1}\rho_e^{-j_{B_1}}};
\br^*_{B_k}).
\end{align*}
By Lemma \ref{L:yz}, $j_{B_1} = j_{\cs_j \DDot{} B} = j_B$.
%
Furthermore, if $B$ does not satisfy Condition (I) of being initial, then $\sigma_{B_1} = \overline{\cs_j}\sigma_B$ and $\tau_{B_1} = \tau_B$ by Proposition \ref{P:Scopes-vector}(1), and so $\overline{\cs_j}\sigma_{B_1}\tau_{B_1} = \sigma_B \tau_B$.
On the other hand,
if $B$ satisfies Condition (I) but not Condition (II) of being initial, then setting $a = \sigma_B^{-1}(j)$, we have $\overline{\cs_j}\sigma_B = \sigma_B\cs_{a}$ by Lemma \ref{L:y_{i+1}=y_i}, and
$\sigma_{B_1} = \sigma_B$ and $\tau_{B_1} = \cs_{a}\tau_B$ by Proposition \ref{P:Scopes-vector}(2),
so that
$\overline{\cs_j}\sigma_{B_1}\tau_{B_1} = \overline{\cs_j}\sigma_{B}\cs_a\tau_{B} = \sigma_B \tau_B$ too.
In both cases, we get
\begin{align*}
\bij((\bm;\br)^{w_{0}}) = ((\bl^*)^{\sigma_{B}\tau_{B}\rho_e^{-j_B}};
\br^*_{B_k})
\end{align*} as desired.
\end{proof}

\begin{prop} \label{P:not-Scopes}
Let $B$ be a core block of $\HH_n$.  Suppose that there exists $j \in [1,\,e]$ such that
$$
y^B_{j+1} = y^B_j + \delta_{je} \quad \text{ and } \quad z^B_j,z^B_{j+1} > \Ht_{I_B}(z^B_{j}) = \Ht_{I_B}(z^B_{j+1}).
$$
Then there exists $\bl \in \PP^{\ell}$ lying in $B$ having both an addable node and a removable node, both with $(e;\br)$-residue $\equiv_e j$.

As before, $j+1$ is to be read as $1$ when $j=e$.
\end{prop}

\begin{proof}
Let $(\bk;\bu) \in \PP^{\ell} \times \ZZ^{\ell}$ be such that $\UU(\beta_{\bu}(\bk)) = \beta_{|\br^{w_0}|}(\core(B))$.  Then $\bu \in \AAbar$ \cite[Proposition 2.14]{JL-cores}.
%

Let $i_1 \in [1,\,\ell]$ be such that $\ell - i_1 = \Ht_{I_B}(z^B_j)$.
Let $i_2 = \max(\{0\} \cup \{ i < i_1 \mid \mvi_i(B) = 0\})$.
Since $z^B_j,z^B_{j+1} > \Ht_{I_B}(z^B_{j}) = \Ht_{I_B}(z^B_{j+1})$, we have $\ell - i_2 > z^B_j, z^B_{j+1} > \ell - i_1$; in particular, $i_2 +1 < i_1$.

Now $\mvi_{i_2+1}(B) , \mvi_{i_2+2}(B) , \dotsc , \mvi_{i_1 -1}(B) > 0$ by definition of $i_2$.
Set
$$\bv = (v_1,\dotsc, v_{\ell}) := \mv(B) - \Be_{i_2+1} - \Be_{i_2+2} - \dotsb - \Be_{i_1 -1}.$$
Then $\bv \in (\ZZ_{\geq 0})^{\ell}$.
Let $\br^{w_0} = (r'_1,\dotsc, r'_{\ell})$, and recall that $\br^{w_{0}} \in \AAbar$.
Set $\bt = (t_1,\dotsc, t_{\ell}) \in \ZZ^{\ell}$ by
$$t_a := r'_a  - \delta_{a,i_2+1} + \delta_{a,i_1}$$
for all $a \in [1,\,\ell]$.
Then for $a \in [1, \ell-1]$,
$$t_{a+1} - t_a = r'_{a+1} - r'_a - \delta_{a+1,i_2+1} + \delta_{a+1,i_1} + \delta_{a,i_2+1} - \delta_{a,i_1}, $$
so that $t_{a+1} \geq t_a$ when $a \notin \{i_1,i_2\}$.
When $a \in \{ i_1,i_2\}$,
since $\bu = (u_1,\dotsc, u_{\ell}) \in \AAbar$ and
$u_a = r'_a - \mvi_a(B) + \mvi_{a-1}(B)
$ for all $a \in [1,\ell]$
by Corollary \ref{C:t*}, we have
\begin{align*}
t_{a+1} - t_a &= r'_{a+1} - r'_a - \delta_{a+1,i_2+1} + \delta_{a+1,i_1} + \delta_{a,i_2+1} - \delta_{a,i_1} \\
&= u_{a+1} - u_a + \mvi_{a+1}(B) - 2\mvi_a(B) + \mvi_{a-1}(B) - \delta_{a+1,i_2+1} + \delta_{a+1,i_1} + \delta_{a,i_2+1} - \delta_{a,i_1} \\
&= u_{a+1} - u_a + \mvi_{a+1}(B) + \mvi_{a-1}(B) - \delta_{a+1,i_2+1} + \delta_{a+1,i_1} + \delta_{a,i_2+1} - \delta_{a,i_1} \geq 0
\end{align*}
since $\mvi_{i_1-1}(B), \mvi_{i_2+1}(B) \geq 1$.
Consequently,
as
$$
t_a = r'_a - \delta_{a,i_2+1} + \delta_{a,i_1} = u_a + \mvi_a(B) - \mvi_{a-1}(B) - \delta_{a,i_2+1} + \delta_{a,i_1} = u_a + v_a - v_{a-1},$$
there exists $\bm = (\mu^{(1)},\dotsc, \mu^{(\ell)}) \in \PP^{\ell}$ such that $\mv_e(\bm;\bt) = \bv$ and $\core_e(\bm;\bt) = \UU(\beta_{\bu}(\bk))$ by Corollary \ref{C:mv-construction}.

Note that
\begin{alignat}{3}
t_{i_1} - t_{i_2+1} &= r'_{i_1} - r'_{i_2+1} + 2 &&\geq 2; \label{E:i_1-i_2+1} \\
t_{i_1} - t_a &= r'_{i_1} - r'_a + 1 &&\geq 1 ; \label{E:i_1-i_a} \\
t_{a} - t_{i_2+1} &= r'_a - r'_{i_2+1} + 1 &&\geq 1 \label{E:i_a-i_2+1}
\end{alignat}
for all $a \in [i_2+2, \, i_1-1]$.
In particular, there exist distinct $b_1, b_2 \in \beta_{t_{i_1}}(\mu^{(i_1)}) \setminus \beta_{t_{i_2+1}}(\mu^{(i_2+1)})$ by Lemma \ref{L:hub}(1).

For $k \in \{ 1,2 \}$,
define $\bnu_k = (\nu_k^{(1)},\dotsc, \nu_k^{(\ell)}) \in \PP^{\ell}$ by
$$
\beta_{r'_c}(\nu_k^{(c)}) =
\begin{cases}
\beta_{t_{i_2+1}}(\mu^{(i_2+1)}) \cup \{b_k\},  &\text{if } c = i_2 + 1; \\
\beta_{t_{i_1}}(\mu^{(i_1)}) \setminus \{b_k\},  &\text{if } c = i_1; \\
\beta_{t_c}(\mu^{(c)}),  &\text{otherwise}.
\end{cases}
$$
Then $\UU(\beta_{\br^{w_0}}(\bnu_k))$ can be obtained from $\UU(\beta_{\bt}(\bm))$ by replacing $\uu_{i_1}(b_k)$ with $\uu_{i_2+1}(b_k)$, which equals $\uu_{i_1}(b_k) + (i_1 - (i_2+1))e$ by Lemma \ref{L:Uglov map}(4).
Thus
\begin{align*}
\core_e(\bnu_k;\br^{w_0}) &= \core_e(\UU(\beta_{\br^{w_0}}(\bnu_k))) = \core_e(\UU(\beta_{\bt}(\bm))) = \core_e(\bm;\bt) = \UU(\beta_{\bu}(\kappa)); \\
\wt_e(\bnu_k;\br^{w_0}) &= \wt_e(\UU(\beta_{\br^{w_0}}(\bnu_k))) = \wt_e(\UU(\beta_{\bt}(\bm))) + i_1 - (i_2+1) \\
& = |\mv_e(\bm;\bt)| +i_1 - (i_2+1) = |\mv(B)| = \wt(B). 
\end{align*}
Consequently,
\begin{gather*}
\core_{\HH}(\bnu_k^{{w_0}^{-1}}) = \beta^{-1}(\core_e(\bnu_k;\br^{w_0})) = \beta^{-1}(\UU(\beta_{\bu}(\kappa))) = \core(B); \\
\wt_{\HH}(\bnu_k^{{w_0}^{-1}}) = \wt_e(\bnu_k;\br^{w_0}) = \wt(B),
\end{gather*}
so that $\bnu_k^{{w_0}^{-1}}$ lies in $B$.

We now consider three separate cases:
\medskip

\noindent \textbf{Case 1. $\bm$ has some addable node and some removable node, both with $(e,\bt)$-residue $\equiv_e j$:}

If there exists $k \in \{1,2\}$ such that $b_k \not\equiv_e j$ and $b_k \not\equiv_e j-1$, then clearly $\bnu_k$ has some addable node and some removable node,  both with $(e;\br^{w_0})$-residue $\equiv_e j$, inherited from $(\bm;\bt)$.

On the other hand, if for all $k \in \{1,2\}$, $b_k \equiv_e j$ or $b_k \equiv_e j-1$, let $b_k = q_k e + b'_k$ where $b'_k \in \ZZ/e\ZZ$ and $q_i \in \ZZ$.
Then
\begin{alignat*}{2}
\uu_{i_1}(b_k) &= ((q_k+1)\ell - i_1)e + b'_k &&\notin \UU(\beta_{\br^{w_0}}(\bnu_k)); \\
\uu_{i_2+1}(b_k) &= ((q_k+1)\ell - i_2-1)e + b'_k &&\in \UU(\beta_{\br^{w_0}}(\bnu_k)).
\end{alignat*}
This implies, by Lemma \ref{L:core-block-abacus} with $i = \ell$, that
\begin{align*}
(q_k+1)\ell - i_1 &\geq \left\lfloor (x^B_{b'_k+1})/\ell \right\rfloor  \ell = y^B_{b'_k+1}\ell ; \\
(q_k+1)\ell - i_2-1 &< \left\lceil (x^B_{b'_k+1})/\ell \right\rceil \ell \leq (y^B_{b'_k+1}+1)\ell.
\end{align*}
Consequently,
$$
y^B_{b'_k+1}\ell - (\ell - i_1) \leq q_k\ell \leq y^B_{b'_k+1}\ell + i_2
$$
forcing $q_k = y^B_{b'_k+1}$ since $i_2, \ell-i_1 \in [0,\,\ell-1]$.
Since $b_1 \ne b_2$, this forces $b'_1 \ne b'_2$, and hence $b_1 \not\equiv_e b_2$.
Without loss of generality, we may assume that $b_1 \equiv_e j-1$, and $b_2 \equiv_e j$.
Then $b'_1 = j-1$ and $b'_2 = j - \delta_{je} e$, and so
$$b_2 = q_2e + b'_2 = y^B_{j+1}e + j-\delta_{je}e = y^B_je + j = q_1e + b'_1 + 1 = b_1 + 1$$
since $y^B_{j+1} = y^B_j + \delta_{je}$.
Thus moving $b_2$ from $\beta_{t_{i_1}}(\mu^{(i_1)})$ to $\beta_{t_{i_2+1}}(\mu^{(i_2+1)})$ will yield an addable node of $\nu_2^{(i_1)}$ with $(e,r'_{i_1})$-residue $\equiv_e j$ corresponding to $b_2$.
Thus $\bnu_2$ will have an addable node with $(e,\br^{w_0})$-residue $\equiv_e j$, as well as any removable node with $(e,\br^{w_0})$-residue $\equiv_e j$ inherited from $\bm$.

In all cases, $\bnu_k$ has an addable node and a removable node, both with $(e;\br^{w_0})$-residue $\equiv_e j$ for some $k \in \{1,2\}$.
Let $\bl = \bnu^{{w_0}^{-1}}_k$.  Then $\bl$ lies in $B$, and has an addable node and a removable node, both with $(e,\br)$-residue $\equiv_e j$.

\medskip

\noindent \textbf{Case 2. $\bm$ has no addable node with $(e,\bt)$-residue $\equiv_e j$:}

Let $\quot_e(\UU(\beta_{\br^{w_0}}(\bnu_1))) = (\C_1,\dotsc, \C_e)$.
Firstly,
$$\max(\C_{j}) \geq \fs(\C_{j}) - 1 = x^B_{j} - 1 = y^B_{j}\ell + z^B_{j} - 1 \geq (y^B_{j} + 1)\ell - i_1$$
since $z^B_{j} > \ell -i_1$ so that $z^B_j -1 \geq \ell -i_1$.
On the other hand, if $i_2 >0$, then by Lemma \ref{L:core-block-abacus} with $i = i_2$,
\begin{equation}
\max(\C_{j}) < \lceil (x^B_{j} + i_2)/\ell \rceil \ell - i_2 = \lceil (y^B_{j}\ell + z^B_{j} + i_2)/\ell \rceil \ell - i_2 = (y^B_{j}+ 1)\ell - i_2   \label{E:C_j}
\end{equation}
since $0 \leq z^B_{j} < \ell - i_2$.
If $i_2 = 0$, then by Lemma \ref{L:core-block-abacus} with $i = \ell$,
$$
\max(\C_{j}) < \lceil x^B_{j} /\ell \rceil \ell = \lceil (y^B_{j}\ell + z^B_{j})/\ell \rceil \ell = (y^B_{j}+ 1)\ell = (y^B_{j}+ 1)\ell - i_2
$$
since $0\leq \ell-i_1 < z^B_{j} < \ell$.  Thus, in all cases,
$$
(y^B_{j} + 1)\ell - i_1 \leq \max(\C_j) \leq (y^B_{j}+ 1)\ell - i_2 -1.
$$
Consequently, there exists $a \in [i_2+1,\, i_1]$ such that
$\max(\C_{j}) = y^B_{j} \ell + (\ell -a)$,
so that
$$\uu_{a}(y^B_{j}e + j - 1) = y^B_j\ell e + (\ell-a)e + j- 1 = \max(\C_j)e + j-1 \in \UU(\beta_{\br^{w_0}}(\bnu_1)),$$
and hence $x := y^B_je + j-1 \in \beta_{r'_a}(\nu^{(a)}_1)$.
This implies that $x \in \beta_{t_{a_1}}(\mu^{(a_1)})$ for some $a_1 \in [i_2+1,\, i_1]$.
Since $\bm$ has no addable node with $(e,\bt)$-residue $\equiv_e j$, we have $x+1 =  y^B_je + j \in \beta_{t_{a_1}}(\mu^{(a_1)})$ too.

Now,
$$
\min(\ZZ \setminus \C_{j+1}) \leq \fs(\C_{j+1}) = x^B_{j+1} = y^B_{j+1} \ell + z^B_{j+1} \leq y^B_{j+1}\ell + \ell- i_2-1.
$$
On the other hand, by Lemma \ref{L:core-block-abacus} with $i = i_1$, we have
$$
\min(\ZZ \setminus \C_{j+1}) \geq \lfloor (x^B_{j+1} +i_1)/\ell \rfloor \ell - i_1 \geq
\lfloor (y^B_{j+1}\ell + z^B_{j+1} + i_1)/\ell \rfloor \ell - i_1 = (y^B_{j+1} + 1)\ell - i_1
$$
since $\ell- i_1 \leq z_{j+1} < \ell$.
Thus, there exists $a' \in [i_2+1,\, i_1]$ such that $\min(\ZZ\setminus \C_{j+1}) = y^B_{j+1}\ell + (\ell - a')$,
so that
\begin{align*}
\uu_{a'}(x+1) &=
\uu_{a'}(y^B_{j}e + j) = \uu_{a'}((y^B_j+ \delta_{je})e + j - \delta_{je}e)
= \uu_{a'}(y^B_{j+1}e + j-\delta_{je}e) \\&
=y^B_{j+1}\ell e + (\ell-a')e + j-\delta_{je}e = \min(\ZZ \setminus \C_{j+1})e + j-\delta_{je}e \notin \UU(\beta_{\br^{w_0}}(\bnu_1)).
\end{align*}
Hence $x+1 \notin \beta_{r'_{a'}}(\nu^{(a')}_1)$, and so $x+1 \notin \beta_{t_{a_2}}(\mu^{(a_2)})$ for some $a_2 \in [i_2+1,i_1]$.
Since $(\bm;\bt)$ has no addable node with $(e,\bt)$-residue $\equiv_e j$, we have $x \notin \beta_{t_{a_2}}(\mu^{(a_2)})$ too.

If $a_1 < i_1$, then $t_{a_1} < t_{i_1}$ by \eqref{E:i_1-i_2+1} and \eqref{E:i_1-i_a}, so there exists $y \in \beta_{t_{i_1}}(\mu^{(i_1)}) \setminus \beta_{t_{a_1}}(\mu^{(a_1)})$.
If $a_2 > i_2 + 1$, then $t_{i_2+1} < t_{a_2}$ by \eqref{E:i_1-i_2+1} and \eqref{E:i_a-i_2+1}, so there exists $y' \in \beta_{t_{a_2}}(\mu^{(a_2)}) \setminus \beta_{t_{i_2+1}}(\mu^{(i_2+1)})$.
Note that $\{y,y'\} \cap \{ x,x+1 \} = \emptyset $ since $x,x+1 \in \beta_{t_{a_1}}(\mu^{(a_1)}) \setminus
\beta_{t_{a_2}}(\mu^{(a_2)})$.
Define $\bnu = (\nu^{(1)},\dotsc, \nu^{(\ell)}) \in \PP^{\ell}$ by
$$\beta_{r'_c}(\nu^{(c)}) =
\begin{cases}
\beta_{t_{i_1}} (\mu^{(i_1)}) \setminus \{ x +1 \}, & \text{if } c = i_1 = a_1; \\
\beta_{t_{i_1}} (\mu^{(i_1)}) \setminus \{ y \}, &\text{if } c = i_1 > a_1; \\
\beta_{t_{a_1}} (\mu^{(a_1)}) \cup \{y \} \setminus \{ x+1 \}, & \text{if } c = a_1 < i_1; \\
\beta_{t_{a_2}} (\mu^{(a_2)}) \cup \{x+1 \} \setminus \{ y' \}, & \text{if } c = a_2 > i_2 + 1; \\
\beta_{t_{i_2+1}} (\mu^{(i_2+1)}) \cup \{x+1 \},  & \text{if } c = i_2 + 1 = a_2; \\
\beta_{t_{i_2+1}} (\mu^{(i_2+1)}) \cup \{ y' \}, & \text{if } c = i_2 + 1 < a_2; \\
\beta_{t_c}(\mu^{(c)}), &\text{otherwise}.
\end{cases}$$
Then $\nu^{(a_1)}$ has an addable node with $(e, r'_{a_1})$-residue $\equiv_e j$ while $\nu^{(a_2)}$ has a removable node with $(e,r'_{a_2})$-residue $\equiv_e j$, with both nodes corresponding to $x+1$.
Thus $\bnu$ has an addable node and a removable node, both with $(e,\br^{w_0})$-residue $\equiv_e j$.
Similar to how we showed that $\bnu_k^{{w_0}^{-1}}$ lies in $B$, it is straightforward to verify that $\bnu^{{w_0}^{-1}}$ also lies in $B$. Let $\bl = \bnu^{{w_0}^{-1}}$.
Then $\bl$ lies in $B$ and has an addable node and a removable node, both with $(e,\br)$-residue $\equiv_e j$.
\medskip

\noindent \textbf{Case 3. $\bm$ has no removable node with $(e,\bt)$-residue $\equiv_e j$:}

This is similar to Case 2, and its proof just involves interchanging the roles of $\C_j$ and $\C_{j+1}$.
\end{proof}

\begin{thm} \label{T:Scopes}
Let $B$ and $C$ be two core blocks of Ariki-Koike algebras with (common $\ell$-charge $\br$ and) $\mv(B) = \mv(C)$.
Then $B$ and $C$ are Scopes equivalent if and only if
$\Sc(B) = \Sc(C)$,
in which case $\Sp^{\bl}$ lying in $B$ corresponds to $\Sp^{\bm}$ lying in $C$ if and only if $(\bl^*)^{\sigma_B\tau_B} = (\bm^*)^{\sigma_C\tau_C}$, where $\bij_{\ell,e}((\bl;\br)^{w_{0}}) = (\bl^*; \br^*_B)$ and $\bij_{\ell,e}((\bm;\br)^{w_{0}}) = (\bm^*; \br^*_C)$.
\end{thm}

\begin{proof}
If 
$\Sc(B) = \Sc(C)$, then $B$ and $C$ are Scopes equivalent to initial core blocks $B_k$ and $C_l$ respectively, with $\Sc(B_k) = \Sc(B) = \Sc(C) = \Sc(C_l)$ by Proposition \ref{P:Scopes}.  Since
$$\wt(B_k) = |\mv(B_k)| = |\mv(B)| = |\mv(C)| = |\mv(C_l)| = \wt(C_l)$$
by Proposition \ref{P:moving-vector-Weyl-orbit},  it follows from Lemma \ref{L:unique-j-initial} that $B_k = C_l$.  The correspondence between the Specht modules of $B$ and $C$ then follows from Proposition \ref{P:Scopes}.

For the converse, it suffices to show that if $C = \cs_j \DDot{} B$, and $\bl$ has no addable node with $(e,\br)$-residue $j$ for all $\bl \in \PP^{\ell}$ lying in $B$, then $\Sc(B) = \Sc(C)$.
Let $\bl \in \PP^{\ell}$ be lying in $B$.  Since $\bl$ has no addable node with $(e,\br)$-residue $j$, $\bl^{w_0}$ has no addable node with $(e,\br^{w_0})$-residue $j$.  Hence $\hub_j((\bl;\br)^{w_0}) \geq 0$.
Thus $0 \leq \hub_j((\bl;\br)^{w_0}) = \hub_j(B) = x^B_{j+1} - x^B_{\overline{j}} - \delta_{j0}\ell$.
If $x^B_{j+1} = x^B_{\overline{j}} + \delta_{j0} \ell$, then $B = C$, and the conclusion is trivial.
Thus we may assume that $y^B_{j+1} > y^B_{\overline{j}} + \delta_{j0}$, or $y^B_{j+1} =y^B_{\overline{j}} + \delta_{j0}$ and $z^B_{j+1} > z^B_{\overline{j}}$.
If the former holds, then we have $\Sc(B) = \Sc(C)$ by Proposition \ref{P:Scopes-vector}.
If the latter holds, then we have $\Ht_{I_B}(z^B_{j+1}) \geq z^B_j$ by Proposition \ref{P:not-Scopes}, so that $\Sc(B) = \Sc(C)$ by Proposition \ref{P:Scopes-vector} again.
\end{proof}

\begin{prop} \label{P:number-Scopes-inequivalent}
Let $B$ be a core block of $\HH_n$.
For each $a \in \ZZ/\ell\ZZ$, let $k_a = |\{j \in [1,\,e] \mid z^B_j = a \}|$.
Let $I_B = \{ i_1,\dotsc, i_m \}$ where $0 = i_1 < i_2 < \dotsb < i_m$ (recall that $I_B = \{ \ell - i \mid \mvi_i(B) = 0\}$). Then the $\AW_e$-orbit of $B$ contains exactly
$$
\prod_{b=1}^m \frac{(\sum_{a= i_b + 1}^{i_{b+1}-1} k_a)!}{\prod_{a=i_b + 1}^{i_{b+1}-1} k_a!}$$
Scopes equivalence classes, where $i_{b+1}$ is to be read as $\ell$ when $b = m$.
\end{prop}

\begin{proof}
Let
$$
X = \{ (\bz^B)^{\rho} \mid \rho \in \sym{e},\ j < j' \text{ if } z^B_{\rho(j)} > z^B_{\rho(j')} \text{ and } 
\Ht_{I_B}(z^B_{\rho(j)}) \geq z^B_{\rho(j')} \}.
$$
If $\Ht_{I_B}(z^B_{\rho(j)}) = i_b$, then $z^B_{\rho(j)} > z^B_{\rho(j')}$ and $\Ht_{I_B}(z^B_{\rho(j)}) \geq z^B_{\rho(j')}$ if and only if ($z^B_{\rho(j)} \in [i_b+1,i_{b+1}-1]$ and $z^B_{\rho(j')} \leq i_b$) or ($i_b = z^B_{\rho(j)} > z^B_{\rho(j')}$).
Thus $X$ contains precisely all rearrangements of $\bz^B$ such that the $z^B_j$'s lying in $[i_b+1,\, i_{b+1}-1]$ appears before those which are equal to $i_b$, which in turn appears before those lying in $[i_{b-1}+1,i_{b}-1]$, for each $b \in [2,\, m]$.
Since there are exactly
$$
\frac{(\sum_{a= i_b + 1}^{i_{b+1}-1} k_a)!}{\prod_{a=i_b + 1}^{i_{b+1}-1} k_a!}$$
distinct ways of arranging the $\bz^B_j$'s lying in $[i_b+1,\, i_{b+1}-1]$ for each $b \in [1,m]$, we see that
$$
|X| = \prod_{b=1}^m \frac{(\sum_{a= i_b + 1}^{i_{b+1}-1} k_a)!}{\prod_{a=i_b + 1}^{i_{b+1}-1} k_a!}.
$$

Now, for each $w \in \AW_e$, let $B_w = w \DDot{} B$.  Then $\bz^{B_w} = (\bz^B)^{\overline{w}^{-1}}$ by Lemma \ref{L:yz}, so that $$\Sc(B_w) = (\bz^{B_w})^{\sigma_{B_w}\tau_{B_w}} = (\bz^B)^{\overline{w}^{-1}\sigma_{B_w}\tau_{B_w}}.$$
Let $\rho_w = \overline{w}^{-1}\sigma_{B_w}\tau_{B_w}$.  Then $\rho_w \in \sym{e}$, and
applying Lemma \ref{L:tau} to $\tau_{B_w} = \tau^{I_{B_w}}_{(\bz^{B_w})^{\sigma_{B_w}}} = \tau^{I_B}_{(\bz^B)^{\overline{w}^{-1}\sigma_{B_w}}}$, we get $j < j'$ whenever $z^B_{\rho_w(j)} > z^B_{\rho_w(j')}$ and $\Ht_{I_B}(z^B_{\rho_w(j)}) \geq z^B_{\rho_w(j')}.$
Consequently, $\Sc(B_w) = (\bz^B)^{\rho_w} \in X$.

For the converse, let $(\bz^B)^{\rho} \in X$ where $\rho \in \sym{e}$.
Let $B_0$ be an initial core block that is Scopes equivalent to $B$.
Then $(\bz^B)^{\rho} = \Sc(B_0)^{\rho'} = (\bz^{B_0})^{\sigma_{B_0} \rho'}$ for some $\rho' \in \sym{e}$.
Since $w \mapsto \overline{w}$ is an isomorphism from $\left< \cs_j \mid j \in \ZZ/e\ZZ \setminus \{j_0 \} \right>$ to $\sym{e}$ for all $j_0 \in \ZZ/e\ZZ$,
let $w' \in \left< \cs_j \mid j \in \ZZ/e\ZZ \setminus \{j_{B_0} \} \right>$ such that $\overline{w'} = \sigma_{B_0}\rho'^{-1}\sigma_{B_0}^{-1}$.
Let $B' = w' \DDot{} B_0$.
Then $\by^{B'} = w' \CDot{1} \by^{B_0} = \by^{B_0}$ by Lemma \ref{L:yz} and Lemma \ref{L:equivalent}(3) since $w' \in \left< \cs_j \mid j \in \ZZ/e\ZZ \setminus \{j_{B_0} \} \right>$, so that
$\sigma_{B'} = \sigma_{\by^{B'}} = \sigma_{\by^{B_0}} = \sigma_{B_0}$.
Thus,
$$(\bz^{B'})^{\sigma_{B'}} = (\bz^{B_0})^{\overline{w'}^{-1}\sigma_{B'}} =
(\bz^{B_0})^{(\sigma_{B_0} \rho' \sigma_{B_0}^{-1}) \sigma_{B_0}} =
(\bz^{B_0})^{\sigma_{B_0} \rho'} = (\bz^B)^{\rho}.$$
As $I_{B'} = I_{B_0} = I_{B}$ by Proposition \ref{P:moving-vector-Weyl-orbit}, we have
$$ \tau_{B'} = \tau^{I_{B'}}_{(\bz^{B'})^{\sigma_{B'}}} = \tau^{I_{B}}_{(\bz^{B})^{\rho}} = 1_{\sym{e}}$$
by Corollary \ref{C:tau=1}(2),
since $(\bz^B)^{\rho} \in X$.
Consequently,
$$
\Sc(B') = (\bz^{B'})^{\sigma_{B'}\tau_{B'}} = (\bz^{B'})^{\sigma_{B'}} = (\bz^B)^{\rho}.$$

We have thus shown that $X = \{ \Sc(w \DDot{} B) \mid w \in \AW_e \}$.
By Theorem \ref{T:Scopes} and Proposition \ref{P:moving-vector-Weyl-orbit}, each distinct $\Sc(w \DDot{} B)$ indexes a distinct Scopes equivalence class in $\AW_e \DDot{} B$, and so there are exactly $|X|$ distinct Scopes equivalence classes in $\AW_e \DDot{} B$.
\end{proof}



\begin{cor} \label{C:type-B}
Let $\ell =2$. All core blocks in a $\AW_e$-orbit are Scopes equivalent.
Equivalently, all core blocks with the same moving vector (with respect to any $\br' \in \br^{\EW_{2}} \cap \overline{\mathcal{A}}^2_e$) are Scopes equivalent.
\end{cor}

\begin{proof}
Let $B$ be a core block.
Then $I_B = \{0\}$ or $I_B = \{ 0,1\}$.
In both cases, we apply Proposition \ref{P:number-Scopes-inequivalent} to conclude that the $\AW_e$-orbit of $B$ has only one Scopes equivalence class.
\end{proof}

We end this section with a proof that all core blocks of Ariki-Koike algebras are Scopes equivalent to Rouquier blocks.
Rouquier blocks were first defined for the group algebras of the symmetric groups and Schur algebras by Chuang and the third author in \cite{CT}.
These blocks are very well-behaved, and they have since been generalised to other algebras such as the group algebras of finite general linear groups, Iwahori-Hecke algebras of type $A$ and $q$-Schur algebras.
In \cite{Lyle-RoCK}, Lyle generalises these blocks to Ariki-Koike algebras.
She defines a block $B$ of $\HH_n$ to be Rouquier if and only if
for all $\bl = (\lambda^{(1)},\dotsc, \lambda^{(\ell)}) \in \PP^{\ell}$ lying in $B$,
with
$\bfs(\quot_e(\beta_{r_i}(\lambda^{(i)}))) = (s^{\bl,i}_1,\dotsc, s^{\bl,i}_e)$ for all $i\in [1,\,\ell]$,
then $s^{\bl,i}_{j+1} \geq s^{\bl,i}_{j} + \wt_e(\lambda^{(i)})-1$ for all $i \in [1,\,\ell]$ and $j \in [1,\,e-1]$.

\begin{cor}[{cf.\ \cite[Theorem C]{MNSS}}]
  Every core block of $\HH_n$ is Scopes equivalent to a Rouquier block as defined by Lyle in \cite{Lyle-RoCK}.
\end{cor}

\begin{proof}
%
Suppose first that $\by^B$ is strictly increasing, i.e.\ $y^B_1 < y^B_2 < \dotsb < y^B_e$.
Let $\bl = (\lambda^{(1)},\dotsc, \lambda^{(\ell)}) \in \PP^{\ell}$ be lying in $B$ and let
$\beta_{\br^{w_0}}(\bl^{w_0})  = (\B^{(1)},\dotsc, \B^{(\ell)})$.
Let $\quot_e(\UU(\beta_{\br^{w_0}}(\bl^{w_0})) = (\C_1,\dotsc, \C_e)$ and for each $a \in [1,\,\ell]$, let
$\quot_e(\B^{(a)}) = (\C^{(a)}_1, \dotsc, \C^{(a)}_e)$.
Then $\fs(\C_j) = \sum_{a=1}^{\ell} \fs(\C^{(a)}_j)$ by \cite[Lemma 2.6(2)]{LT}.

By Lemma \ref{L:core-block-abacus} (with $i = \ell$), $\C_j = \ZZ_{< y^B_j \ell} \cup L_j$ for some $L_j \subseteq [y^B_j\ell,\, y^B_j\ell + e-1]$ for all $j \in [1,\,e]$.
Consequently, for each $a \in [1,\,\ell]$ and $j \in [1,\,e]$,
$
\C^{(a)}_j = \ZZ_{< y^B_j} \cup M^a_j$
where $M^a_j =\{ y^B_j \}$ if $ y^B_j\ell + (\ell - a) \in L_j$ and $M^a_j = \emptyset$ otherwise, so that
$$
y^B_j \leq \fs(\C^{(a)}_j) \leq y^B_j+1.$$
Note that if $w_0 = \sigma\bt $ where $\bt = (t_1,\dotsc, t_{\ell}) \in \ZZ^{\ell}$ and $\sigma \in \sym{\ell}$, then
$$\B^{(a)} = \beta_{r_{\sigma(a)} + et_{a}}(\lambda^{(\sigma(a))}) = (\beta_{r_{\sigma(a)}}(\lambda^{(\sigma(a)})))^{+et_a},$$
so that
$$\quot_e(\beta_{r_{\sigma(a)}}(\lambda^{(\sigma(a)})))
= \quot_e((\B^{(a)})^{+(-et_a)}) = ((\C^{(a)}_1)^{+(-t_a)},
\dotsc, (\C^{(a)}_{e})^{+(-t_a)})
$$
by \eqref{E:quot}.
Hence, if $\bfs(\quot_e(\beta_{r_{a}}(\lambda^{(a)}))))
= (s^{\bl,a}_1,\dotsc, s^{\bl,a}_e)$ for $a \in [1,\,\ell]$, then
\begin{align*}
s^{\bl,\sigma(a)}_{j+1} - s^{\bl,\sigma(a)}_j &= \fs((\C^{(a)}_{j+1})^{+(-t_a)}) - \fs((\C^{(a)}_j)^{+(-t_a)}) \\
&= \fs(\C^{(a)}_{j+1}) - \fs(\C^{(a)}_{j}) \geq y^B_{j+1} - y^B_j - 1 \\
&\geq 0 = 
\wt_e(\lambda^{(\sigma(a))})
\end{align*}
for all $a \in [1,\,\ell]$ and $j \in [1,\,e-1]$, so that $B$ is a Rouquier block.

Now for a general core block $B$, let $\by = (y_1,\dotsc, y_e) \in \ZZ^{\ell}$ be such that $y_1 < \dotsb < y_e$ and $|\by| = |\by^B|$.
Then $|\by - \by^B| = |\by| - |\by^B| = 0$, so that $\by - \by^B \in \AW_e$ by \eqref{E:affineWeylgroup}.
Thus $$\by = (\by - \by^B) + \by^B = (\by - \by^B) \CDot{1} \by^B \in \AW_e \CDot{1} \by^B.$$
Let $w \in \AW_e$ be the element of least length such that $w \CDot{1} \by^B = \by$, and
let $\cs_{j_m} \cs_{j_{m-1}} \dotsm \cs_{j_1}$ be a reduced expression for $w$.
Let $B_0 = B$, and
for each $a \in [1,\, m]$, 
let $B_a = \cs_{j_a} \cs_{j_{a-1}} \dotsm \cs_{j_1} \DDot{} B$.
Then $\by^{B_a} = \by^{\cs_{j_a} \DDot{}B_{a-1}} = \cs_{j_a} \CDot{1} \by^{B_{a-1}}$ by Lemma \ref{L:yz},
so that if $y^{B_{a-1}}_{j_a+1} = y^{B_{a-1}}_{\overline{j_a}} + \delta_{\overline{j_a},e}$,
then $\by^{B_a} = \by^{B_{a-1}}$, and hence $\cs_{j_m} \dotsm \widehat{\cs_{j_a}} \dotsm \cs_{j_1} \CDot{1} \by^B = \by$, contradicting the minimality of the length of $w$.
Thus, $y^{B_{a-1}}_{j_a+1} \ne y^{B_{a-1}}_{\overline{j_a}} + \delta_{\overline{j_a},e}$, so that
$\B_{a-1}$ and $\B_{a}$ are Scopes equivalent by Lemma \ref{L:Scopes}.
Since this is true for all $a \in [1,\,m]$, $B$ is Scopes equivalent to $B_{m}$.
As $\by^{B_{m}} = \by^{w \DDot{} B} = w \CDot{1} \by^B = \by$ is strictly increasing, $B_{m}$ is Rouquier and our proof is complete.
\end{proof}

\section{Simple modules and decomposition numbers of core blocks} \label{S:simple-decomp}

Core blocks satisfying Condition (I) of being initial are very special.
In this section, we exploit their properties to express the number of simple modules lying in a core block as a classical Kostka number---the number of semistandard generalised tableaux of a particular shape and a particular type---and relate the graded decomposition numbers of core blocks when $\mathrm{char}(\FF) =0$ with some corresponding graded decomposition numbers for Iwahori-Hecke algebras of type $A$.

\begin{prop} \label{P:I-2}
Let $B$ be a core block of $\HH_n$ satisfying Condition (I) of being initial.
Let $\bl \in \PP^{\ell}$ be lying in $B$ with $\beta_{\br^{w_{0}}}(\bl^{w_{0}}) = (\B^{(1)},\dotsc, \B^{(\ell)})$.
For each $a \in [1,\ell]$, we have
$$
\min(\ZZ \setminus \B^{(a)}) \geq |\by^B| \qquad
\text{and} \qquad
\max(\B^{(a)}) < |\by^B|+e.
$$
%
Consequently, $\B^{(a)} = \ZZ_{<|\by^B|} \cup L_a$ for some $L_a \subseteq [|\by^B|, |\by^B| + e-1]$ with $|L_a| = r'_a - |\by^B|$, where $\br^{w_0} = (r'_1,\dotsc, r'_{\ell})$.
\end{prop}

\begin{proof}
Set $(\C_1,\dotsc, \C_e) := \quot_e(\UU(\beta_{\br^{w_{0}}}(\bl^{w_{0}}))) =\quot_e(\UU(\B^{(1)},\dotsc, \B^{(\ell)}))$.
Let $x \in \ZZ$, and write $x = b_x e + r_x$, where $b_x \in \ZZ$ and $r_x \in \ZZ/e\ZZ$.
By Lemma \ref{L:core-block-abacus} with $i = \ell$, for each $j \in [1,\, e]$,
$x \in \C_j$ for all $x < y^B_j \ell $ while $x \notin \C_j$ for all $x \geq (y^B_j +1) \ell $.


Let $a \in [1,\, \ell]$.

Suppose that $x < |\by^B|$.
Then
$ x < \cy_Be + j_B$ by Definition \ref{D:yz},
so that either
$b_x \leq \cy_B - 1 $
or ($b_x =  \cy_B $ and $r_x < j_B$).
In the former case, we have
\begin{align*}
b_x\ell + \ell -a &\leq (\cy_B - 1)\ell + \ell-a
= (y^B_{r_x+1}-\bbone_{r_x+1 \leq j_B})\ell - a
< y^B_{r_x+1}\ell.
\end{align*}
where the equality in the middle follows from Lemma \ref{L:equivalent}(2).
In the latter case, we similarly have
\begin{align*}
b_x\ell + \ell -a &= \cy_B \ell + \ell-a
= y^B_{r_x+1} \ell - a
< y^B_{r_x+1}\ell.
\end{align*}
Consequently, $b_x\ell + \ell -a \in \C_{r_x+1}$ in all cases, so that
$$
\uu_a(x) = b_x\ell e + (\ell-a)e + r_x = (b_x\ell + \ell - a)e + r_x \in \UU(\B^{(1)},\dotsc,\B^{(\ell)})
= \bigcup_{i=1}^{\ell} \uu_i(\B^{(i)}), $$
and hence $x \in \B^{(a)}$.
Thus $\min(\ZZ \setminus \B^{(a)}) \geq |\by^B|$.

Now suppose that $x \geq |\by^B| + e$.  Then
$x \geq (\cy_B + 1) e + j_B,$
so that either $b_x \geq \cy_B + 2$ or ($b_x = \cy_B + 1$ and $r_x \geq j_B$).
In the former case, we have
\begin{align*}
b_x\ell+\ell-a &\geq (\cy_B + 2)\ell + \ell - a
= (y^B_{r_x+1} + 1)\ell + (1- \bbone_{r_x+1 \leq j_B}) \ell + (\ell  - a) \geq (y^B_{r_x+1} + 1)\ell.
\end{align*}
In the latter case, we have
\begin{align*}
b_x\ell+\ell-a &= (\cy_B + 1)\ell + \ell - a = (y^B_{r_x+1} + 1)\ell + (\ell  - a) \geq (y^B_{r_x+1} + 1)\ell .
\end{align*}
Consequently, $b_x\ell + \ell -a \notin \C_{r_x+1}$ in all cases, so that
$$
\uu_a(x) = b_x\ell e + (\ell-a)e + r_x = (b_x\ell + \ell - a)e + r_x \notin \UU(\B^{(1)},\dotsc,\B^{(\ell)}) = \bigcup_{i=1}^{\ell} \uu_i(\B^{(i)})
$$
and hence $x \notin \B^{(a)}$.
Thus $\max(\B^{(a)}) < |\by^B| + e$ as desired.

The last assertion follows from Lemma \ref{L:hub}(2).
\end{proof}

\begin{cor} \label{C:I}
Let $B$ be a core block of $\HH_n$ satisfying Condition (I) of being initial.
For each $i \in [1,\, \ell]$, let $t_i \in \ZZ/e\ZZ$ be such that $t_i \equiv_e r_i - |\by^B|$.
If $\bl = (\lambda^{(1)},\dotsc, \lambda^{(\ell)}) \in \PP^{\ell}$ lies in $B$, then
$$
\beta_{t_i}(\lambda^{(i)}) = \ZZ_{<0} \cup L^{\bl}_i
$$
for some $L^{\bl}_i \subseteq [0,e-1]$ such that $|L^{\bl}_i| = t_i$.
\end{cor}

\begin{proof}
Let $w_0 = \sigma \bu$, where $\sigma \in \sym{\ell}$ and $\bu = (u_1,\dotsc, u_{\ell}) \in \ZZ^{\ell}$.
Then $\br^{w_0} = (r_{\sigma(1)} + eu_1, \dotsc, r_{\sigma(\ell)} + e u_{\ell})$ and $\bl^{w_0} = (\lambda^{(\sigma(1))}, \dotsc, \lambda^{(\sigma(\ell))})$.  Thus, by Proposition \ref{P:I-2}, for each $a \in [1,\ell]$,
$$
\beta_{r_{\sigma(a)}+eu_a}(\lambda^{(\sigma(a))}) = \ZZ_{<|\by^B|} \cup L_a
$$
for some $L_a \subseteq [|\by^B|, |\by^B|+e-1]$, with $|L_a| = r_{\sigma(a)} + eu_a - |\by^B|$.
Now,
$$[0,\, e] \ni |L_a| = r_{\sigma(a)} + eu_a - |\by^B| \equiv_e t_{\sigma(a)} \in \ZZ/e\ZZ,$$
so that $t_{\sigma(a)} = |L_a| = r_{\sigma(a)} + eu_a - |\by^B|$ unless $|L_a| = e$, in which case $L_a = [|\by^B|, |\by^B| + e-1]$ and $t_{\sigma(a)} = r_{\sigma(a)} + eu_a - |\by^B| -e$ .
Define
$$
L^{\bl}_{\sigma(a)} :=
\begin{cases}
 L_a^{+(-|\by^B|)}, &\text{if } |L_a| < e; \\
 \emptyset, &\text{otherwise.}
\end{cases}
$$
Then $L^{\bl}_{\sigma(a)} \subseteq [0,e-1]$ with $|L^{\bl}_{\sigma(a)}| = t_{\sigma(a)}$, and
\begin{align*}
\ZZ_{<0} \cup L_{\sigma(a)}^{\bl} &=
\begin{cases}
(\ZZ_{<|\by^B|} \cup L_a)^{+(-|\by^B|)}  &\text{if } |L_a| < e \\
(\ZZ_{<|\by^B|} \cup L_a)^{+(-|\by^B|-e)}     &\text{otherwise}
\end{cases} \\
&=
\begin{cases}
(\beta_{r_{\sigma(a)} + eu_a-|\by^B|}(\lambda^{(\sigma(a))}))  &\text{if } |L_a| < e \\
(\beta_{r_{\sigma(a)} + eu_a-|\by^B| -e}(\lambda^{(\sigma(a))}))     &\text{otherwise}
\end{cases} \\
&= \beta_{t_{\sigma(a)}}(\lambda^{(\sigma(a))}).
\end{align*}
\end{proof}

Recall that for $(\bl;\bt) \in \PP^{\ell} \times \ZZ^{\ell}$ and $j \in \ZZ/e\ZZ$, $\AR^{e,\bt}_j(\bl)$ denotes the set of addable and removable nodes of $\bl$ with $(e,\bt)$-residue $j$, and that for $(a,b,i) \in \AR^{e,\bt}_j(\bl)$, $\cont_{\bt}(a,b,i) = b-a + t_i$ where $\bt = (t_1,\dotsc, t_{\ell})$.

\begin{lem} \label{L:AR}
Let $\bl = (\lambda^{(1)},\dotsc, \lambda^{(\ell)}) \in \PP^{\ell}$ and let $\bt = (t_1,\dotsc, t_{\ell}) \in \ZZ^{\ell}$.
Suppose that there exists $m \in \ZZ$ such that for each $i \in [1,\,\ell]$, $\beta_{t_i}(\lambda^{(i)}) = \ZZ_{< m} \cup L_i$ for some subset $L_i \subseteq [m,m+e-1]$.
\begin{enumerate}
\item $\bl$ has no node with $(e,\bt)$-residue $\equiv_e m$; in particular it has no removable node with $(e,\bt)$-residue $\equiv_e m$.

\item For all $j \in \ZZ/e\ZZ$ with $j \not\equiv_e m$ and $\mathfrak{n} \in \AR^{e,\bt}_j(\bl)$, $\cont_{\bt}(\mathfrak{n}) = m+j_m$, where $j_m\in \ZZ/e\ZZ$ with $j_m \equiv_e j-m$.
\end{enumerate}
\end{lem}

\begin{proof} \hfill
\begin{enumerate}
\item If $\bl$ has a node with $(e,\bt)$-residue $\equiv_e m$, then there exist $x,y, z \in \ZZ$ with $x < y \leq z$ such that $x \notin \beta_{t_i}(\lambda^{(i)})$, $y \equiv_e m$ and $z \in \beta_{t_i}(\lambda^{(i)})$ for some $i \in [1,\,\ell]$.
But if $x \notin \beta_{t_i}(\lambda^{(i)})$ and $z \in \beta_{t_i}(\lambda^{(i)})$, then $x \geq m$ and $z \leq m+e-1$ since $\ZZ_{< m} \subseteq \beta_{t_i}(\lambda^{(i)}) \subseteq \ZZ_{< m+e}$, so that for all $y \in [x+1,\, z] \subseteq [m+1,m+e-1]$, $y \not\equiv_e m$.

\item Let $\mathfrak{n} = (a,b,i) \in \AR^{e,t}_j (\bl)$ where $j \not\equiv_e m$.
Then exactly one of $\cont_{\bt}(\mathfrak{n})$ and $\cont_{\bt}(\mathfrak{n})-1$ lies in $\beta_{t_i}(\lambda^{(i)}) = \ZZ_{< m} \cup L_i$.
Since $\cont_{\bt}(\mathfrak{n}) \equiv_e j \not\equiv_e m$, we must have $j \equiv_e \cont_{\bt}(\mathfrak{n}) \in [m+1,m+e-1]$.
Thus $\cont_{\bt}(\mathfrak{n}) = m+ j_m$, where $j_m \in [1,\,e-1]$ such that $j_m \equiv_e j-m$.
\end{enumerate}
\end{proof}

Given $(\bl;\bt) \in \PP^{\ell} \times \ZZ^{\ell}$ and $j \in \ZZ/e\ZZ$, there are two total orders on the set $\AR^{e,\bt}_j(\bl)$ that are of interest:
\begin{align*}
(a,b,i) \succeq_{\bt} (a',b',i')\ &\Leftrightarrow\ \cont_{\bt}(a,b,i) > \cont_{\bt}(a',b',i') \ \vee\
(\cont_{\bt}(a,b,i) = \cont_{\bt}(a',b',i')\ \wedge\ i \leq i'); \\
(a,b,i) \unrhd_{\bt} (a',b',i')\ &\Leftrightarrow\ i < i'\ \vee\
(i = i'\ \wedge\ \cont_{\bt}(a,b,i) \geq \cont_{\bt}(a',b',i')).
\end{align*}

It is easy to see that if $\bt \equiv_e \bu$, then $\unrhd_{\bt} = \unrhd_{\bu}$.
Furthermore, $\unrhd_{\bt}\ =\ \succeq_{\bt}$ when $t_i \gg t_{i+1}$ for all $i \in [1,\,\ell-1]$.
In addition, when $\ell = 1$, $\succeq_t\ =\ \succeq_u\ =\ \unrhd_u\ =\ \unrhd_t$ for all $t, u \in \ZZ$.

The {\em $j$-signature} of $\bl$ associated to the order $\geq\ \in \{\succeq_{\bt},\unrhd_{\bt}\}$ is the sequence of $+$ and $-$ signs obtained
by examining the nodes in $\AR^{e,\bt}_j(\bl)$ in ascending order with respect to $\geq$ and writing a $+$ (respectively $-$) for
each addable (respectively removable) node. From this, the {\em reduced $j$-signature} is obtained by successively deleting all adjacent pairs $-+$.
If there are any $ - $ signs in the reduced $j$-signature, the corresponding removable nodes of $\bl$ with $(e,\bt)$-residue $j$ are called {\em normal} and the smallest among these, with respect to $\geq$, is called {\em good}.

If by repeatedly removing good nodes (possibly with different $(e,\bt)$-residues) with respect to $\succeq_{\bt}$ (resp.\ $\unrhd_{\bt}$) from $\bl$ we reach $\EP$, we say that $(\bl;\bt)$ is {\em Uglov} (resp.\ {\em Kleshchev}).

Kleshchev $\ell$-partitions play an important role in the representation theory of Ariki-Koike algebras as they index the latter's simple modules. Indeed,
if $(\bl;\br)$ is Kleshchev, then the associated Specht module
$\Sp^{\bl}$ has a simple head $D^{\bl}$, and
$\{ D^{\bl} \mid \bl \in \PP^{\ell}(n),\, (\bl;\br) \text{ Kleshchev} \}$ is a complete set of pairwise non-isomorphic simple $\HH_n$-modules.

The problem of finding a non-recursive description for all Kleshchev and Uglov partitions is currently still open.
This problem has however been solved for Uglov $\ell$-partitions with $\bt \in \AAbar$, now commonly known as FLOTW $\ell$-partitions.
A non-recursive description of such $\ell$-partitions is given in \cite{Jacon-Kleshchev-multipartitions}.

\begin{thm}[{see \cite[Proposition 4.2.1]{Jacon-Kleshchev-multipartitions}}] \label{T:FLOTW}
Let $\bl = (\lambda^{(1)},\dotsc, \lambda^{(\ell)}) \in \PP^{\ell}$ and $\bt = (t_1,\dotsc, t_{\ell}) \in \AAbar$.  Then $(\bl;\bt)$ is Uglov if and only if
\begin{enumerate}
\item $\lambda^{(i)}_a \geq \lambda^{(i+1)}_{a+t_{i+1}-t_i}$ for all $i \in [1,\,\ell-1]$ and $a \in \ZZ^+$;
\item $\lambda^{(\ell)}_a \geq  \lambda^{(1)}_{a+e + t_1 - t_{\ell}}$ for all $a \in \ZZ^+$;
\item $\{ \res_e^{\bt}(a,\lambda^{(i)}_a, i) \mid i \in [1,\, \ell],\ a \in \ZZ^+, \, \lambda^{(a)}_i = k \} \ne \ZZ/e\ZZ$ for all $k \in \ZZ^+$.
\end{enumerate}
\end{thm}

Since the $a$-th largest element of $\beta_{t_i}(\lambda^{(i)})$ equals $\lambda^{(i)}_a - a + t_i = \cont_{\bt}(a,\lambda^{(i)}_a,i)$, we have the following equivalent characterisation of FLOTW partitions in terms of $\beta$-sets, which is more useful for us:

\begin{cor} \label{C:FLOTW}
Let $\bl = (\lambda^{(1)},\dotsc, \lambda^{(\ell)}) \in \PP^{\ell}$ and $\bt = (t_1,\dotsc, t_{\ell}) \in \AAbar$.  Then $(\bl;\bt)$ is Uglov if and only if
\begin{enumerate}
\item $a$-th largest element of $\beta_{t_i}(\lambda^{(i)}) \geq $
$(a+t_{i+1}-t_i)$-th largest element of $\beta_{t_{i+1}}(\lambda^{(i+1)})$ for all $i \in [1,\,\ell-1]$ and $a \in \ZZ^+$;
\item $a$-th largest element of $\beta_{t_{\ell}}(\lambda^{(\ell)}) \geq e\ + $
$(a+e + t_{1}-t_{\ell})$-th largest element of $\beta_{t_1}(\lambda^{(1)})$ for all $a \in \ZZ^+$;
\item $\{ \res_e^{\bt}(a,\lambda^{(i)}_a, i) \mid i \in [1,\, \ell],\ a \in \ZZ^+, \, \lambda^{(a)}_i = k \} \ne \ZZ/e\ZZ$ for all $k \in \ZZ^+$.
\end{enumerate}
\end{cor}

For $\eta \in \PP$, a {\em generalised $\eta$-tableau} is a function $\mathsf{T} : [\eta] \to \ZZ^+$.  Furthermore, $\mathsf{T}$ is {\em column semistandard} if $\mathsf{T}(i,j) < \mathsf{T}(i,j')$ and $\mathsf{T}(i,j) \leq \mathsf{T}(i',j)$ for all $(i,j),(i,j'), (i',j) \in [\eta]$ with $i < i'$ and $j < j'$.
Furthermore, the {\em type} of $\mathsf{T}$ is the composition $(|\mathsf{T}^{-1}\{1\}|, |\mathsf{T}^{-1}\{2\}|,\dotsc )$.

\begin{prop} \label{P:FLOTW}
Let $\bl = (\lambda^{(1)},\dotsc, \lambda^{(\ell)}) \in \PP^{\ell}$ and let $\bt = (t_1,\dotsc, t_{\ell}) \in \AAbar$.
Suppose that there exists $m \in \ZZ$ such that for each $i \in [1,\,\ell]$, $\beta_{t_i}(\lambda^{(i)}) = \ZZ_{< m} \cup L_i$ for some subset $L_i \subseteq [m,m+e-1]$.
Let $\T^{\bl}$ denote the generalised $(|L_{\ell}|, |L_{\ell-1}|, \dotsc, |L_1|)$-tableau $\T^{\bl}$, defined by
$$
\T^{\bl}(i,j) = 1-m\ + \text{ $j$-th least element of } L_{\ell-i+1}.
$$
\begin{enumerate}
\item The type of $\T^{\bl}$ is $(k_1,\dotsc, k_e)$ where $k_j = |\{ a \in [1,\,\ell] : m+j-1 \in L_a \}|$ for all $j \in [1,\,e]$.
\item $(\bl;\bt)$ is Uglov if and only if $\T^{\bl}$ is column semistandard.
\end{enumerate}
\end{prop}

\begin{proof} \hfill
\begin{enumerate}
\item This is immediate from the definition of $\T^{\bl}$.

\item Note first that $|L_i| - |L_{i'}| = t_i - t_{i'}$ for all $i,i' \in [1,\, \ell]$ by Lemma \ref{L:hub}(2).

For each $i \in [1,\,\ell]$, the $a$-th largest element of $\beta_{t_i}(\lambda^{(i)})$ is the $a$-th largest element of $L_i$ if $a \leq |L_i|$, and equals $m - a + |L_i|$ if $a > |L_i|$.
    Now $a > |L_i|$ if and only if
    $$
    a+ t_{i+1} - t_{i} = a + |L_{i+1}| - |L_{i}| > |L_{i+1}|,
    $$
    so that when these hold, we have
    \begin{align*}
    \text{$a$-th largest element of $\beta_{t_i}(\lambda^{(i)})$}\ &=
    m-a + |L_i| = m- (a - t_i + t_{i+1}) + |L_{i+1}| \\
    &=
    \text{$(a + t_{i+1} - t_i)$-th largest element of $\beta_{t_{i+1}}(\lambda^{(i+1)})$}.
    \end{align*}
    Thus Condition (1) in Corollary \ref{C:FLOTW} is reduced to: for all $i \in [1,\,\ell-1]$
    $$
    \text{$a$-th largest element of $L_i$}\ \geq\
    \text{$(a+ t_{i+1} - t_i)$-th largest element of $L_{i+1}$}
    $$
    for all $a \in [1,\,|L_i|]$,
    which is equivalent to
    $$\T^{\bl}(\ell-i+1, |L_i|+1-a) \geq \T^{\bl}(\ell-i,|L_i|+1-a)$$
    for all $a \in [1,\,|L_i|]$.
    Replacing $i$ and $a$ with $\ell -i$ and $|L_i|+1-a$ respectively, the condition becomes $$\T^{\bl}(i,a) \leq \T^{\bl}(i+1,a)$$ for all $i \in [1,\,\ell-1]$ and $a \in [1,|L_i|]$.

Since $a+e + t_1 - t_{\ell} = a+ e + |L_1| - |L_{\ell}| > |L_1|$,
\begin{align*}
(a+e + t_1 - t_{\ell})&\text{-th largest element of $\beta_{t_1}^{(\lambda^{(1)})}$}\ = m - a -e + |L_{\ell}| \\
&=\ \text{$(e + a)$-th largest element of $\beta_{t_{\ell}}(\lambda^{(\ell)})$} \\
&\leq \text{$a$-th largest element of $\beta_{t_{\ell}}(\lambda^{(\ell)})$} -e.
\end{align*}
Hence Condition (2) in Corollary \ref{C:FLOTW} holds for $\bl$.

Condition (3) in Corollary \ref{C:FLOTW} also holds for $\bl$, since $\bl$ has no node with $(e,\bt)$-residue $\equiv_e m$ by Lemma \ref{L:AR}(1).

Thus, $\bl$ is Uglov if and only if $\T^{\bl}(i,a) \leq \T^{\bl}(i+1,a)$ for all $i \in [1,\,\ell-1]$ and $a \in [1,\, |L_i|]$, if and only if $\T^{\bl}$ is column semistandard, since $\T^{\bl}(i,b) < \T^{\bl}(i,b+1)$ for all $i \in [1,\,\ell]$ and $b \in [1,\,|L_i|-1]$ follows from the definition of $\T^{\bl}$.
\end{enumerate}
\end{proof}

Recall that $\AR(\bl)$ denotes the set of addable and removable nodes of $\bl$ for all $\bl \in \PP^{\ell}$.

\begin{lem} \label{L:I}
Let $\bl = (\lambda^{(1)},\dotsc, \lambda^{(\ell)}) \in \PP^{\ell}$ and
$\bt = (t_1,\dotsc, t_{\ell} ) \in \ZZ^{\ell}$.
Suppose that there exists $m \in \ZZ$ such that 
for each $i \in [1,\,\ell]$, $\beta_{t_i}(\lambda^{(i)}) = \ZZ_{< m} \cup L_i$ for some $L_i \subseteq [m,\, m+e+1]$.
\begin{enumerate}
\item For all $\mathfrak{n} = (a,b,i), \mathfrak{n}'= (a',b',i') \in \AR(\bl)$ with $\res^{\bt}_e(\mathfrak{n}) = \res^{\bt}_e(\mathfrak{n}') \not\equiv_e m$,
    $$
    \mathfrak{n} \succeq_{\bt} \mathfrak{n}'
    \ \Leftrightarrow\ 
    i \leq i'
    \ \Leftrightarrow\ \mathfrak{n} \unrhd_{\bt} \mathfrak{n}'.
    $$
\item $(\bl;\bt)$ is Kleshchev if and only if $(\bl;\bt)$ is Uglov.
\end{enumerate}
\end{lem}

\begin{proof} \hfill
\begin{enumerate}

\item
By Lemma \ref{L:AR}(2), $\cont_{\bt}(\mathfrak{n}) = \cont_{\bt}(\mathfrak{n}')$ if and only if $\res_e^{\bt}(\mathfrak{n}) = \res^{\bt}_e(\mathfrak{n}')$.
Consequently,
$$
\mathfrak{n} \succeq_{\bt} \mathfrak{n}' \ \Leftrightarrow\  
i \leq i' \ \Leftrightarrow\ \mathfrak{n} \unrhd_{\bt} \mathfrak{n}'.
$$

\item
We prove by induction on $|\bl|$.
By Lemma \ref{L:AR}(1) and part (1), a removable node of $\bl$ is good with respect to $\succeq_{\bt}$ if and only if it is good with respect to $\unrhd_{\bt}$.
Consequently, $\bl$ has no good node with respect to $\succeq_{\bt}$ if and only if $\bl$ has no good node with respect to $\unrhd_{\bt}$, in which case $(\bl;\bt)$ is neither Kleshchev nor Uglov.
Otherwise, let $\mathfrak{n} = (a,b,i)$ be a good node of $\bl$ with respect to both $\unrhd_{\bt}$ and $\succeq_{\bt}$, and let
$\bm = (\mu^{(1)},\dotsc, \mu^{(\ell)}) := \bl^{-\mathfrak{n}}$.
Then
\begin{align*}
\beta_{t_i}(\mu^{(i)}) &= \beta_{t_i}(\lambda^{(i)}) \setminus \{ \cont_{\bt}(a,b,i) \} \cup \{ \cont_{\bt}(a,b,i) - 1\} \\
&= \ZZ_{< m} \cup L_i \setminus \{ \cont_{\bt}(a,b,i) \} \cup \{ \cont_{\bt}(a,b,i) - 1\}
= \ZZ_{< m} \cup M_i
\end{align*}
where $M_i = L_i \setminus \{ \cont_{\bt}(a,b,i) \} \cup \{ \cont_{\bt}(a,b,i) - 1\} \subseteq [m,m+e-1]$,
while $\beta_{t_{i'}}(\mu^{(i')}) = \beta_{t_{i'}}(\lambda^{(i')}) = \ZZ_{< m}\cup L_{i'}$ for all $i' \in [1,\,\ell]$ with $i'\ne i$.
By induction, $(\bm;\bt)$ is Kleshchev if and only if $(\bm;\bt)$ is Uglov.
Thus, $(\bl;\bt)$ is Kleshchev if and only if $(\bl;\bt)$ is Uglov.
\end{enumerate}
\end{proof}

For $\eta \in \PP$, its conjugate $\eta'$ is the partition with $[\eta'] = \{ (j,i) \mid (i,j) \in [\eta] \}$.
A generalised $\eta$-tableau $\T$ is {\em semistandard} if and only if $\T'$, defined by $\T'(j,i) = \T(i,j)$, is a column semistandard generalised $\eta'$-tableau.
Semistandard generalised tableaux appears as composition multiplicities of permutation modules of symmetric groups.
More precisely, the number of semistandard generalised $\eta$-tableaux of type $\zeta$ equals the composition multiplicity of the Specht module $S^{\eta}$ in the permutation module $M^{\zeta}$ in characteristic zero \cite[Theorem 14.1]{James-SLN}.
If $\tilde{\zeta}$ is a composition obtained by rearranging the parts of $\zeta$, then it is well known that $M^{\zeta}$ and $M^{\tilde{\zeta}}$ are isomorphic,
so that the composition multiplicities of $S^{\eta}$ in $M^{\zeta}$ and in $M^{\tilde{\zeta}}$ are the same.
Consequently, the number of semistandard generalised $\eta$-tableaux of type $\zeta$ equals that of type $\tilde{\zeta}$.
Since the semistandard generalised $\eta$-tableaux are in a type-preserving bijective correspondence with the column semistandard generalised $\eta'$-tableaux,
we have that the number of column semistandard generalised $\eta$-tableaux of type $\zeta$ equals that of type $\tilde{\zeta}$ too.
We shall use this result freely in the proof of the following theorem.

\begin{thm} \label{T:number-simple}
Let $B$ be a core block of $\HH_n$, and write $\br^{w_{0}} = (r'_1,\dotsc, r'_{\ell})$.
Then the number of simple $\HH_n$-modules lying in $B$ equals the number of generalised column semistandard $\eta^B$-tableaux of type $\bz^B$, where
$\eta^B= (\eta^B_1,\dotsc, \eta^B_{\ell}) \in \PP$ with
$
\eta^B_i = r'_{\ell+1-i} - |\by^B|
$
for all $i \in [1,\,\ell]$.
\end{thm}

\begin{proof}
Since $\HH_{\FF,q,\br}(n) = \HH_{\FF,q,\br^{w_0}}(n)$, by replacing $\br$ with $\br^{w_0}$, we may assume that $w_0 = 1_{\EW_{\ell}}$.

We first assume that $B$ is an initial core block.
For each $\bl = (\lambda^{(1)},\dotsc, \lambda^{(\ell)}) \in \PP^{\ell}$ lying in $B$ with $\beta_{\br}(\bl) = (\B^{(1)},\dotsc, \B^{(\ell)})$, and each $i \in [1,\, \ell]$, we have
$\B^{(i)} = \ZZ_{<|\by^B|} \cup L_i$
for some $L_i \subseteq [|\by^B|,\,|\by^B|+e-1]$ with $|L_i| = r'_i - |\by^B| = \eta_{\ell +1 -i}$ by Proposition \ref{P:I-2}.
If
$\quot_e(\UU(\beta_{\br}(\bl))) = (\C_1,\dotsc, \C_e)$, then for each $j \in [1,\,e]$,
\begin{align*}
\C_j &=
\left\{ \frac{x-j+1}{e}\ \bigg|\ x \in \UU(\beta_{\br}(\bl)), \, x \equiv_e j-1 \right\}
= \left\{ \frac{x-j+1}{e}\ \bigg|\ x \in \bigcup_{i=1}^{\ell} \uu_i(\B^{(i)}), \, x \equiv_e j-1 \right\} \\
&= \bigcup_{i=1}^{\ell}  \left\{ \frac{x-j+1}{e} \ \bigg|\  x \in \uu_i(\ZZ_{<|\by^B|} \cup L_i),\, x \equiv_e j-1 \right\} \\
&= \bigcup_{i=1}^{\ell} \left\{ \frac{\uu_i(x') - j+1}{e} \ \bigg|\ x' \in \ZZ_{<|\by^B|} \cup L_i, \, x' \equiv_e j-1 \right\} \\
&= \bigcup_{i=1}^{\ell} \left( \left\{ a\ell + \ell - i \ \bigg|\ a \in \ZZ,\, ae+j-1 < |\by^B| \right\} \cup \left\{ \frac{\uu_i(x') - j+1}{e} \ \bigg|\ x' \in L_i, \, x' \equiv_e j-1 \right\}\right).
\end{align*}
Since $|\by^B| = \cy_Be + j_B$ by Definition \ref{D:yz}, we see that
the necessary and sufficient condition for $ae+j-1 <|\by^B|$ is $a \leq \cy_B$ when $j-1 < j_B$, and $a \leq \cy_B -1$ when $j-1 \geq j_B$.
Consequently, when $j \leq j_B$, we have
\begin{align}
\C_j &= \bigcup_{i=1}^{\ell} \left( \left\{ a\ell + \ell -i \ \bigg|\ a \in \ZZ_{\leq \cy_B} \right\} \cup \left\{ \frac{\uu_i(x') - j+1}{e} \ \bigg|\ x' \in L_i, \, x' \equiv_e j-1 \right\}\right) \notag \\
&= \ZZ_{< (\cy_B+1)\ell} \cup \bigcup_{i=1}^{\ell}\left\{ \frac{\uu_i(x') - j+1}{e} \ \bigg|\ x' \in L_i, \, x' \equiv_e j-1 \right\} \notag \\
&= \ZZ_{< y^B_j\ell} \cup \bigcup_{i=1}^{\ell} \left\{ \frac{\uu_i(x') - j+1}{e} \ \bigg|\ x' \in L_i, \, x' \equiv_e j-1 \right\} \label{E:C_j-1}
\end{align}
by Lemma \ref{L:equivalent}(2).
On the other hand, when $j > j_B$, we similarly have
\begin{align}
\C_j &= \bigcup_{i=1}^{\ell} \left( \left\{ a\ell + \ell -i \ \bigg|\ a \in \ZZ_{\leq \cy_B-1} \right\} \cup \left\{ \frac{\uu_i(x') - j+1}{e} \ \bigg|\ x' \in L_i, \, x' \equiv_e j-1 \right\}\right) \notag \\
&= \ZZ_{< \cy_B\ell} \cup \bigcup_{i=1}^{\ell}\left\{ \frac{\uu_i(x') - j+1}{e} \ \bigg|\ x' \in L_i, \, x' \equiv_e j-1 \right\} \notag \\
&= \ZZ_{< y^B_j\ell} \cup \bigcup_{i=1}^{\ell} \left\{ \frac{\uu_i(x') - j+1}{e} \ \bigg|\ x' \in L_i, \, x' \equiv_e j-1 \right\}. \label{E:C_j-2}
\end{align}
From \eqref{E:C_j-1} and \eqref{E:C_j-2}, since the unions there are disjoint unions, we get
\begin{align*}
y^B_{j} \ell + z^B_j = x^B_j = \fs(\C_j)
&= y^B_j \ell + \sum_{i=1}^{\ell} |\{ x' \in L_i \mid x' \equiv_e j-1 \}|,
\end{align*}
so that
$$z^B_j = \sum_{i=1}^{\ell} |\{ x' \in L_i \mid x' \equiv_e j-1 \}|
= |\{ i \in [1,\ell] \mid |\by^B| + \tilde{j} \in L_i \}|,$$
where $\tilde{j} \in \ZZ/e\ZZ$ satisfies $|\by^B| + \tilde{j} \equiv_e j-1$.
Consequently, $\T^{\bl}$ defined in Proposition \ref{P:FLOTW} is a generalised $\eta$-tableau of type
$(z^B_{j_B+1}, z^B_{j_B+2}\dotsc, z^B_{e}, z^B_{1},\dotsc, z^B_{j_B}) = (\bz^B)^{\sigma_B}$ by Proposition \ref{P:FLOTW}(1) and Corollary \ref{C:I-sigma_B}.
By Lemma \ref{L:I} and Proposition \ref{P:FLOTW}(2), $(\bl;\br)$ is Kleshchev if and only if $\T^{\bl}$ is column semistandard.
Since $\eta$ depends only on $\br$ and $|\by^B|$, and is independent of $\bl$,
the number of simple modules lying in $B$ equals the number of column semistandard generalised $\eta$-tableaux of type
$(\bz^B)^{\sigma_B}$ and hence of type $\bz^B$.

For a general core block $B$, let $B_0$ be the initial core block that is Scopes equivalent to $B$ which exists by Proposition \ref{P:Scopes} and is unique by Lemma \ref{L:unique-j-initial}.
Then $B = w \DDot{} B_0$ for some $w \in \AW_e$, and $B$ and $B_0$ has the same number of simple modules.
By Lemma \ref{L:yz}, we have
\begin{align*}
  |\by^B| = |\by^{B_0}| \qquad \text{and} \qquad \bz^B = (\bz^{B_0})^{\overline{{w}^{-1}}}.
\end{align*}
In particular, $\eta^B = \eta^{B_0}$.
Since the number of column semistandard generalised $\eta^B$-tablaux of type $\bz^B$ equals that of type $(\bz^B)^{\sigma}$ for any $\sigma \in \sym{e}$, the Theorem follows.
\end{proof}

\begin{cor}[{cf.\ \cite[Theorem 3.20(2)]{Lyle}}]
Let $\ell =2$, and let $B$ be a core block of $\HH_n$.  Then the number of simple modules lying in $B$ equals
$$
\binom{|\bz^B|}{r'_1 - |\by^B|} - \binom{|\bz^B|}{r'_1 - |\by^B|-1},
$$
where $\br^{w_0} = (r'_1,r'_2)$.
\end{cor}

\begin{proof}
Let $a = r'_1 - |\by^B|$ and $b = r'_2 - |\by^B|$.
By Theorem \ref{T:number-simple}, the number of simple modules lying in $B$ equals the number of column semistandard generalised $(b,a)$-tableaux of type $\bz^B$.
Since $\ell =2$, $z^B_j \in \{0,1\}$ for all $j \in [1,\,e]$.  Consequently,
a generalised $(b, a)$-tableau of type $\bz^B$ is column semistandard if and only if it is a standard $(b, a)$-tableau.
It is well known that the number of standard $(b, a)$-tableaux is equal to the dimension of the Specht module $S^{(b, a)}$ of the symmetric group (or Iwahori-Hecke algebra of type $A$), which is known to be $\binom{a+b}{a} - \binom{a+b}{a-1}$ (see, for example, \cite[Example 14.4]{James-SLN}).
To complete the proof, we only need to observe that $(b, a)$ and $\bz^B$ have the same size (so that we can talk about $(b, a)$-tableaux of type $\bz^B$).
\end{proof}

The final main result of this paper is to connect the graded decomposition numbers of the core blocks of Ariki-Koike algebras with those of Iwahori-Hecke algebras of type $A$, when $\mathrm{char}(\FF) = 0$.
We do this through the $v$-decomposition numbers arising from the canonical bases of the Fock space representations of the quantum affine algebra $\Uv$, which we now proceed to give a brief account.

Let $v$ be an indeterminate. The quantum affine algebra $\Uv$ is a unital $\mathbb{C}(v)$-algebra generated by $\{ e_j, f_j, k_j, k_j^{-1} \mid j \in \ZZ/e\ZZ \} \cup \{D,D^{-1}\}$ subject to certain relations which we do not
need here (interested readers may refer to \cite[Section 4]{Leclerc} for more details).
The level $\ell$ Fock space representation $\CF^{\ell}$ of $\Uv$ is a $\CC(v)$-vector space with $\PP^\ell \times \ZZ^{\ell}$ as its $\CC(v)$-basis.
For each $\bt \in \ZZ^{\ell}$, the $\CC(v)$-subspace $\CF_{\bt}$ spanned by $\PP^{\ell} \times \{ \bt \}$ is invariant under the $\Uv$-action, so that $\CF^{\ell}$ decomposes naturally into a direct sum of the $\CF_{\bt}$'s, i.e.\ $\CF^{\ell} = \bigoplus_{\bt \in \ZZ^{\ell}} \CF_{\bt}$.
For our purposes, we require only the precise description of the action of the generator $f_j$ of $\Uv$ on $\CF_{\bt}$, which is
$$
f_j(\bl;\bt) = \sum_{\mathfrak{n} \in \AS^{e,\bt}_j(\bl)}
             v^{N_{\bt}^{\bl}(\mathfrak{n})} (\bl^{+\mathfrak{n}};\bt)
$$
where 
$
N_{\bt}^{\bl}(\mathfrak{n})
= |\{ \mathfrak{n}' \in \AS(\bl) \mid 
\mathfrak{n}' \rhd_{\bt} \mathfrak{n} \}|
- |\{ \mathfrak{n}' \in \RS(\bl) \mid 
\mathfrak{n}' \rhd_{\bt} \mathfrak{n} \}|.
$ 
(Recall that $\AS(\bl)$ (resp.\ $\RS(\bl)$) denotes the set of addable (resp.\ removable) nodes of $\bl$, and $\AS^{e,\bt}_j(\bl)$ is the subset of $\AS(\bl)$ consisting of nodes with $(e,\bt)$-residue $j$.)

\begin{prop} \label{P:Fock}
Let $\bl = (\lambda^{(1)},\dotsc, \lambda^{(\ell)}) \in \PP^{\ell}$ and
$\bt = (t_1,\dotsc, t_{\ell} ) \in \ZZ^{\ell}$.
Suppose that there exists $m \in \ZZ$ such that 
for each $i \in [1,\,\ell]$, $\beta_{t_i}(\lambda^{(i)}) = \ZZ_{< m} \cup L_i$ for some $L_i \subseteq [m,\, m+e+1]$.
Let $\bu = \bt - m\bone$, and for each $\bnu \in \PP^{\ell}$, let $\Phi_{\bu}(\bnu) := \beta^{-1}(\UU(\beta_{\bu}(\bnu)))$.
Let $j \in \ZZ/e\ZZ$ with $j \not\equiv_e m$, and let $j_m \in \ZZ/e\ZZ$ with $j_m \equiv_e j-m$.

\begin{enumerate}
\item
There is a bijective function $\phi: \AR^{e,\bt}_j(\bl) \to \AR^{e,|\bu|}_{j_m}(\Phi_{\bu}(\bl))$ with the following properties for all $\mathfrak{n}, \mathfrak{n}' \in \AR^{e,\bt}_j(\bl)$:
\begin{enumerate}
\item $\uu_i(\cont_{\bu}(\mathfrak{n})) = \cont_{|\bu|}(\phi(\mathfrak{n}))$ if $\mathfrak{n} = (a,b,i)$;
\item $\mathfrak{n}$ is addable if and only if $\phi(\mathfrak{n})$ is addable, in which case $\Phi_{\bu}(\bl^{+\mathfrak{n}}) = (\Phi_{\bu}(\bl))^{+\phi(\mathfrak{n})}$;
\item $\mathfrak{n}$ is removable if and only if $\phi(\mathfrak{n})$ is removable, in which case $\Phi_{\bu}(\bl^{-\mathfrak{n}}) = (\Phi_{\bu}(\bl))^{-\phi(\mathfrak{n})}$;
\item $\mathfrak{n} \unrhd_{\bt} \mathfrak{n}'$ if and only if $\phi(\mathfrak{n}) \unrhd_{|\bu|} \phi(\mathfrak{n}')$.
\end{enumerate}
\item
Define $\Psi_{\bt,\bu} : \CF_{\bt} \to \CF_{|\bu|}$ by $\Psi_{\bt,\bu}(\bm;\bt) = (\Phi_{\bu}(\bm); |\bu|)$ for all $\bm \in \PP^{\ell}$ and extending it linearly.  Then
\begin{align*}
\Psi_{\bt,\bu}(f_j(\bl;\bt)) = f_{j_m}(\Psi_{\bt,\bu}(\bl;\bt))
= f_{j_m}(\Phi_{\bu}(\bl);|\bu|).
\end{align*}
\end{enumerate}
\end{prop}

\begin{proof} \hfill
\begin{enumerate}
\item
Since $\bt = \bu + m\bone$, we have
$\res^{\bt}_e (a,b,i) \equiv_e \res^{\bu}_e (a,b,i) + m,
$
so that $\AR^{e,\bt}_j (\bl) = \AR^{e,\bu}_{j_m} (\bl)$.
Consequently, Lemma \ref{L:ARl1} provides the bijection $\phi$ satisfying (a)--(c).

For part (d), let $\mathfrak{n} = (a,b,i), \mathfrak{n}' = (a', b', i') \in \AR^{e,\bt}_j(\bl)$.
By Lemma \ref{L:AR}(2), $\cont_{\bt}(\mathfrak{n}) = \cont_{\bt}(\mathfrak{n}')$, so that
$\cont_{\bu}(\mathfrak{n}) = \cont_{\bu}(\mathfrak{n}')$ since $\bu = \bt - m\bone$.
If $\mathfrak{n} \unrhd_{\bt} \mathfrak{n}'$, then
$i \leq i'$ by Lemma \ref{L:I}(1) .
Consequently,
$$
\cont_{|\bu|}(\phi(\mathfrak{n}))
= \uu_i(\cont_{\bu}(\mathfrak{n}))
= \uu_i(\cont_{\bu}(\mathfrak{n}'))
\geq \uu_{i'}(\cont_{\bu}(\mathfrak{n}'))
= \cont_{|\bu|}(\phi(\mathfrak{n}')),
$$
so that $\phi(\mathfrak{n}) \unrhd_{|\bu|} \phi(\mathfrak{n}')$.
Conversely, if $\phi(\mathfrak{n}) \unrhd_{|\bu|} \phi(\mathfrak{n}')$, then
$\cont_{|\bu|}(\phi(\mathfrak{n}))
\geq
\cont_{|\bu|}(\phi(\mathfrak{n}'))
$.
Thus
$$
\uu_i(\cont_{\bu}(\mathfrak{n}))
= \cont_{|\bu|}(\phi(\mathfrak{n}))
\geq \cont_{|\bu|}(\phi(\mathfrak{n}'))
= \uu_{i'}(\cont_{\bu}(\mathfrak{n}'))
= \uu_{i'}(\cont_{\bu}(\mathfrak{n}))
$$
so that $i \leq i'$ and hence $\mathfrak{n} \unrhd_{\bt} \mathfrak{n}'$ by Lemme \ref{L:I}(1).

\item
For $\mathfrak{n} \in \AS^{e,\bt}_j(\bl)$, we have by part (1),
\begin{align*}
N^{\bl}_{\bt}(\mathfrak{n}) &=
|\{ \mathfrak{n}' \in \AS(\bl) \mid
\mathfrak{n}' \rhd_{\bt} \mathfrak{n} \}|
- |\{ \mathfrak{n}' \in \RS(\bl) \mid
\mathfrak{n}' \rhd_{\bt} \mathfrak{n} \}| \\
&= |\{ \phi(\mathfrak{n}') \in \AS(\Phi_{\bu}(\bl)) \mid
\phi(\mathfrak{n}') \rhd_{|\bu|} \phi(\mathfrak{n}) \}|
- |\{ \phi(\mathfrak{n}') \in \RS(\Phi_{\bu}(\bl)) \mid
\phi(\mathfrak{n}') \rhd_{|\bu|} \phi(\mathfrak{n}) \}| \\
&= |\{ \mathfrak{m} \in \AS(\Phi_{\bu}(\bl)) \mid
\mathfrak{m} \rhd_{|\bu|} \phi(\mathfrak{n}) \}|
- |\{ \mathfrak{m} \in \RS(\Phi_{\bu}(\bl)) \mid
\mathfrak{m} \rhd_{|\bu|} \phi(\mathfrak{n}) \}| \\
&= N^{\Phi_{\bu}(\bl)}_{|\bu|}(\phi(\mathfrak{n})).
\end{align*}
Thus,
\begin{align*}
\Psi_{\bt,\bu}(f_j(\bl;\bt)) &= \Psi_{\bt,\bu} \left( \sum_{\mathfrak{n} \in \AS^{e,\bt}_j(\bl)}
             v^{N_{\bt}^{\bl}(\mathfrak{n})} (\bl^{+\mathfrak{n}};\bt) \right)
             = \sum_{\mathfrak{n} \in \AS^{e,\bt}_j(\bl)}
             v^{N_{\bt}^{\bl}(\mathfrak{n})} (\Phi_{\bu}(\bl^{+\mathfrak{n}});|\bu|) \\
             &= \sum_{\phi(\mathfrak{n}) \in \AS^{e,|\bu|}_{j_m}(\Phi_{\bu}(\bl))}
             v^{N_{|\bu|}^{\Phi_{\bu}(\bl)}(\phi(\mathfrak{n}))} ((\Phi_{\bu}(\bl))^{+\mathfrak{\phi(n)}};|\bu|) \\
             &= f_{j_m}(\Phi_{\bu}(\bl);|\bu|) = f_{j_m}(\Psi_{\bt,\bu}(\bl;\bt)).
\end{align*}

\end{enumerate}
\end{proof}

The Fock space possesses a bar involution $x \mapsto \overline{x}$ (see for example \cite[Subsection 3.12]{BK-GradedDecompNos}) which satisfies in particular
$\overline{a(v) x + y} = a(v^{-1}) \overline{x} + \overline{y}$ and $\overline{f_j(x)} = f_j \overline{x}$ for all $a(v) \in \CC(v)$, $x,y \in \CF_{\bt}$ and $j \in \ZZ/e\ZZ$.
There is a distinguished basis $\{ G(\bl;\bt) : \bl \in \PP^{\ell} \}$ for $\CF_{\bt}$, called the {\em canonical basis}, which has the following characterising properties:
\begin{itemize}
  \item $\overline{G(\bl;\bt)} = G(\bl;\bt)$;
  \item $G(\bl;\bt) - (\bl;\bt) \in \sum_{\bm \in \PP^{\ell}} v\ZZ[v] (\bm;\bt)$.
\end{itemize}

For $\bl,\bm \in \PP^{\ell}$ and $\bt \in \ZZ^{\ell}$, define $d_{\bl\bm}^{\bt}(v) \in \CC(v)$ by
$$
G(\bm;\bt) = \sum_{\bl \in \PP^{\ell}} d_{\bl\bm}^{\bt}(v) (\bl;\bt).$$
These {\em $v$-decomposition numbers} $d_{\bl\bm}^{\bt}(v)$ enjoy the following remarkable properties.
\begin{enumerate}
  \item \cite[Remark 3.19]{BK-GradedDecompNos} $d^{\bt}_{\bm\bm}(v) = 1$ and $d^{\bt}_{\bl\bm}(v) \in v\ZZ_{\geq 0}[v]$ when $\bl\ne \bm$.
  \item \cite{Ariki} When $(\bm;\br)$ is Kleshchev, then $d^{\br}_{\bl\bm}(1)$ equals the composition multiplicity of $D^{\bm}$ in the Specht module $\Sp^{\bl}$ of $\HH_n$ in characteristic $0$.  (In fact, more is true: the $v$-decomposition numbers are the graded decomposition numbers of the Ariki-Koike algebras \cite[Corollary 5.15]{BK-GradedDecompNos}.)
\end{enumerate}
In particular, if $d^{\br}_{\bl\bm}(v) \ne 0$, then $\bl$ and $\bm$ lies in the same block of $\HH_n$ and $\bm \unrhd \bl$.
(Here, $\bm = (\mu^{(1)},\dotsc, \mu^{(\ell)}) \unrhd \bl = (\lambda^{(1)}, \dotsc, \lambda^{(\ell)})$ if and only if
$\sum_{a=1}^{i-1} |\mu^{(a)}| + \sum_{b = 1}^j \mu^{(i)}_b \geq
\sum_{a=1}^{i-1} |\lambda^{(a)}| + \sum_{b = 1}^j \lambda^{(i)}_b$
for all $i \in [1,\,\ell]$ and $j \in \ZZ^+$.)

We note that $d^{\bt}_{\bl\bm}(v) = d^{\bu}_{\bl\bm}(v)$ when $\bt \equiv_e \bu$.
Furthermore, when $\ell = 1$,
$d^t_{\lambda\mu}(v) = d^u_{\lambda\mu}(v)$ for all $t,u \in \ZZ$, so that
we may write $d_{\lambda\mu}(v)$ without ambiguity.

\begin{prop}\label{P:decomp}
Let $B$ be a core block of $\HH_n$, and suppose that $w_0 \in \ZZ^{\ell}$ and that $B$ satisfies Condition (I) of being initial.
Let $\bu = \br^{w_0} - |\by^B|\bone$, and for each $\bnu \in \PP^{\ell}$, let $\Phi_{\bu}(\bnu) = \beta^{-1}(\UU(\beta_{\bu}(\bnu)))$. Let $\Psi_{\br,\bu} : \CF_{\br} \to \CF_{|\bu|}$ be the $\mathbb{C}(v)$-linear map with $\Psi_{\br,\bu}(\bl;\br) = (\Phi_{\bu}(\bl);|\bu|)$ for all $\bl \in \PP^{\ell}$.
If $\bm$ is an  $\ell$-partition lying in $B$ such that $(\bm;\br)$ is Kleshchev,
then $\Phi_{\bu}(\bm)$ is $e$-regular and
$$
\Psi_{\br,\bu}(G(\bm;\br)) = G(\Phi_{\bu}(\bm);|\bu|).
$$
In particular $
d^{\br}_{\bl \bm}(v)
= d_{\Phi_{\bu}(\bl) \Phi_{\bu}(\bm)} (v)
$
for all $\bl \in \PP^{\ell}$, and $d_{\lambda, \Phi_{\bu}(\bm)}(v) = 0$ if $\lambda \notin \Psi_{\bu}(\PP^{\ell})$.
\end{prop}

\begin{proof}
Since $w_0 \in \ZZ^{\ell}$, we have $\bl^{w_0} = \bl$ for all $\bl \in \PP^{\ell}$, and $\br \equiv_e \br^{w_0}$, so that $\unrhd_{\br}\ = \ \unrhd_{\br^{w_0}}$ on $\AR(\bl)$ for all $\bl \in \PP^{\ell}$.  In particular, $(\bl;\br)$ is Kleshchev if and only if $(\bl;\br^{w_0})$ is Kleshchev.
Furthermore, $d^{\br^{w_0}}_{\bl\bm}(v) = d^{\br}_{\bl\bm}(v)$ for all $\bl,\bm \in \PP^{\ell}$.  As such, it suffices to prove the Theorem with $\br$ replaced by $\br^{w_0}$.

We prove by induction.
If $|\bm| = 0$, then $\bm = \EP$.
Thus $(\bm;\br^{w_0})$ is Kleshchev and $G(\bm;\br^{w_0}) = (\bm;\br^{w_0})$.
Since $\br^{w_0} \in \AAbar$, $\bu = \br^{w_0} - |\by^B|\bone \in \AAbar$, and so $\Phi_{\bu}(\bm) = \beta^{-1}(\UU(\beta_{\bu}(\EP)))$
is an $e$-core partition by \cite[Proposition 2.22]{JL-cores}.
Hence $\Phi_{\bu}(\bm)$ is $e$-regular, and $G(\Phi_{\bu}(\bm);|\bu|) = (\Phi_{\bu}(\bm);|\bu|)$ (see, for example, \cite[Proposition 11(i)]{Leclerc}).
Thus
$$\Psi_{\br^{w_0},\bu}(G(\bm;\br^{w_0})) = \Psi_{\br^{w_0},\bu} (\bm;\br^{w_0}) = (\Phi_{\bu}(\bm);|\bu|) = G(\Phi_{\bu}(\bm);|\bu|).$$

Now suppose that $(\bm;\br^{w_0})$ is Kleshchev and $|\bm| > 0$,
and assume that the statement holds for all $\bnu \in \PP^{\ell}$ with $|\bnu| < |\bm|$, or $|\bnu| = |\bm|$ and $\bnu \lhd \bm$.
By Proposition \ref{P:I-2}, for each $i \in [1,\,\ell]$,
$\beta_{r'_i}(\mu^{(i)}) = \ZZ_{< |\by^B|} \cup M_i$ for some
$M_i \subseteq [|\by^B|,\, |\by^B|+e-1]$, where $\br^{w_0} = (r'_1, \dotsc, r'_{\ell})$.
Since $(\bm;\br^{w_0})$ is Kleshchev, it is also Uglov by Lemma \ref{L:I}(2), so that the generalised $(|M_{\ell}|, |M_{\ell-1}|,\dotsc, |M_1|)$-tableau $\T^{\bm}$, 
defined by $\T^{\mu}(a,b) = b$-th least element of $M_{\ell-a+1}$, is column semistandard by Proposition \ref{P:FLOTW}.
Let $$i_0 := \min\{ i \in [1,\,\ell] \mid |\mu^{(i)}| > 0\},$$
and let $\mathfrak{m}$ be the topmost removable node of $\mu^{(i_0)}$.
Then for each $i \in [1,\,i_0 -1]$, $\mu^{(i)} = \varnothing$ so that $\beta_{r'_i}(\mu^{(i)}) = \ZZ_{<r'_i}$.
Since $\T^{\bm}$  is column semistandard, we thus have
$\ZZ_{<r'_i} \subseteq \beta_{r'_{i_0}}(\mu^{(i_0)})$, so that
$$\cont_{\br^{w_0}}(\mathfrak{m}) -1 \geq \min(\ZZ \setminus \beta_{r'_{i_0}}(\mu^{(i_0)}))
\geq r'_i = \max(\beta_{r'_i}(\mu^{(i)}))+1,$$
for all $i\in [1,\,i_0-1]$.
Together with Lemma \ref{L:AR}(2), we conclude that
there does not exist $\mathfrak{m}' \in \AR(\bm)$ with $\mathfrak{m}' \succ_{\br^{w_0}} \mathfrak{m}$.
Consequently, there does not exist $\mathfrak{m}' \in \AR(\bm)$ such that $\mathfrak{m}' \rhd_{\br^{w_0}} \mathfrak{m}$
by Lemmas \ref{L:AR}(1) and \ref{L:I}(1).
In particular, $\mathfrak{m}$ is a good node of $\bm$ with respect to $\unrhd_{\br^{w_0}}$.

Let $j = \res^{\br^{w_0}}_e(\mathfrak{m})$, and let $\tilde{j} \in \ZZ/e\ZZ$ with $\tilde{j} \equiv_e j - |\by^B|$.
By Lemma \ref{L:AR}(1), $j \not\equiv_e |\by^B|$, and so $\tilde{j} \ne 0$.
By Proposition \ref{P:Fock}(1), there exists a bijection
$\phi : \AR^{e,\br^{w_0}}_j(\bm) \to \AR^{e,|\bu|}_{\tilde{j}}(\Phi_{\bu}(\bm))$ satisfying, for all
$\mathfrak{n}, \mathfrak{n}' \in \AR^{e,\br^{w_0}}_j(\bm)$,
$\phi(\mathfrak{n})$ is removable if and only if $\mathfrak{n}$ is removable, in which case $\Phi_{\bu}(\bl^{-\mathfrak{n}}) = (\Phi_{\bu}(\bl))^{-\phi(\mathfrak{n})}$, and
$\phi(\mathfrak{n}) \unrhd_{|\bu|} \phi(\mathfrak{n}')$ if and only if $\mathfrak{n} \unrhd_{\br^{w_0}} \mathfrak{n}'$.
Thus,
$\phi(\mathfrak{m})$ is a good node of $\Phi_{\bu}(\bm)$ with respect to $\unrhd_{|\bu|}$.

Let $\bnu = (\nu^{(1)},\dotsc, \nu^{(\ell)}) := \bm^{-\mathfrak{m}}$.
Then $\beta_{r'_a}(\nu^{(a)}) = \beta_{r'_a}(\mu^{(a)}) = \ZZ_{<|\by^B|} \cup M_a$ for all $a \in [1,\ell] \setminus \{ i_0 \}$,
while
\begin{align*}
\beta_{r'_{i_0}}(\nu^{(i_0)}) &= \beta_{r'_{i_0}}(\mu^{(i_0)}) \setminus \{ \cont_{\br^{w_0}}(\mathfrak{m}) \} \cup \{ \cont_{\br^{w_0}}(\mathfrak{m}) -1\} \\
&= \ZZ_{<|\by^B|} \cup (M_{i_0} \setminus \{ \cont_{\br^{w_0}}(\mathfrak{m}) \} \cup \{ \cont_{\br^{w_0}}(\mathfrak{m}) -1\}).
\end{align*}
From this, we conclude that the block $B'$ of $\HH_{n-1}$ in which $\bnu$ lies is a core block with $\mvi^{w_0}_{\ell}(B') = 0$.
Furthermore,
$\by^{B'} = \by^{B}$ (while $z^{B'}_a = z^{B}_a$ for all $a \in [1,\,e] \setminus \{ \overline{j}, j+1 \}$, and $z^{B'}_{j+1} = z^B_{j+1} - 1$ and $z^{B'}_{\overline{j}} = z^B_{\overline{j}} + 1$).
Thus $B'$ also satisfies Condition (I) of being initial.
Consequently, by Proposition \ref{P:I-2}, for all $\bl = (\lambda^{(1)},\dotsc, \lambda^{\ell}) \in \PP^{\ell}$ lying in $B'$ and $i \in [1,\ell]$,
$\beta_{r'_i}(\lambda^{(i)}) = \ZZ_{< |\by^B|} \cup L^{\bl}_i$ for some $L^{\bl}_i \subseteq [|\by^B|,\, |\by^B|+e-1]$.
In particular,
\begin{equation}
\Psi_{\br^{w_0},\bu}(f_j (\bl;\br^{w_0})) = f_{\tilde{j}}(\Psi_{\br^{w_0},\bu} (\bl;\br^{w_0})) \label{E:Psi}
\end{equation}
for all $\bl \in \PP^{\ell}$ lying in $B'$ by Proposition \ref{P:Fock}(2).
Note also that $\AR^{e,\br^{w_0}}_j(\bnu) = \AR^{e,\br^{w_0}}_j(\bm)$.

As $\mathfrak{m}$ is the largest node (with respect to $\unrhd_{\br^{w_0}}$) in $\AR^{e,\br^{w_0}}_j(\bm)= \AR^{e,\br^{w_0}}_j(\bnu)$,
we see that
$$f_j (G(\bnu;\br^{w_0})) = f_j(\bnu;\br^{w_0}) + \sum_{\tilde{\bl} \lhd \bnu} d^{\br^{w_0}}_{\tilde{\bl} \bnu}(v)\, f_j (\tilde{\bl},\br^{w_0})
=
(\bm;\br^{w_0})+
\sum_{\substack{\bl \in \PP^{\ell} \\ \bl \lhd \bm}} c_{\bl}(v) (\bl;\br^{w_0})$$
where $c_{\bl}(v) \in \ZZ_{\geq 0}[v,v^{-1}].$
Consequently,
$$
f_j (G(\bnu;\br^{w_0})) = G(\bm;\br^{w_0}) +  \sum_{\substack{\bl \in \PP^{\ell} \\ \bl \lhd \bm}} a_{\bl}(v) G(\bl;\br^{w_0})
$$
where $a_{\bl}(v) \in \ZZ_{\geq 0} [v,v^{-1}]$ with $a_{\bl}(v^{-1}) = a_{\bl}(v)$.
Applying $\Psi_{\br^{w_0},\bu}$ to this last equation and using \eqref{E:Psi}, we get
\begin{align*}
f_{\tilde{j}} (\Psi_{\br^{w_0},\bu} (G(\bnu;\br^{w_0})) =
\Psi_{\br^{w_0},\bu} (f_j (G(\bnu;\br^{w_0}))) = \Psi_{\br^{w_0},\bu} (G(\bm;\br^{w_0})) +
\sum_{\substack{\bl \in \PP^{\ell} \\ \bl \lhd \bm}} a_{\bl}(v)\, \Psi_{\br^{w_0}} (G(\bl;\br^{w_0})).
\end{align*}
We then apply induction hypothesis to $\bnu$ and $\bl$ with $\bl \lhd \bm$ to get
\begin{align*}
f_{\tilde{j}} (G(\Phi_{\bu}(\bnu);|\bu|)) = \Psi_{\br^{w_0},\bu} (G(\bm;\br^{w_0})) +
\sum_{\substack{\bl \in \PP^{\ell} \\ \bl \lhd \bm}} a_{\bl}(v)\,  G(\Phi_{\bu}(\bl);|\bu|)).
\end{align*}
Since
$\Psi_{\br^{w_0},\bu} (G(\bm;\br^{w_0})) - (\Phi_{\bu}(\bm);|\bu|) =
\sum_{\substack{\bl \in \PP^{\ell} \\ \bl \ne \bm}} d^{\br^{w_0}}_{\bl\bm}(v)\, (\Phi_{\bu}(\bl);|\bu|) \in \sum_{\lambda \in \PP} v\ZZ[v] (\lambda;|\bu|)$,
and
\begin{align*}
\overline{\Psi_{\br^{w_0},\bu} (G(\bm;\br^{w_0}))}
&= \overline{f_{\tilde{j}} (G(\Phi_{\bu}(\bnu);|\bu|))} - \overline{\sum_{\substack{\bl \in \PP^{\ell} \\ \bl \lhd \bm}} a_{\bl}(v)\,  G(\Phi_{\bu}(\bl);|\bu|))} \\
&= f_{\tilde{j}} (G(\Phi_{\bu}(\bnu);|\bu|)) - \sum_{\substack{\bl \in \PP^{\ell} \\ \bl \lhd \bm}} a_{\bl}(v)\,  G(\Phi_{\bu}(\bl);|\bu|)) \\
&= \Psi_{\br^{w_0},\bu} (G(\bm;\br^{w_0})),
\end{align*}
we see that $\Psi_{\br^{w_0}, \bu}(G(\bm;\br^{w_0}))$ satisfies the characterising properties of $G(\Phi_{\bu}(\bm);|\bu|)$ so that
$$\Psi_{\br^{w_0}, \bu}(G(\bm;\br^{w_0})) = G(\Phi_{\bu}(\bm);|\bu|),$$ as desired.
\end{proof}

\begin{thm} \label{T:decomp}
Let $B$ be a core block of $\HH_n$, and suppose that there exists $w \in \ZZ^{\ell}$ such that $\br^w \in \AAbar$ and $\mvi^{\br^w}_{\ell}(B) = 0$.
\begin{enumerate}
\item
For each $\bl \in \PP^{\ell}$ lying in $B$ with $\bij_{\ell,e}((\bl;\br)^{w}) = (\bl^*; \br^*_B)$,
there exists a unique $\Omega(\bl) \in \PP^{\ell}$ such that
$$\bij_{\ell,e}((\Omega(\bl);\br)^{w}) = ((\bl^*)^{\sigma_B\tau_B}; \bv_e(|\by^B|)\ell + \Sc(B))^{\rho_e^{-j_B}},$$
where $\rho_e = (1,2,\dotsc, e) \in \sym{e}$.
(Recall the notation of $\bv_e(x)$ in Subsection \ref{SS:notations}.)

\item
Let $\bu := \br^{w} - |\by^B|\bone$, and for each $\bnu \in \PP^{\ell}$, define $\Phi_{\bu}(\bnu) := \beta^{-1}(\UU(\beta_{\bu}(\bnu)))$.
\begin{enumerate}
\item
If $\bm \in \PP^{\ell}$ lying in $B$ such that $(\bm;\br)$ is Kleshchev, then $\Phi_{\bu}(\Omega(\bm))$ is $e$-regular.

\item
There exists a $\CC(v)$-linear map $\Theta : \CF_{\br} \to \CF_{|\bu|}$ such that
$\Theta(\bl;\br) = (\Phi_{\bu}(\Omega(\bl)), |\bu|)$ for all $\bl \in \PP^{\ell}$ lying in $B$, and
$$
\Theta(G(\bm;\br)) = G(\Phi_{\bu}(\Omega(\bm)), |\bu|)
$$
for all $\bm \in \PP^{\ell}$ lying in $B$ such that $(\bm;\br)$ is Kleshchev.

In particular, for all $\bm \in \PP^{\ell}$ lying in $B$ such that $(\bm;\br)$ is Kleshchev, $$d^{\br}_{\bl\bm}(v)  =
d_{\Phi_{\bu}(\Omega(\bl)), \Phi_{\bu}(\Omega(\bm))}(v)$$ for all $\bl \in \PP^{\ell}$ lying in $B$, and $d_{\lambda, \Phi_{\bu}(\Omega(\bm))}(v) = 0$ if $\lambda \ne  \Phi_{\bu}(\Omega(\bnu))$ for any $\bnu \in \PP^{\ell}$ lying in $B$.
\end{enumerate}
\end{enumerate}
\end{thm}

\begin{proof}
Since $\br^w \in \AAbar$ and $\mvi^{\br^w}_{\ell} (B) = 0$, we may choose $w_0$ to be $w$.

By Proposition \ref{P:Scopes} and Lemma \ref{L:unique-j-initial}, the unique initial core block $B_0$ that is Scopes equivalent to $B$ has $\br^*_{B_0} = (\bv_e(|\by^B|)\ell + \Sc(B))^{\rho_e^{-j_B}}$.
Furthermore, each $\bl \in \PP^{\ell}$ lying in $B$ with $\bij_{\ell,e}((\bl;\br)^{w_0}) = (\bl^*,\br^*_B)$ corresponds bijectively to a $\bm \in \PP^{\ell}$ lying in $B_0$ satisfying
\begin{align*}
\bij_{\ell,e}((\bm;\br)^{w_0}) = ((\bl^*)^{\sigma_B\tau_B\rho_e^{-j_B}}; \br^*_{B_0})
= ((\bl^*)^{\sigma_B\tau_B}; \bv_e(|\by^B|)\ell + \Sc(B))^{\rho_e^{-j_B}}.
\end{align*}
Thus $\Omega(\bl)$ exists, and is unique since $\bij_{\ell,e}$ is bijective between $\PP^{\ell}\times \ZZ^{\ell}$ and $\PP^e \times \ZZ^e$.  This shows part (1).

Since $\bl \mapsto \Omega(\bl)$ is the map induced by the Scopes equivalence between $B$ and $B_0$,
we have in particular that
for
$\bm \in \PP^{\ell}$ lying in $B$, $(\bm;\br)$ is Kleshchev if and only if $(\Omega(\bm);\br)$ is. 
Furthermore,
$d^{\br}_{\bl\bm}(v) = d^{\br}_{\Omega(\bl) \Omega(\bm)}(v)$ for all $\bl,\bm$ lying $B$ with $(\bm;\br)$ Kleshchev (see \cite[Theorem 20]{LM} for a proof when $\ell =1$; the proof for $\ell > 1$ is similar).
In other words, if $\Xi$ is a $\CC(v)$-linear endomorphism of $\CF_{\br}$ sending $(\bnu;\br)$ to $(\Omega(\bnu);\br)$ for all $\bnu \in \PP^{\ell}$ lying $B$, then we get $\Xi(G(\bm;\br)) = G(\Omega(\bm);\br)$ for all $\bm \in \PP^{\ell}$ lying in $B$ such that $(\bm;\br)$ is Kleshchev.

As $|\by^{B_0}| = |\by^B|$ by Lemma \ref{L:yz}, we have $\bu = \br^{w_0} - |\by^B|\bone = \br^{w_0} - |\by^{B_0}|\bone$.
Let $\Theta : \CF_{\br} \to \CF_{|\bu|}$ be the map obtained  by composing the above map $\Xi$ with $\Psi_{\br,\bu}$ from Proposition \ref{P:decomp}.
For $\bm \in \PP^{\ell}$ lying in $B$ such that $(\bm;\br)$ is Kleshchev, $(\Omega(\bm);\br)$ is Kleshchev, and applying Proposition \ref{P:decomp}, we conclude that
$\Phi_{\bu}(\Omega(\bm))$ is $e$-regular, and that
$$
\Theta (G(\bm;\br)) = \Psi_{\br,\bu}( \Xi(G(\bm;\br))) = \Psi_{\br,\bu}(G(\Omega(\bm);\br)) = G(\Phi_{\bu}(\Omega(\bm));|\bu|),
$$
proving part (2).
\end{proof}

\begin{rem}
Applying Ariki's (resp.\ Brundan-Kleshchev's) theorem to Theorem \ref{T:decomp} yields the corresponding decomposition numbers  (resp.\ graded decomposition numbers) of all core blocks of Ariki-Koike algebras satisfying the conditions in Theorem \ref{T:decomp} in terms of those of the Iwahori-Hecke algebras of type $A$ when $\mathrm{char}(\FF) = 0$.
\end{rem}

\begin{footnotesize}
	\bibliographystyle{alpha}
	\bibliography{references.bib}

\begin{thebibliography}{MNSS24}

\bibitem[AK94]{AK}
Susumu Ariki and Kazuhiko Koike.
\newblock A {H}ecke algebra of {$({\bf Z}/r{\bf Z})\wr{\frak S}_n$} and
  construction of its irreducible representations.
\newblock {\em Adv. Math.}, 106(2):216--243, 1994.

\bibitem[Ari96]{Ariki}
Susumu Ariki.
\newblock On the decomposition numbers of the {H}ecke algebra of {$G(m,1,n)$}.
\newblock {\em J. Math. Kyoto Univ.}, 36(4):789--808, 1996.

\bibitem[BK09]{BK-GradedDecompNos}
Jonathan Brundan and Alexander Kleshchev.
\newblock Graded decomposition numbers for cyclotomic {H}ecke algebras.
\newblock {\em Adv. Math.}, 222(6):1883--1942, 2009.

\bibitem[CR08]{CR}
Joseph Chuang and Rapha\"{e}l Rouquier.
\newblock Derived equivalences for symmetric groups and {$\mathfrak
  {sl}_2$}-categorification.
\newblock {\em Ann. of Math. (2)}, 167(1):245--298, 2008.

\bibitem[CT03]{CT}
Joseph Chuang and Kai~Meng Tan.
\newblock Filtrations in {R}ouquier blocks of symmetric groups and {S}chur
  algebras.
\newblock {\em Proc. London Math. Soc. (3)}, 86(3):685--706, 2003.

\bibitem[Del24]{D}
Alice Dell'Arciprete.
\newblock Equivalence of {$v$}-decomposition matrices for blocks of
  {A}riki-{K}oike algebras.
\newblock {\em J. Pure Appl. Algebra}, 228(7):Paper No. 107639, 24, 2024.

\bibitem[DJM98]{DJM}
Richard Dipper, Gordon James, and Andrew Mathas.
\newblock Cyclotomic {$q$}-{S}chur algebras.
\newblock {\em Math. Z.}, 229(3):385--416, 1998.

\bibitem[DR01]{Du-Rui}
Jie Du and Hebing Rui.
\newblock Specht modules for {A}riki-{K}oike algebras.
\newblock {\em Comm. Algebra}, 29(10):4701--4719, 2001.

\bibitem[Fay07]{Fayers-Coreblock}
Matthew Fayers.
\newblock Core blocks of {A}riki–{K}oike algebras.
\newblock {\em J. Algebr. Comb.}, 26:47--81, 2007.

\bibitem[Jac18]{Jacon-Kleshchev-multipartitions}
Nicolas Jacon.
\newblock Kleshchev multipartitions and extended {Y}oung diagrams.
\newblock {\em Adv. Math.}, 339:367--403, 2018.

\bibitem[Jam78]{James-SLN}
G.~D. James.
\newblock {\em The representation theory of the symmetric groups}, volume 682
  of {\em Lecture Notes in Mathematics}.
\newblock Springer, Berlin, 1978.

\bibitem[JL21]{JL-cores}
Nicolas Jacon and C\'{e}dric Lecouvey.
\newblock Cores of {A}riki-{K}oike algebras.
\newblock {\em Doc. Math.}, 26:103--124, 2021.

\bibitem[Lec02]{Leclerc}
Bernard Leclerc.
\newblock Symmetric functions and the {F}ock space.
\newblock In {\em Symmetric functions 2001: surveys of developments and
  perspectives}, volume~74 of {\em NATO Sci. Ser. II Math. Phys. Chem.}, pages
  153--177. Kluwer Acad. Publ., Dordrecht, 2002.

\bibitem[LM02]{LM}
Bernard Leclerc and Hyohe Miyachi.
\newblock Some closed formulas for canonical bases of {F}ock spaces.
\newblock {\em Represent. Theory}, 6:290--312, 2002.

\bibitem[LQ23]{LQ-movingvector}
Yanbo Li and Xiangyu Qi.
\newblock Moving vectors {I}: {R}epresentation type of blocks of
  {A}riki-{K}oike algebras.
\newblock {\em arXiv:2302.06107}, 2023.

\bibitem[LR16]{LyleRuff-Decompositionnumber}
Sin\'{e}ad Lyle and Oliver Ruff.
\newblock Graded decomposition numbers of {A}riki–{K}oike algebras for blocks
  of small weight.
\newblock {\em J. Pure. Appl. Alg.}, 220:2112--2142, 2016.

\bibitem[LT24]{LT}
Yanbo Li and Kai~Meng Tan.
\newblock Cores and weights of multipartitions and blocks of {A}riki-{K}oike
  algebras.
\newblock {\em arXiv:2408.10626}, 2024.

\bibitem[Lyl23a]{Lyle}
Sin\'{e}ad Lyle.
\newblock Core blocks for {H}ecke algebras of type {B} and sign sequences.
\newblock {\em arXiv:2307.16319}, 2023.

\bibitem[Lyl23b]{Lyle-RoCK}
Sin\'{e}ad Lyle.
\newblock {R}ouquier blocks for {A}riki-{K}oike algebras.
\newblock {\em arXiv:2206.14720}, 2023.

\bibitem[MNSS24]{MNSS}
Robert Muth, Thomas Nicewicz, Liron Speyer, and Louise Sutton.
\newblock A skew {S}pecht perspective of {R}o{CK} blocks and cuspidal systems
  for {KLR} algebras in affine type {A}.
\newblock {\em arXiv:2405.15759}, 2024.

\bibitem[Ugl00]{Uglov}
Denis Uglov.
\newblock Canonical bases of higher-level {$q$}-deformed {F}ock spaces and
  {K}azhdan-{L}usztig polynomials.
\newblock In {\em Physical combinatorics ({K}yoto, 1999)}, volume 191 of {\em
  Progr. Math.}, pages 249--299. Birkh\"auser Boston, Boston, MA, 2000.

\bibitem[Web24]{Webster}
Ben Webster.
\newblock Rock blocks for affine categorical representations.
\newblock {\em J. Comb. Algebra}, pages published online first, {DOI
  10.4171/JCA/88}, 2024.

\end{thebibliography}
\end{footnotesize}

\end{document}